\documentclass[12pt]{amsart}
\usepackage{amsmath,amssymb,euscript,mathrsfs, eufrak,tikz}
\usepackage[all]{xy}
\usepackage{enumerate}

\pushQED{\qed}

\newcommand{\define}{\textbf}

\renewcommand{\setminus}{\smallsetminus}
\renewcommand{\phi}{\varphi}

\renewcommand{\tilde}{\widetilde}
\renewcommand{\hat}{\widehat}
\renewcommand{\bar}{\overline}
\renewcommand{\L}{\mathbb{L}}

\newcommand{\D}{\mathbb{D}}

\newcommand{\C}{\mathbb{C}}
\newcommand{\Q}{\mathbb{Q}}
\newcommand{\R}{\mathbb{R}}

\newcommand{\Z}{\mathbb{Z}}

\newcommand{\A}{\mathbb{A}}

\newcommand{\K}{\mathbb{K}}

\newcommand{\cS}{\mathcal{S}}

\newcommand{\cH}{\mathcal{H}}

\def\<{\ensuremath{\langle}}
\def\>{\ensuremath{\rangle}}

\newcommand{\lc}{l^*}
\newcommand{\cHp}{{}^p\mathcal{H}}

\newcommand{\cC}{\mathcal{C}}

\DeclareMathOperator{\Sym}{Sym}

\DeclareMathOperator{\rk}{rk}

\DeclareMathOperator{\lk}{lk}

\DeclareMathOperator{\Int}{Int}
\DeclareMathOperator{\IInt}{int}

\DeclareMathOperator{\Trop}{Trop}

\DeclareMathOperator{\Var}{Var}

\DeclareMathOperator{\vol}{vol}

\DeclareMathOperator{\Ehr}{Ehr}
\DeclareMathOperator{\rec}{rec}

\DeclareMathOperator{\ex}{ex}
\DeclareMathOperator{\DD}{D}

\DeclareMathOperator{\Bary}{Bary}

\newtheorem{theorem}{Theorem}[section]
\newtheorem{lemma}[theorem]{Lemma}
\newtheorem{proposition}[theorem]{Proposition}
\newtheorem{corollary}[theorem]{Corollary}

\newtheorem*{ntheorem}{Theorem}

\theoremstyle{definition}
\newtheorem{definition}[theorem]{Definition}
\newtheorem{remark}[theorem]{Remark}
\newtheorem{example}[theorem]{Example}

\newcommand{\excise}[1]{}

\begin{document}

\title[]{Local $h$-polynomials, invariants of subdivisions, and mixed Ehrhart theory}
\author{Eric Katz}
\address{Department of Combinatorics \& Optimization, University of Waterloo, 200 University Avenue West, Waterloo, ON, Canada N2L 3G1} \email{eekatz@math.uwaterloo.ca}

\author{Alan Stapledon}
\address{Department of Mathematics\\University of Sydney\\ NSW, Australia 2006}
\email{alan.stapledon@sydney.edu.au}

\keywords{polytopes, Ehrhart theory, Eulerian posets}
\date{}
\thanks{}

\begin{abstract} There are natural polynomial invariants of polytopes and lattice polytopes coming from enumerative combinatorics and Ehrhart theory, namely the $h$- and $h^*$-polynomials, respectively.  In this paper, we study their generalization to subdivisions and lattice subdivisions of polytopes.  By abstracting constructions in mixed Hodge theory, we introduce multivariable polynomials 
which specialize to the $h$-, $h^*$- polynomials. These polynomials, the mixed $h$-polynomial and the (refined) limit mixed $h^*$-polynomial have rich symmetry, non-negativity, and unimodality properties, 
which both refine known properties of the classical polynomials, and reveal new structure. For example, we prove a lower bound theorem for a related invariant called the local $h^*$-polynomial. 
 We introduce our polynomials by developing a very general formalism for studying subdivisions of Eulerian posets that  extends the work of Stanley, Brenti and Athanasiadis on local $h$-vectors.
 In particular,  we prove a conjecture of Nill and Schepers, and answer a question of Athanasiadis. 
\end{abstract}

\maketitle

\section{Introduction}\label{s:intro}

Given a polytope $P$, one can associate an invariant called the toric $h$-vector, which, if $P$ is simplicial, encodes the number of faces of a given dimension (see Example~\ref{e:hEulerian}). Symmetry, non-negativity and unimodality results
constrain the possible toric $h$-vectors (see \cite{BraRemarks} for a survey). 
If $\cS$ is a polyhedral subdivision of $P$, then we are interested in understanding the combinatorics of the cells of $\cS$. 
Analogously to above,  one can associate an invariant  $h(\cS; t)$ called the \define{$h$-polynomial}, which, if $\cS$ is a triangulation, encodes the number of cells of $\cS$ of a given dimension. 
When $\cS$ is the trivial subdivision, $h(\cS; t)$ is Stanley's \emph{$g$-polynomial} of $P$. The \define{local $h$-polynomial} $l_P(\cS;t)$ has symmetric coefficients and was 
introduced by Stanley in \cite{StaSubdivisions}. 
When $\cS$ is a rational polyhedral subdivision, he proved that $h(\cS; t)$ and $l_P(\cS;t)$ have non-negative coefficients, 
and, in addition, when $\cS$ is regular, the coefficients of $l_P(\cS;t)$ are symmetric and unimodal \cite[Theorem~7.9]{StaSubdivisions}. 
 The case when $\cS$ is a triangulation
of a simplex has been of particular interest and  we refer the reader to \cite{AthSurvey} for a survey. 

If $P$ is a lattice polytope,  one is interested in understanding the Ehrhart theory of $P$ i.e. the enumerative combinatorics of the lattice points in all dilates of $P$. 
In this case, one can encode the count of the number of lattice points in dilates of $P$ in a certain invariant $h^*(P;t)$ called the $h^*$-polynomial  (also known as the \emph{$\delta$-polynomial} or {Ehrhart $h$-polynomial}).  The \define{local $h^*$-polynomial} $\lc(P;t)$ has symmetric coefficients and was
introduced by Stanley in \cite[Example~7.13]{StaSubdivisions}, and later independently by Borisov and Mavlyutov in \cite{BMString} with the notation $\tilde{S}(t)$. 
Stanley proved that the coefficients of $h^*(P;t)$ are non-negative in \cite{StaDecomp}, and conjectured that the coefficients of $\lc(P;t)$ were non-negative in  \cite[Example~7.13]{StaSubdivisions}. 
The latter conjecture was proved by Karu \cite[Theorem~1.1]{KarEhrhart}, and 
independently by Borisov and Mavlyutov \cite[Proposition~5.1]{BMString}.
If $\cS$ is a polyhedral subdivision of $P$ into lattice polytopes, then the Ehrhart theory of $P$ reduces to understanding the combinatorics of $\cS$, together with the Ehrhart theory of the cells of $\cS$. 
In the special case when the cells of $\cS$ are unimodular simplices (although 
such a $\cS$ may not exist for a given $P$), understanding the Ehrhart theory of $P$ is equivalent to understanding the combinatorics of $\cS$, and 
$h^*(P;t) = h(\cS;t)$ and $\lc(P;t) = l_P(\cS;t)$. In contrast to the case of the local $h$-polynomial $l_P(\cS;t)$, we note that in general the coefficients of $\lc(P;t)$ are not unimodal (see, for example, Example~\ref{e:nonunimdal}).

Motivated by geometry, specifically mixed Hodge theory, we will introduce new, powerful combinatorial invariants which are multivariable generalizations of the classical invariants. In the case of polytopes, we introduce the \define{mixed $h$-polynomial} $h_P(\cS;u,v) = 1 + \sum_{i,j} h_{i,j} u^i v^j \in \Z[u,v]$ (Definition~\ref{d:mixedpoly}). 
This polynomial is 
invariant under the interchange of $u$ and $v$, and refines the 
 $h$-polynomial 
in the sense that $h_P(\cS;u,1) =  h(\cS;u)$. The local $h$-polynomial $l_P(\cS;t)$
is recovered by considering the terms of highest combined degree in $u$ and $v$ in $h_P(\cS;u,v)$. When $\cS$ is a triangulation of a simplex, $h_P(\cS;u,v)$ precisely 
encodes the numbers $f_{i,j} := \# \{ F \in \cS \mid \dim F + 1 = i, \dim \sigma(F) - \dim F = j \}$, where, for any cell $F$ in $\cS$, $\sigma(F)$ denotes the smallest face of $P$ containing $F$. 
In Corollary~\ref{c:refine}, we prove that if $\cS$ is a rational polyhedral subdivision, then the coefficients of  $h_P(\cS;u,v)$ are non-negative, and, if, furthermore, $\cS$ is a regular subdivision, then 
 the sequences $\{ h_{i,k-i} \}_{i = 0,\ldots,k}$ are
symmetric and unimodal for all $k \ge 0$.  This generalizes the above results of Stanley for $h(\cS; t)$ and $l_P(\cS;t)$. 
When $\cS$ is a triangulation of a simplex, it may be deformed to a rational triangulation without changing the combinatorics (as in \cite{StaSubdivisions}), and the above result implies
non-trivial inequalities between the numbers $\{ f_{i,j} \}_{i,j \ge 0}$. 

In the lattice polytope case, we assume that $P$ and the cells of $\cS$ are lattice polytopes and introduce a number of new invariants. Firstly, we introduce the \define{mixed $h^*$-polynomial} $h^*(P;u,v) \in \Z[u,v]$ (Section~\ref{s:ehrhart}), which
only depends on $P$ and not $\cS$.  This polynomial is invariant under the interchange of $u$ and $v$, and refines the \define{$h^*$-polynomial} $h^*(P;t)$ in the sense that $h^*(P;u,1) =  h^*(P;u)$.   
We recover the local $h$-polynomial $\lc(P;t)$ by considering 
the terms of highest combined degree in $u$ and $v$ in $h^*(P;u,v)$. In Theorem~\ref{t:basic}, we prove that the coefficients of the mixed $h^*$-polynomial are non-negative. This 
generalizes the above results of Stanley, Karu, Borisov and Mavlyutov. When the cells of $\cS$ are unimodular simplices, then $h^*(P;u,v) = h_P(\cS;u,v)$. 
 More generally, we introduce the following diagram of invariants refining the $h^*$-polynomial:
\[\xymatrix{
h^*(P,\cS;u,v,w)   \ar[r]^{\substack{u \mapsto uw^{-1} \\ v \mapsto 1}}   \ar[d]^{w \mapsto 1}  & h^*(P;u,w) \ar[d]^{w \mapsto 1}  &\\
 h^*(P,\cS;u,v) \ar[r]^{v \mapsto 1} &   h^*(P;u). 
}\] 
The invariants $h^*(P,\cS;u,v) \in \Z[u,v]$ and $h^*(P,\cS;u,v,w) \in \Z[u,v,w]$ are called the \define{limit mixed $h^*$-polynomial} and \define{refined limit mixed $h^*$-polynomial}, respectively.
We note that although analogues of these invariants exist in the polytope case (see Remark~\ref{r:generalizations}), we will not study them in this work. 
We may write \[ h^*(P,\cS;u,v,w) = 1 + uvw^2 \sum_{r = 0}^{ \dim P - 1}  \sum_{ 0 \le p,q \le r} h^*_{p,q,r} u^p v^q w^r, \] 
 for some  integers $h^*_{p,q,r}$, that we visualize in diamonds as follows: the \define{$r$-local $h^*$-diamond} of $(P,\cS)$ is obtained by placing $h^*_{p,q,r}$ at point $(q - p, p + q)$  in $\Z^2$ for $0 \le p,q \le r$ 
(see Figure~\ref{f:localdiamondintro} below). The coefficients of the $h^*$-polynomial are recovered by stacking the diamonds on top of each other, summing the entries, and then summing along a fixed choice of 
diagonal. 
Also, the coefficients of the local  $h^*$-polynomial are recovered by  summing the coefficients of the   $(\dim P - 1)$-local $h^*$-diamond  along a fixed choice of 
diagonal.
 In Theorem~\ref{t:combinatorics}, we show that the coefficients of the diamonds are non-negative and symmetric about the horizontal and vertical axes.
In Theorem~\ref{t:lowerentry}, we prove that each horizontal strip of the $r$-local $h^*$-diamond satisfies the following lower bound theorem: its first entry is a lower bound for the other entries. If $\cS$ is a regular subdivision, then 
we prove that the coefficients of each vertical strip of the $r$-local $h^*$-diamond are symmetric and unimodal (Theorem~\ref{t:unimodality}).  
In the special case of regular, unimodular triangulations, the coefficients of each diamond vanish away from the central vertical strip, and we deduce that the coefficients of the local $h$-polynomial are unimodal
(Example~\ref{e:unimodularrefined}). 

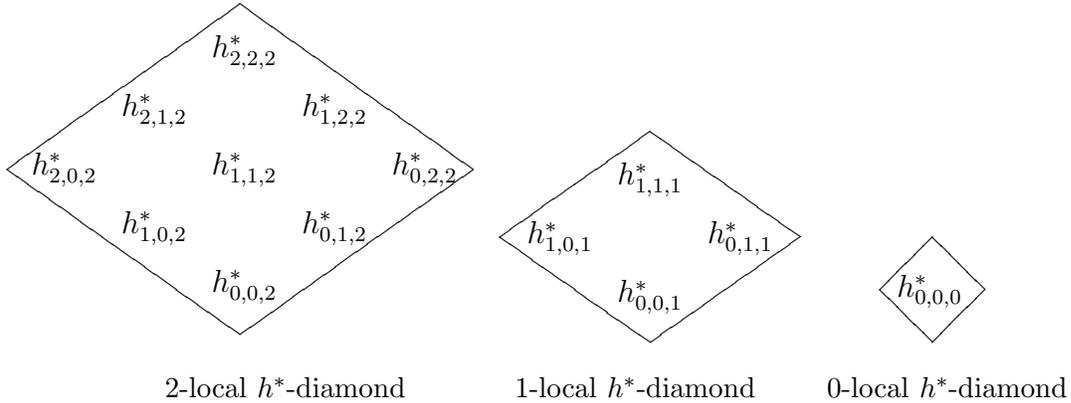
\begin{figure}[htb]
\begin{center}
\[
\begin{array}{p{5cm}p{4cm}p{5cm}}
\xy
(4,3)*{}="A"; (-27,25)*{}="B"; (4,47)*{}="C"; (35,25)*{}="D";
"A"; "B" **\dir{-};
"A"; "D" **\dir{-};
"C"; "B" **\dir{-};
"C"; "D" **\dir{-};
(-9.7,20.5)*{
\xymatrix@=0.2em{ 
   &  & h^*_{2,2,2} &                                               \\
    & h^*_{2,1,2} & & h^*_{1,2,2} &                                  \\
        h^*_{2,0,2}  &  & h^*_{1,1,2} & & h^*_{0,2,2}                 \\
          & h^*_{1,0,2} & & h^*_{0,1,2} &                           \\
         &  & h^*_{0,0,2} & &                                \\
 }
};
(10,-4)*{\hbox{\small $2$-local $h^*$-diamond}};
\endxy
&
\xy
(2,2)*{}="A"; (-18,16)*{}="B"; (2,30)*{}="C"; (22,16)*{}="D";
"A"; "B" **\dir{-};
"A"; "D" **\dir{-};
"C"; "B" **\dir{-};
"C"; "D" **\dir{-};
(-5,12)*{
\xymatrix@=0.2em{ 
          & h^*_{1,1,1} & {}                 \\
           h^*_{1,0,1} & & h^*_{0,1,1}                          \\
           & h^*_{0,0,1} &               \\
 }
};
(0,-4)*{\hbox{\small $1$-local $h^*$-diamond}};

\endxy
&
\xy
(-2,2)*{}="A"; (-9,9)*{}="B"; (-2,16)*{}="C"; (5,9)*{}="D";
"A"; "B" **\dir{-};
"A"; "D" **\dir{-};
"C"; "B" **\dir{-};
"C"; "D" **\dir{-};
(-5,9)*{
\xymatrix@=0.2em{ 
          & {} & {}                 \\
           {} & &{}                          \\
           & h^*_{0,0,0} &               \\
 }
};
(0,-4)*{\hbox{\small $0$-local $h^*$-diamond}};

\endxy

\end{array} 
\]
\end{center}

\caption{ $r$-local $h^*$-diamonds of $(P,\cS)$ when $\dim P = 3$}
\label{f:localdiamondintro}
\end{figure}

We have seen that the coefficients of the diamonds refine the coefficients of the $h^*$-polynomial and  local $h^*$-polynomial. 
 Hence the structure of these diamonds provides insights into the structure of these classical invariants. For example, 
we deduce the following lower bound theorem (Theorem~\ref{t:lower}, Remark~\ref{r:deducelower}): 

\begin{ntheorem}
Let $P \subseteq \R^d$ be a  lattice polytope.  Then the local $h^*$-polynomial $l^*(P;t) = l_1^*t + \cdots + l_{\dim P}^*t^{\dim P}$ satisfies $l_1^* = \#(\Int(P) \cap \Z^d) \le l_i^*$ for $1 \le i \le \dim P$. 
\end{ntheorem}

We employ two distinct ways to view the above combinatorics from a geometric perspective. One the one hand, assume that $P \subseteq \R^d$ is a rational polytope, and consider the 
cone generated by $P \times \{ 1 \} \subseteq \R^d \times \R$. Then $\cS$ induces a fan refinement of the cone, and we have a corresponding morphism of toric varieties, to which one
can apply the decomposition theorem for intersection cohomology. The relative hard Lefschetz theorem \cite[Theorem~1.6.3]{CMBulletin} naturally gives rise to unimodal sequences. 
This approach is explained in detail in Section~\ref{s:geometric}, and we use it to prove the results above for the mixed $h$-polynomial $h_P(\cS;u,v)$.  We note that this extends ideas of Stanley \cite{StaSubdivisions} and also overlaps
with recent work of de Cataldo, Migliorini, and Mustata \cite{CMM}.

On the other hand, assuming that $\cS$ is a regular lattice subdivision, the normal fan of $P$ corresponds to a different toric variety, and we may consider a family of hypersurfaces over the punctured disc with Newton polytope $P$ whose coefficients have asymptotics related to a height function inducing $\cS$.  The cohomology of the generic member of the family of hypersurfaces inherits a weight filtration induced by the  monodromy around the puncture. 
This gives rise to invariants called \emph{refined limit mixed Hodge numbers}, which provide a geometric interpretation of the coefficients $h^*_{p,q,r}$. The action of the monodromy operator 
naturally produces unimodal sequences. This perspective is explored in detail in \cite{geometricpaper} and summarized in Remark~\ref{r:geometric}.

Finally, in the paper itself we develop a new combinatorial theory for studying locally Eulerian posets that generalizes previous work of Stanley \cite{StaSubdivisions}, Brenti \cite{BrentiPKernels} and Athanasiadis \cite{AthFlag}. 
Here a \emph{locally Eulerian} poset abstracts certain classes of 
polyhedral complexes. 
We introduce the notion of a \define{strong formal subdivision} $\sigma: \Gamma \rightarrow B$ of locally Eulerian posets, which abstracts the notion of polyhedral subdivisions, and study an associated \define{pushforward map}. 
Following Stanley \cite{StaSubdivisions}, we consider  the notion of polynomials that obey a particular symmetry property  with respect to a kernel, and are called \define{acceptable}, and we determine the behavior of acceptable functions
under pushforward. We also study the corresponding theory in 
the context of Ehrhart theory.  
We believe this work will be of considerable interest to combinatorialists. In particular, in Remark~\ref{r:conject}, we prove a conjecture of Nill and Schepers \cite{NSCombinatorial}, and answer a question of Athanasiadis  \cite[Question~2.16]{AthSurvey}. In this introduction, we have presented a special case, firstly, for simplicity of 
exposition, and, secondly, to allow readers the opportunity to understand some of the main results without using this machinery.  In the case described in the introduction, $\Gamma$ is the poset of cells of $\cS$, $B$ is the poset of faces of $P$ and
$\sigma$ takes a cell $F$ of $\cS$ to the smallest face $Q$ of $P$ such that $F \subseteq Q$. Moreover,   in this case, with notation to be introduced  later  in the paper,   $h(\cS;t) := h(\Gamma;t)$,  $l_P(\cS;t) := l_B(\Gamma;t)$ and  $h_P(\cS;u,v) := h_B(\Gamma;u,v)$.

 \subsection{Organization of the paper} 
 
 In Section 2, we review the formalism for Eulerian posets.  In Section 3, we study strong formal subdivisions of locally Eulerian posets.  In Section 4, we introduce the local $h$-polynomial of a strong formal subdivision which is then generalized to the mixed $h$-polynomial in Section 5.  In Section 6, we prove that the local $h$-polynomial is non-negative and is also symmetric and unimodal in the case of regular polyhedral subdivisions by relating it to intersection cohomology.  Section 7 begins the treatment of Ehrhart theory while Section 8 introduces and establishes the properties of the limit mixed $h^*$-polynomial.    Section 9 deals with the refined limit mixed $h^*$-polynoimal, proving a symmetry result that is important for \cite{geometricpaper}.  Section 10 is an aside on tropical geometry in which hypersurfaces are replaced by more general sch\"{o}n subvarieties.

\medskip
\noindent
{\it Notation and conventions.}  
Throughout, we will consider the empty polytope to have dimension $-1$.  If $a(t)$ and $b(t)$ are real-valued polynomials, then we write $a(t) \ge b(t)$ if $a(t) - b(t)$ has non-negative coefficients. 
 All posets considered in this paper will be finite. 
We denote by $\hat{1}$ (respectively $\hat{0}$) the unique maximal element (respectively minimal element) of a finite poset if it exists. 

\medskip
\noindent
{\it Acknowledgements.}  We would like to thank Mark de Cataldo, Mircea Mustata, and Paul Bressler for valuable conversations.  Special thanks must go to Benjamin Nill who introduced the authors to Stanley's subdivision theory and  advocated its importance in Ehrhart theory.  

\section{Combinatorics of posets}

In this section, we will study posets with a view to understanding Eulerian posets.  The motivation is geometric.  The poset will be thought of as describing the open strata of a nice stratification of a complex algebraic variety. Associated to the closed strata will be symmetric functions such as the Poincar\'{e} polynomial of intersection cohomology which is symmetric by Poincar\'{e} duality.  The functions associated to open strata will be obtained by subtracting off the contributions of smaller closed strata where each contribution is weighted by a function $\gamma$ which encodes the singularities along the smaller strata.  The functions associated to open strata are not symmetric but they obey a reciprocity theorem.  They are called ``acceptable".  
The approach here is a generalization of that of Stanley \cite[Sec.~6]{StaSubdivisions},
influenced by the work of Brenti \cite{BrentiTwisted, BrentiPKernels}.

Let $B$ be a finite poset. For any pair $x \leq x'$ in $B$, we may consider the interval $[x,x'] = \{ y \in B \mid x \leq y \leq x' \}$. The set of all intervals is identified with $\Int(B) := \{ (x,x') \in B^2 \mid x \le x' \}$. 
Let $R$ be a commutative ring with identity of characteristic not $2$. 
Following \cite[Section~3.6]{Stanley1}, 
 the \define{incidence algebra} $I(B;R)$ of $B$ with coefficients in $R$ is the associative $R$-algebra with elements given by all functions $f: \Int(B) \rightarrow R$ and addition and  multiplication 
defined by
\[
(f + g)(x,x') = f(x,x') + g(x,x'), \; \; \; (f  * g)(x,x') = \sum_{x \le y \le x'} f(x,y) g(y,x'), 
\]
for all $f,g$ in $I(B;R)$ and $x \le x'$  in $B$. 
The ring $R$ is embedded as a subalgebra via $r \mapsto f_r$, where $f_r(x,x') = r$ if $x = x'$ and $0$ otherwise. 
We may think of this algebra as follows: fix an ordering $\{ x_1, \ldots, x_r \}$ of $B$ such that $x_i \le x_j$ implies that $i \le j$. Then $I(B;R)$ is isomorphic to the subalgebra $A = \{ (a_{i,j})_{1 \le i,j \le r} \mid a_{i,j} = 0 \textrm{ unless } x_i \le x_j \}$ of the algebra of $r \times r$ upper triangular matrices.   In particular, it follows immediately that $I(B;R)$ is associative 
and an element $f$ in $I(B;R)$ is invertible if and only if 
$f(x,x) \in R$ is invertible for all $x \in B$.  For $x \le x'$ in $B$, define $e_{x,x'}$ in $I(B;R)$ by
\[
e_{x,x'}(z,z') = \left\{\begin{array}{cl} 1 & \text{if }  x = z,x' = z',  \\ 0 & \text{otherwise}. \end{array}\right.
\]
Then $I(B;R)$ is a free  $R$-module with basis $\{ e_{x,x'} \mid x \le x' \in B \}$. 

Similarly, following \cite[Section~6]{StaSubdivisions}, we define  $R^B$ to be the 
$R$-module consisting of all functions $f: B \rightarrow R$. Then $R^B$ is a right $I(B;R)$-module via
\[
(f * g)(x') = \sum_{x \le x'} f(x) g(x,x'),
\] 
for all $f$ in $R^B$, $g$ in $I(B;R)$ and $x'$ in $B$.  Observe that when $B$ contains a minimal element $\hat{0}$, then $R^B$ may be identified with the 
two-sided ideal of $I(B;R)$ consisting of all functions $f: I(B) \rightarrow R$ satisfying $f(x,x') = 0$ for $x \neq \hat{0}$.
For $x$ in $B$, define
$e_x$ in $R^B$ by 
\[
e_{x}(z) = \left\{\begin{array}{cl} 1 & \text{if }  x = z,  \\ 0 & \text{otherwise}. \end{array}\right.
\]
Then $R^B$ is a free $R$-module with basis $\{ e_x \mid x \in B \}$. 
 
Fix a ring involution  $r \mapsto \bar{r}$ of $R$, and let $S(R) = \{ r \in R \mid r = \bar{r} \}$ be the invariant subring of $R$. 
The incidence algebra $I(B;R)$ gets an induced ring involution with invariant subring $I(B;S(R))$. 
In particular, we may view $I(B;S(R))$ as the free $S(R)$-submodule of $I(B;R)$ with basis $\{ e_{x,x'} \mid x \le x' \in B \}$. Similarly, $R^B$ gets an induced $S(R)$-module involution with
invariant submodule $S(R)^B$, and we view  $S(R)^B$ as the free $S(R)$-submodule of $R^B$ with basis $\{ e_{x} \mid x \in B \}$.

For any subset $A$ of $R$, 
let 
\[
U(B;A) = \{ f \in I(B;R) \mid f(x,x) = 1 \textrm{ for all } x \in B, f(x,x') \in A \textrm{ for all } x < x' \}.
\]
Clearly, $U(B;R)$ is invariant under the involution and the corresponding invariant subring is $U(B;S(R))$. The following extends the corresponding definition in \cite[Definition~6.2]{StaSubdivisions}
(cf. \cite[Theorem~6.5]{StaSubdivisions}). 

\begin{definition} 
An element $\kappa \in U(B;R)$ 
is a \define{$B$-kernel} if  $\kappa * \bar{\kappa} = 1$. 
\end{definition}

We now extend the notion of an element being  symmetric. When the kernel below is the identity, this corresponds to the usual notion that an element is symmetric if it is invariant under the involution. 
The following is an extension the corresponding definitions in \cite[Definition~6.2]{StaSubdivisions}. 

\begin{definition}
Let $\kappa$ be a $B$-kernel. 
An element $f$ in $I(B;R)$ (respectively in $R^B$) is \define{$\kappa$-totally acceptable} (respectively  \define{$\kappa$-acceptable} )  if $\bar{f} = f * \kappa$. 
The set of all $\kappa$-totally acceptable functions forms a left  $S(R)$-submodule of $I(B;R)$ denoted $\mathcal{T}(B,\kappa)  = \mathcal{T}(B,\kappa;R)$. 
The set of all $\kappa$-acceptable functions forms a left  $S(R)$-submodule of $R^B$ denoted $\mathcal{A}(B,\kappa)  = \mathcal{A}(B,\kappa;R)$. 
\end{definition}

We will produce a basis for $\mathcal{A}(B,\kappa)$ and $\mathcal{T}(B,\kappa) $ by using particular  elements of the incidence algebra.  
\begin{definition} An element $\gamma\in U(B;R)\cap \mathcal{T}(B,\kappa) $ is called an \define{acceptability operator}.
\end{definition}
Note that an acceptability operator $\gamma$ determines $\kappa$ since
 $\gamma^{-1}*\bar{\gamma}=\kappa$ by definition.  We note that, in a special case, Brenti uses the alternative notation ``Kazhdan-Lusztig-Stanley function" or ``KLS-function'"
for the acceptability operator (see \cite[Theorem~2.1]{BrentiPKernels}).

\begin{lemma} \label{l:symacc}
Fix a $B$-kernel $\kappa$ and an acceptability operator $\gamma$. Then right multiplication by $\gamma$ gives an isomorphism of left $S(R)$-modules:
\[
I(B;S(R)) \rightarrow \mathcal{T}(B,\kappa) , 
\]
\[
f \mapsto f * \gamma. 
\]
and
\[
S(R)^B \rightarrow \mathcal{A}(B,\kappa), 
\]
\[
f \mapsto f * \gamma. 
\]
In particular,  $\mathcal{T}(B,\kappa)$ is a free left $S(R)$-module with basis  $\{ e_{x,x'} * \gamma  \mid x \le x' \in B  \}$ and $\mathcal{A}(B,\kappa)$ is a free left $S(R)$-module with basis  $\{ e_{x} * \gamma \mid x \in  B \}$.
\end{lemma}

\begin{proof}
Let $f\in I(B;R)$ or $f\in R^B$.  Then
\[
\bar{f} = f 
\iff \bar{f*\gamma}= f*\gamma * \gamma^{-1} * \bar{\gamma}
\iff  \bar{f*\gamma}= f*\gamma * \kappa, 
\]
and the result follows. 
\end{proof}

Next we consider criterion to guarantee the existence and uniqueness of an acceptability operator. 
Fix a subring $\tilde{R}$ of $R$ invariant under the involution (often we set $ \tilde{R} = R$), and  fix a splitting 
of $\Z$-modules 
\begin{equation}\label{e:split}
\tilde{R} = S(\tilde{R}) \oplus R',
\end{equation}
where $R'$ is a $\Z$-submodule of $\tilde{R}$.    Let $\tilde{\mathcal{K}}$ denote the set of $B$-kernels defined over $\tilde{R}$.  Let $D:\tilde{R}\rightarrow R'$ be the projection onto $R'$.

\begin{lemma} \label{l:gamma} Fix the splitting \eqref{e:split} above with projection $D$, and assume that $2 \in \tilde{R}$ is invertible. Then there is a bijection
\[
U(B;R')\rightarrow\tilde{\mathcal{K}}, 
\]
\[
\gamma\mapsto \kappa=\gamma^{-1}*\bar{\gamma}.
\]
The inverse, denoted $\gamma_\kappa = \gamma_{B,\kappa}$ 
(throughout we abuse notation and ignore the dependence on the choice of splitting) is recursively defined by 
\begin{equation}\label{e:kernelrecur}
\gamma_\kappa(x,x') = D\big(-\frac{1}{2} \sum_{x \le z < x'} \gamma_\kappa(x,z) \kappa(z,x')\big),
\end{equation}
for $x < x'$ in $B$. 
\end{lemma}

\begin{proof}
We verify that the functions are inverses of each other. 
Note that if $\gamma \in U(B;R')$ and  $\kappa=\gamma^{-1}*\bar{\gamma}$, then the statement $\gamma = \gamma_\kappa$ is equivalent to the statement that  for all $x<x'$,
\[ \gamma(x,x') =  D\big(-\frac{1}{2} \sum_{x\leq z<x'} \gamma(x,z)(\gamma^{-1}*\bar{\gamma})(z,x')\big),\]
which holds if and only if  $\gamma(x,x') = -D(\bar{\gamma}(x,x'))$. Since  $\gamma(x,x') \in R'$, the latter is equivalent to the fact that $\gamma(x,x') + \bar{\gamma}(x,x')$ is symmetric. 

Given $\kappa \in \tilde{\mathcal{K}}$, it remains to show that $\bar{\gamma_\kappa} = \gamma_\kappa * \kappa$. By definition, this is equivalent to  the statement that for all $x<x'$,
\begin{equation}\label{e:def}
\bar{\gamma_\kappa}(x,x')-\gamma_\kappa(x,x') =\sum_{x\leq z<x'} \gamma_\kappa(x,z)\kappa(z,x').
\end{equation}
We will use the following easily-proved observation: if $g\in \tilde{R}$ is antisymmetric,  then the unique solution for $a\in R'$ to 
$\overline{a}-a=g$ is $a=-\frac{1}{2}D(g)$. Setting $g$ to be the right hand side of \eqref{e:def}, it remains to show that $g$ is anti-symmetric. This follows from the generalized Dehn-Sommerville equations \cite[Prop 6.4]{StaSubdivisions}. Explicitly, using induction on the length of a maximal chain in an interval, we compute:
\begin{align*}
\bar{g} &= \sum_{x\leq z<x'} \bar{\gamma_\kappa}(x,z)   \bar{\kappa}(z,x') \\
&= \sum_{x\leq z<x'} \big[  \sum_{x \le x'' \le z} \gamma_\kappa(x,x'') \kappa(x'',z) \big] \bar{\kappa}(z,x') \\
&=  \sum_{x\leq x'' <x'}  \gamma_\kappa(x,x'') \sum_{x'' \leq z<x'} \kappa(x'',z) \bar{\kappa}(z,x') \\
&= -g.
\end{align*}

\end{proof}

\begin{remark}
An alternative proof of the fact that the map in Lemma~\ref{l:gamma} is a bijection is given in \cite[Prop 6.11]{StaSubdivisions} for a specific splitting, but the proof holds in  the more general case.  
\end{remark}

In fact, with the setup above, one easily obtains a  non-recursive formula for the acceptability operator $\gamma_\kappa$. 

\begin{corollary}
Fix the splitting \eqref{e:split} above with projection $D$, and assume that $2 \in \tilde{R}$ is invertible. 
Then we have the following explicit formula for $\gamma_\kappa$:
\[
\gamma_\kappa(x,x') = \sum_{ x = x_0 < x_1 < \cdots  < x_r = x'} \big(-\frac{1}{2}\big)^r    D(\ldots (D(D(\kappa(x_0,x_1))\kappa(x_1,x_2))\ldots)\kappa(x_{r-1},x_r))
\]
for $x < x'$ in $B$. 
\end{corollary}
\begin{proof}
This follows since one easily verifies that the above formula satisfies the recurrence \eqref{e:kernelrecur}. 
\end{proof}

\begin{remark}
This corollary is closely related to the work of Brenti \cite[Cor 6.5]{BrentiTwisted}.
\end{remark}

\begin{example}
If $\kappa$ is the identity, then $\gamma_\kappa$ is the identity, $\mathcal{T}(B,\kappa) = I(B;S(R))$ and $\mathcal{A}(B,\kappa) = S(R)^B$. 
\end{example}

\begin{example}\label{e:standardsplit}
Let $R = \Z[t^{\pm 1/2}]$ with involution $\bar{r(t)} = r(t^{-1})$. 
We may choose a splitting by setting $R = \tilde{R}$ and $R' = \{ \sum_{i \in \Z_{< 0}} \alpha_i t^{i/2} \mid \alpha_i \in \Z \}$. 
In this case,  $D( \sum_{i \in \Z} \alpha_i t^{i/2} ) = \sum_{i \in \Z_{< 0}} (\alpha_i - \alpha_{-i})  t^{i/2} $. Note that in this case $2 \in R$ is not invertible. We will consider this example in further detail in Section~\ref{s:locEulerian}. 
\end{example}

\begin{example}\label{e:-1}
If $2 \in R$ is invertible, then let $R = \tilde{R}$  and let $R'$ be the ($-1$)-eigenspace of $R$ with respect to the involution. Then one may verify that  $\gamma_\kappa$ is inductively defined by:
\[
\gamma_\kappa(x,x') = -\frac{1}{2} \sum_{x \le z < x'} \gamma_\kappa(x,z) \kappa(z,x')
\]
for $x < x'$ in $B$. Alternatively, can verify $\gamma_\kappa$ is given directly by the formula:
\begin{equation}\label{e:gammaformula}
\gamma_\kappa(x,x') = \sum_{ x = x_0 < x_1 < \cdots  < x_r = x'} \big(-\frac{1}{2}\big)^r \prod_{i = 0}^{r - 1} \kappa(x_i,x_{i+ 1})
\end{equation}
for $x < x'$ in $B$. 
\end{example}

\begin{example}\label{e:groth}
Let $k$ be a field of characteristic zero, and consider the
\define{Grothendieck ring} $K_0(\Var_k)$ of varieties over $k$. That is, $K_0(\Var_k)$ is the free $\Z$-module with basis $\{ [V] \mid V \textrm{ is a variety over } k \}$ modulo the
relations $[W] = [V] + [W \smallsetminus V]$ whenever $V$ is a closed subvariety of $W$. We set $\L := [\A^1]$, and let $R = K_0(\Var_k)[\L^{\pm 1/2}]$.  
By \cite[Corollary~3.4]{BitUniv}, there is an involution  $\mathcal{D}_k$ on $K_0(\Var_k)[\L^{-1}]$ that sends $\L$ to $\L^{-1}$ and is characterized by
$\mathcal{D}_k([X]) = \L^{-\dim X}[X]$ for any smooth, proper variety $X$ over $k$. This extends to a ring involution of $R$ by 
setting $\mathcal{D}_k(\L^{1/2})  = \L^{-1/2}$. 
If we set $\tilde{R} = \Z[\L^{\pm 1/2}]$, then as in Example~\ref{e:standardsplit}, we obtain a splitting as above by setting  $R' = \{ \sum_{i \in \Z_{< 0}} \alpha_i \L^{i/2} \mid \alpha_i \in \Z \}$. 
\end{example}

\section{Locally Eulerian posets and strong formal subdivisions}\label{s:locEulerian}

In this section, we will study subdivisions of locally Eulerian posets.  There is a good notion of subdivisions, \emph{strong formal subdivisions}, which abstract polyhedral subdivisions in the same way that Eulerian posets abstract polytopes.  Given such a subdivision, 
we will study an induced pushforward on acceptable functions on these posets.
\subsection{Locally Eulerian posets}

A finite poset $B$ is \define{locally graded} if for every interval $[x,x']$ in $B$, all maximal chains in $[x,x']$ have the same length, denoted $\rho(x,x')$. For example, $\rho(x,x) = 0$. 
The \define{rank} $\rk(B)$ of $B$ is the longest length of a maximal chain in $B$. 

\begin{definition}
A \define{ranked poset} is a pair $(B,r)$ such that $B$ is locally graded, $r: B \rightarrow \Z$ and $\rho(x,x') = r(x') - r(x)$ for all $x \le x'$ in $B$. We call $r$ a \define{rank function}.
\end{definition}

If $B$ is locally graded and contains a unique minimal element $\hat{0}$, then $B$ is \define{lower graded}. 
A lower graded poset is naturally a ranked poset with rank function $\rho(x) := \rho(\hat{0},x)$.  
Note that the following definition does not depend on the choice of a rank function. 

\begin{definition}
A lower graded  poset with unique minimal  and  maximal elements $\hat{0}$ and $\hat{1}$ respectively is \define{Eulerian} if every interval of positive length contains as many elements of even rank as odd rank. 
A locally graded poset is  \define{locally Eulerian} if every interval is Eulerian. A locally Eulerian poset is \define{lower Eulerian} if it contains a unique minimal element $\hat{0}$. 
\end{definition}

\begin{example}\label{e:polytope}
The poset of faces of a polytope $P$ (including the empty face) is an Eulerian poset under inclusion, called the \define{face poset} of $P$, with 
$\rho(Q) = \dim Q + 1$ for every face $Q$ of $P$.    
\end{example}

\begin{example}
The Boolean algebra on $r$ elements consists of all subsets of a set of cardinality $r$ and forms an Eulerian poset under inclusion of rank $r$. This is the face poset of an $(r - 1)$-dimensional 
simplex. 
\end{example}

\begin{example}\label{e:flip}
If $B$ is a poset, then $B^*$ is the poset with the same elements as $B$ and all orderings reversed. In particular, $B$ is Eulerian if and only if $B^*$ is Eulerian. 
\end{example}

Throughout this section we will use the setup of Example~\ref{e:standardsplit}. 
That is, 
$R = \Z[t^{\pm 1/2}]$ with involution $\bar{r(t)} = r(t^{-1})$, and 
 splitting $R = S(R) \oplus R'$, with $R' = \{ \sum_{i \in \Z_{< 0}} \alpha_i t^{i/2} \mid \alpha_i \in \Z \}$. We also set $q = t^{1/2} - t^{-1/2}$. 
 Let $B$ be a locally graded poset, and let $\kappa(x,x')=q^{\rho(x,x')}$ for all $x \le x'$ in $B$.

\begin{lemma}\cite[Proposition~7.1]{StaSubdivisions} \label{l:kernel} The function $\kappa$ is a $B$-kernel if and only if $B$ is locally Eulerian.
\end{lemma}
\begin{proof}
By unpacking the definition, and using the fact that $\bar{q} = -q$, we see that for $x<x'$,
\[(\kappa*\bar{\kappa})(x,x')=q^{\rho(x,x')}\sum_{x\leq z\leq x'} (-1)^{\rho(z,x')}.\]
This is 
zero if and only if the interval $[x,x']$ has the same number of elements $z$ for which $\rho(z,x')$ is even as it is odd. 
\end{proof}

Now assume that $B$ is locally Eulerian. By Lemma \ref{l:gamma}, we can produce a unique acceptability operator  $\gamma_B :=\gamma_{B,\kappa} \in I(B;R')$ (as we will see in Lemma~\ref{l:gpoly} below, we do not need to invert $2 \in R$ in this case).  
In order to describe $\gamma_B$,  we first recall the definition of the \define{$g$-polynomial} of an Eulerian poset \cite{Stanley1}.

\begin{definition}\label{d:g}
Let $B$ be an Eulerian poset of rank $n$. If $n = 0$, 
then $g(B;t) = 1$. If $n > 0$, then $g(B;t)\in\Z[t]$ is the unique polynomial in $t$ 
of degree strictly less the $n/2$ satisfying
\[
t^{n}g(B;t^{-1}) = \sum_{x \in B} g([\hat{0},x];t) (t - 1)^{n - \rho(\hat{0}_B,x)}. 
\]
\end{definition}

\begin{example}\label{e:gconstant}
The constant term of $g(B;t)$ is $g(B;0) = 1$. The linear coefficient of $g(B;t)$ is $\#\{ x \in B \mid \rho(\hat{0}_B,x) = 1 \} - n$. 
\end{example}
 
\begin{example}\label{e:boolean}
One can verify that $g(B;t) = 1$ if and only if $B$ is the Boolean algebra on $r$ elements for some $r$ \cite[Remark~4.2]{BNCombinatorial}.
\end{example}

\begin{lemma}\label{l:gpoly}
Let $B$ be a locally Eulerian poset. Then with the notation above,  the acceptability operator $\gamma_B:=\gamma_{B,\kappa}\in I(B;R')$ is given by
\[
\gamma_B(x,x') = t^{-\rho(x,x')/2} g([x,x'];t) 
\]
for $x \le x'$ in $B$. 
\end{lemma}
\begin{proof}
This follows from unpacking Definition~\ref{d:g} and comparing with \eqref{e:def} in the proof of Lemma~\ref{l:gamma}. 
\end{proof}

The acceptability operator $\gamma_B$ also has an easy to describe inverse. We will use the following theorem of Stanley. 
Recall from Example~\ref{e:flip} that if $B$ is Eulerian, then the dual poset $B^*$ is also Eulerian. 

\begin{theorem}\cite[Corollary~8.3]{StaSubdivisions}\label{t:inverse}
Let $B$ be a locally Eulerian poset. Then with the notation above,  the inverse of the acceptability operator $\gamma_B$ is given by
\[
(\gamma_B^{-1})(x,x') = (-1)^{\rho(x,x')} t^{-\rho(x,x')/2} g([x,x']^*;t) 
\]
for $x \le x'$ in $B$.
\end{theorem}

We will also need the notion of the $h$-polynomial of a lower Eulerian poset.  

\begin{definition}\label{d:hpoly}\cite[Example~7.2]{StaSubdivisions}
Let $B$ be a lower Eulerian poset of rank $n$. Then the \define{$h$-polynomial} of $B$ is defined by
\[
t^n h(B;t^{-1}) = \sum_{x \in B}  g([\hat{0},x];t) (t - 1)^{n - \rho(\hat{0}_B,x)}.
\]
\end{definition}

\begin{example}\label{e:hconstant}
Let $B$ be a lower Eulerian poset of rank $n$. By comparison of the recursive formulas in Definition~\ref{d:g} and Definition~\ref{d:hpoly}, we see that, as in Example~\ref{e:gconstant},  
the constant term of $h(B;t)$ is $h(B;0) = 1$, and the linear coefficient of $h(B;t)$ is $\#\{ x \in B \mid \rho(\hat{0}_B,x) = 1 \} - n$. 
\end{example}

\begin{example}\label{e:hEulerian}
If $B$ is an Eulerian poset of rank $n$, then comparison of the recursive formulas in Definition~\ref{d:g} and Definition~\ref{d:hpoly} implies that $h(B;t) = g(B;t)$. 
 If, furthermore, $n > 0$, then $B \setminus \{ \hat{1} \}$ is a 
 lower Eulerian poset of rank $n - 1$, and 
\[
(1 - t) h(B \setminus \{ \hat{1} \};t) =  g(B;t) - t^ng(B;t^{-1}).
\]
In particular, $h(B \setminus \{ \hat{1} \};t)$ is a polynomial of degree $n - 1$ with symmetric coefficients. For example, if $B$ is the Boolean algebra on $r$ elements, then 
 $h(B \setminus \{ \hat{1} \};t) = 1 + t + \cdots + t^{r - 1}$. 
 More generally, if $P$ is a polytope, and $B$ is the dual poset of the face poset of $P$ (see Example~\ref{e:polytope}), 
 then $h(B \setminus \{ \hat{1} \};t)$ is the \emph{toric $h$-polynomial} of $P$ (see \cite[(3),(4)]{BraRemarks}). When $P$ is a rational polytope, the coefficients of the toric $h$-polynomial have a well-known interpretation as the dimensions of intersection cohomology of the toric variety associated to the normal fan of $P$. For the relation with the geometry in this paper, see Remark~\ref{r:geometric} and 
 Section~6.2 of \cite{geometricpaper}, where the intersection cohomology groups above contribute to the `non-primitive' intersection cohomology of a hypersurface with Newton polytope $P$. 
\end{example}

\subsection{Strong formal subdivisions}

A function $f: \Gamma \rightarrow B$ between finite posets is \define{order-preserving} if $y \le y' \in \Gamma$ implies $f(y) \le f(y') \in B$.
If $(\Gamma,r_\Gamma)$ and $(B,r_B)$ are ranked posets, then a function  $f: \Gamma \rightarrow B$ is \define{rank-increasing} if $r_\Gamma(y) \le r_B(f(y))$ for all $y \in \Gamma$. 
Throughout, we let $\sigma: \Gamma \rightarrow B$ be an order-preserving, rank-increasing function  between locally Eulerian posets with rank functions $\rho_\Gamma$ and $\rho_B$ respectively. 
\begin{definition}
Consider an order-preserving  function $\sigma: \Gamma \rightarrow B$ between locally Eulerian posets. Then for all $y \in \Gamma$ and $x \in B$ such that $\sigma(y) \le x$, we may 
consider the lower Eulerian posets $\Gamma_{\ge y} := \{ y' \in \Gamma \mid y \le y' \}$ and $(\Gamma_{\ge y})_x := \{ y' \in \Gamma \mid y \le y', \sigma(y') \le x \}$, and the locally Eulerian poset 
$\IInt((\Gamma_{\ge y})_x):=\sigma^{-1}(x) \cap (\Gamma_{\ge y})_x$.   If $\Gamma$ is lower Eulerian, then we write $\Gamma_x  := (\Gamma_{\ge \hat{0}_\Gamma})_x$ and 
 $\IInt(\Gamma_x)  := \IInt((\Gamma_{\ge \hat{0}_\Gamma})_x)$.
\end{definition}

\begin{definition}
Let $\sigma: \Gamma \rightarrow B$ be an order-preserving, rank-increasing function  between locally Eulerian posets with rank functions $\rho_\Gamma$ and $\rho_B$ respectively. 
Then $\sigma$ is \define{strongly surjective} if it is surjective and for all $y \in \Gamma$ and $x \in B$ such that $\sigma(y) \le x$, there exists $y \le y' \in \Gamma$ such that $\rho_\Gamma(y') = \rho_B(x)$ and $\sigma(y') = x$. 
\end{definition}

Below we introduce a well-behaved notion of morphisms of locally Eulerian posets that are fibered in locally Eulerian posets.  In the way that Eulerian posets are modelled on polytopes, these morphisms are modelled on subdivisions of polytopes.  The property of subdivisions that we will abstract is that the subdivision of the relative interior of a polytope has Euler characteristic equal to $1$.

\begin{definition}\label{d:sfs}
Let $\sigma: \Gamma \rightarrow B$ be an order-preserving, rank-increasing function  between locally Eulerian posets with rank functions $\rho_\Gamma$ and $\rho_B$ respectively. 
Then $\sigma$ is a \define{strong formal subdivision} if it is strongly surjective and  for all $y \in \Gamma$ and $x \in B$ such that $\sigma(y) \le x$,
\begin{equation}\label{e:sfs}
\sum_{ \substack{y \le y'\\ \sigma(y') = x}} (-1)^{\rho_B(x) - \rho_\Gamma(y')} = 1.
\end{equation}
If $\Gamma$ and $B$ are lower Eulerian, then it follows that $\sigma(\hat{0}_\Gamma) = \hat{0}_B$, and the \define{rank} $\rk(\sigma) \in \Z_{\ge 0}$ of $\sigma$ is defined to be $\rk(\sigma) := \rho_B(\hat{0}_B) - \rho_\Gamma(\hat{0}_\Gamma)$. Alternatively, it follows from the fact that $\sigma$ is rank-increasing and strongly surjective that $\rk(\sigma) = \rk(\Gamma)-\rk(B)$. 
\end{definition}

The following is a straightforward consequence of M\"{o}bius inversion.

\begin{lemma}\label{l:equivalent}
With the notation above, let $\sigma:\Gamma\rightarrow B$ be order-preserving, rank-increasing and strongly surjective.  Then $\sigma$ is a strong formal subdivision if and only if 
\[
\sum_{ \substack{ \sigma(y') \le x \\ y \le y' }}  (-1)^{\rho_B(x) - \rho_\Gamma(y')} =   \left\{\begin{array}{cl} 1 & \text{if }  \sigma(y) = x,  \\ 0 & \text{otherwise}, \end{array}\right.  \textrm{ for all } y \in \Gamma, x \in B \textrm{ such that } \sigma(y) \le x.
\]
\end{lemma}

\begin{example} The simplest example of a strong formal subdivision is the identity where $\Gamma=B$ and $\rho_\Gamma = \rho_B$.
\end{example}

\begin{remark}\label{r:induced}
If $\sigma: \Gamma \rightarrow B$ is a strong formal subdivision, then for all $y \in \Gamma$, the restriction 
$\sigma: \Gamma_{\ge y} \rightarrow B_{\ge \sigma(y)}$ is a strong formal subdivision of rank $\rho_B(\sigma(y)) - \rho_\Gamma(y)$ between lower Eulerian posets
 with rank functions restricted from $\Gamma$ and $B$ respectively.
\end{remark}

\begin{remark}\label{r:restrict} One may restrict strong formal subdivisions to order ideals.  More precisely, let $\sigma:\Gamma\rightarrow B$ be a strong formal subdivision and let $I$ be an order ideal of $B$ i.e. if $x' \in I$ and $x < x'$ then $x \in I$. Then $\sigma^{-1}(I)$ and $I$ are locally Eulerian posets with rank functions restricted from $\Gamma$ and $B$ respectively. One easily verifies from the definitions that the restricted function 
$\sigma: \sigma^{-1}(I) \rightarrow I$ is a strong formal subdivision. 
\end{remark}

\begin{remark}\label{r:localrestrict}
Combining Remark~\ref{r:induced} and Remark~\ref{r:restrict}, we see that if $\sigma: \Gamma \rightarrow B$ is a strong formal subdivision, and $y \in \Gamma$, $x \in B$ such that $\sigma(y) \le x$, then the 
 restriction 
$\sigma: (\Gamma_{\ge y})_x \rightarrow [ \sigma(y),x]$ is a strong formal subdivision of rank $\rho_B(\sigma(y)) - \rho_\Gamma(y)$ between the lower Eulerian poset $(\Gamma_{\ge y})_x$ and 
the Eulerian poset  $[ \sigma(y),x]$ with rank functions restricted from $\Gamma$ and $B$ respectively.
\end{remark}

For a poset $B$ with minimum element $\hat{0}$, define the barycentric subdivision, $\Bary(B)$ to be the set of all chains of elements of $B$ containing $\hat{0}$, partially ordered under refinement.   The rank of a chain is the number of non-zero elements in the chain. Let $B$ be a lower Eulerian poset with rank function $\rho_B(x) = \rho_B(\hat{0}_B,x)$. There is a rank-increasing map $\sigma:\Bary(B)\rightarrow B$ given by
\[\sigma:\hat{0} = x_0 < x_1<x_2<\ldots<x_k \mapsto x_k.\]

\begin{lemma}\label{l:bary}
The barycentric subdivision $\sigma:\Bary(B)\rightarrow B$ of a lower Eulerian poset is a strong formal subdivision of rank $0$.
\end{lemma}

\begin{proof}
We first show that $\sigma$ is strongly surjective.
Let $y=\{\hat{0} = x_0 < x_1<x_2<\ldots<x_k\}\in \Bary(B)$. Then if $\sigma(y)=x_k \le x$, we can refine $y$ to a maximal chain $y'$ terminating at $x$.  Consequently, $\sigma(y')=x$ and $\rho_{\Bary(B)}(y')=\rho_B(x)$.

Now, let $y\in \Bary(B)$ and $x \in B$ such that $\sigma(y) \le x$.  We must show
\[\sum_{ \substack{y \le y'\\ \sigma(y') = x}} (-1)^{\rho_B(x) - \rho_{\Bary(B)}(y')} = 1.\]
The elements $y'$ satisfying $y'\geq y$ and $\sigma(y')=x$ are the chains refining $y$ and terminating with $x$.  Without loss of generality, we may suppose that $y$ terminates with a non-zero element $x$.  Therefore, such chains must refine
\[\hat{0} = x_0 < x_1<x_2<\ldots  <x_k=x.\]
Let $\cC(x_i,x_{i+1})$ be the set of all chains in $B$ starting with $x_i$ and terminating with $x_{i+1}$. 
 For $z\in \cC(x_i,x_{i+1})$, let $l(z)$ be the length of the chain.  Then the sum can be rewritten
\[(-1)^{\rho_B(x)}\prod_{i=0}^{k-1} \sum_{z\in \cC(x_i,x_{i+1})} (-1)^{l(z)}=(-1)^{\rho_B(x)}\prod_{i=1}^{k-1} \mu_B(x_i,x_{i+1})=1\]
where we have used Philip Hall's theorem \cite[Prop~3.8.5]{Stanley1} and the fact that $\mu_B(x_i,x_{i+1})=(-1)^{\rho_B(x_{i+1})-\rho_B(x_i)}$ for a locally Eulerian poset. 
\end{proof}

Our most important examples are polyhedral subdivisions of polytopes. 

\begin{definition}
A \define{polyhedral subdivision} $\cS$ of a polytope $P\subset\R^n$ is a subdivision of $P$ into a finite number of polytopes such that the intersection of any two polytopes is a (possibly empty) face of both.  A \define{lattice polyhedral subdivision} of a lattice polytope $P$ is a polyhedral subdivision of $P$ into lattice polytopes.  
Let $F$ be  a (possibly empty) cell of $\cS$. The \define{link} $\lk_\cS(F)$  of $F$ in $\cS$ is the subcomplex consisting of all 
cells $F'$ of $\cS$ that contain $F$ under inclusion.  We will often abuse notation and identify $\lk_\cS(F)$ with its associated lower Eulerian poset. 
\end{definition}

As in Example~\ref{e:polytope}, let $\cS'$ and $\cS$ be polyhedral subdivisions of a polytope $P$, such that $\cS'$ is a refinement of $\cS$.  We identify $\cS$ and $\cS'$ with their face posets ordered under inclusion.
For every cell $F'$ in $\cS'$, let $\sigma(F')$ denote the smallest cell of $\cS$ containing $F'$. When $F'$ is the empty cell of $\cS'$, $\sigma(F')$ is the empty cell of $\cS$. By abuse of notation, we write $\sigma: \cS' \rightarrow \cS$ for the corresponding order-preserving function of posets. For example, when $\cS$ is the trivial subdivision of $P$,  then the corresponding poset is the face poset of $P$. 

\begin{lemma} \label{l:polyhedralsubdivision} Let $\cS'$ and $\cS$ be polyhedral subdivisions of a polytope $P$ such that $\cS'$ refines $\cS$.  
Then the function $\sigma:\cS' \rightarrow \cS$  
is a strong formal subdivision of rank $0$.
\end{lemma}

\begin{proof}
The map $\sigma$ is clearly strongly surjective.
Now, let $F$ be a face of $\cS'$ and $Q$ be a face of $\cS$ containing $F$.  By Lemma~\ref{l:equivalent}, we must show
\[
\sum_{ \substack{ \sigma(F') \le Q \\ F \le F' }}  (-1)^{\dim Q - \dim F'} =   \left\{\begin{array}{cl} 1 & \text{if }  \sigma(F) = Q,  \\ 0 & \text{otherwise}. \end{array}\right.\]
The sum is the $(-1)^{\dim Q-\dim F-1}\tilde{\chi}(\lk_{\cS'|_Q}(F))$, where $\lk_{\cS'|_Q}(F)$ is link of $F$ in the restriction of the subdivision $\cS'$ to $Q$.  The conclusion follows because 
the link is contractible if $F \subseteq \partial Q$ (including the case when $F'$ is the empty cell), and has the topology of a sphere of dimension $\dim Q - \dim F- 1$ if $\sigma(F) = Q$ and $\dim F < \dim Q$. 
\end{proof}

\begin{remark} 
In the case of polyhedral and barycentric subdivisions, we made use of a topological interpretation of  strong formal  subdivisions.   We explain it in general here.
Let $\sigma: \Gamma \rightarrow B$ be an order-preserving function between lower Eulerian posets, and fix $y$ in $\Gamma$ and $x$ in $B$ 
such that $\sigma(y) \le x$. 
Consider the poset $\hat{(\Gamma_{\ge y})_x} := (\Gamma_{\ge y})_x \cup \{ \hat{1} \}$ obtained from $(\Gamma_{\ge y})_x$ by adding a maximal element $\hat{1}$, 
and let $(\Gamma_{> y})_x$ denote the poset obtained from $(\Gamma_{\ge y})_x$ by removing the minimal element $y$. 
 Since $(\Gamma_{\ge y})_x$ is lower 
Eulerian, one  computes the M\"obius
number  of $\hat{(\Gamma_{\ge y})_x}$  to be
\[
\mu(y,\hat{1}) = - \sum_{\substack{y \le y' \in \Gamma \\ \sigma(y') \le x }} (-1)^{\rho_\Gamma(y') - \rho_\Gamma(y)} =  -(-1)^{\rho_B(x)  - \rho_\Gamma(y)} \sum_{\substack{y \le y' \in \Gamma \\ \sigma(y') \le x }} (-1)^{\rho_B(x) - \rho_\Gamma(y')}. 
\]
By Philip Hall's theorem \cite[Prop~3.8.5]{Stanley1}, $\mu(y,\hat{1}) $ is equal to the reduced Euler characteristic 
$\tilde{\chi}(\Delta((\Gamma_{> y})_x))$ of the order 
complex $\Delta((\Gamma_{> y})_x)$ of  $(\Gamma_{> y})_x$. Recall that $\Delta((\Gamma_{> y})_x)$  is the simplicial complex with $i$-dimensional faces corresponding to 
chains of length $i$ in $(\Gamma_{> y})_x$. 
We conclude that the condition in Lemma~\ref{l:equivalent} is equivalent to the following: 
\[
-\tilde{\chi}(\Delta((\Gamma_{> y})_x)) =  \left\{\begin{array}{cl}   (-1)^{\rho_B(x) - \rho_\Gamma(y)}  & \text{if } \sigma(y) =  x,  \\ 0 & \text{otherwise}. \end{array}\right.
\]
\end{remark}

\subsection{Pushforwards along strong formal subdivisions}

\begin{definition}\label{d:push}
Let $(\Gamma,r_\Gamma)$ and $(B,r_B)$ be ranked posets, and consider  an order-preserving, rank-increasing function $\sigma: \Gamma \rightarrow B$. 
Then the corresponding \define{push-forward} map is the left $S(R)$-module homomorphism:
\[
\sigma_*: R^\Gamma \rightarrow R^B
\]
\[
(\sigma_*f)(x) = \sum_{y \in \sigma^{-1}(x)} f(y) q^{r_B(x) - r_\Gamma(y)} 
\]
For any $x \in B$ and $y \in \Gamma$, define $\eta(x,y) \in R$ by:
\[
\sigma_* (e_y * \gamma_{\Gamma}) = \sum_{x \in B} \eta(x,y) e_x. 
\]
\end{definition}

One easily verifies the following lemma. 

\begin{lemma}\label{l:compose}
Let $(\Omega,r_\Omega)$, $(\Gamma,r_\Gamma)$ and $(B,r_B)$ be ranked posets, and consider order-preserving, rank-increasing functions $\tau: \Omega \rightarrow \Gamma$ and $\sigma: \Gamma \rightarrow B$. Then $\sigma \circ \tau: \Omega \rightarrow B$ is order-preserving and rank-increasing, and $(\sigma \circ \tau)_* = \sigma_* \circ \tau_*$.
\end{lemma}

We will give an alternative criterion for a function to be a strong formal subdivision in terms of the pushforward map. 

\begin{proposition} \label{p:alternate}
Let $\sigma: \Gamma \rightarrow B$ be an order-preserving, rank-increasing function  between locally Eulerian posets with rank functions $\rho_\Gamma$ and $\rho_B$ respectively.
If $\sigma$ is strongly surjective, then the following are equivalent:
\begin{enumerate}

\item $\sigma$ is a strong formal subdivision,

\item the pushforward of an acceptable function is acceptable i.e. 
 \[ \sigma_*: \mathcal{A}(\Gamma,\kappa_\Gamma)\rightarrow \mathcal{A}(B,\kappa_B), \]

\item\label{i:pthird} for all $y$ in $\Gamma$ and $x$ in $B$ such that $\sigma(y) \le x$,
  \[ h( (\Gamma_{\ge y})_x;t) = t^{ \rho_B(x) - \rho_\Gamma(y)}  h( \IInt(\Gamma_{\ge y})_x;t^{-1}),\]  where 
\[
t^{ \rho_B(x) - \rho_\Gamma(y)}  h( \IInt(\Gamma_{\ge y})_x;t^{-1}) := \sum_{ \substack{ y \le y' \\ \sigma(y') = x   } }  g([y,y'];t) (t - 1)^{\rho_B(x) - \rho_\Gamma(y')}.
\]

\end{enumerate}
\end{proposition}
\begin{proof}
Consider any $y$ in $\Gamma$ and $x$ in $B$. 
We compute from the definitions:
\begin{equation}\label{e:definition}
\sigma_* (e_y * \gamma_{\Gamma})(x) = \sum_{ \substack{ y' \in \sigma^{-1}(x) \\ y \le y'} }   \gamma_{\Gamma}(y,y') q^{ \rho_B(x) - \rho_\Gamma(y')} \text{ if  }  \sigma(y) \le x,  
\end{equation}
and equals $0$ otherwise. On the one hand, if $\sigma(y)\le x$, since $\bar{q} = - q$ and $\gamma_\Gamma$ is totally acceptable: 
\begin{eqnarray*}
\bar{\sigma_* (e_y * \gamma_{\Gamma})}(x) &=&\sum_{ \substack{ y' \in \sigma^{-1}(x) \\ y \le y'} }   \bar{\gamma_{\Gamma}}(y,y') (-q)^{ \rho_B(x) - \rho_\Gamma(y')} \\
&=& \sum_{ \substack{ y' \in \sigma^{-1}(x) \\ y \le y'} }  \sum_{y \le z \le y'}  \gamma_{\Gamma}(y,z) q^{\rho_\Gamma(y') - \rho_\Gamma(z) } (-q)^{ \rho_B(x) - \rho_\Gamma(y')} \\
&=&\sum_{\substack{y \le z \\ \sigma(z) \le x}} \gamma_{\Gamma}(y,z) q^{\rho_B(x)-\rho_\Gamma(z)}\left(\sum_{\substack{z\le y'\\\sigma(y')=x}}(-1)^{\rho_B(x)-\rho_{\Gamma}(y')}\right),
\end{eqnarray*}
and equals $0$ otherwise.
On the other hand, if $\sigma(y)\le x$,
\begin{eqnarray*}
(\sigma_* (e_y * \gamma_{\Gamma})*\kappa_B)(x)& =& \sum_{x'\le x} 
\left(\sum_{ \substack{ z\in \sigma^{-1}(x') \\ y \le z} }   \gamma_{\Gamma}(y,z) q^{ \rho_B(x') - \rho_\Gamma(z)}\right) q^{\rho_B(x)-\rho_B(x')}\\
&=&  \sum_{ \substack{ y \le z \\\sigma(z)\le x} }   \gamma_{\Gamma}(y,z) q^{ \rho_B(x) - \rho_\Gamma(z)},
\end{eqnarray*}
and equals $0$ otherwise. It follows easily that \eqref{e:sfs} in Definition~\ref{d:sfs} holds if and only if $\bar{\sigma_* (e_y * \gamma_{\Gamma})} = \sigma_* (e_y * \gamma_{\Gamma})*\kappa_B$ for all $y \in \Gamma$. By Lemma~\ref{l:symacc}, the latter condition is equivalent to the condition that $\sigma_*$ takes acceptable functions to acceptable functions.

Finally,  for $\sigma(y) \le x$, unpacking the expressions for $\sigma_* (e_y * \gamma_{\Gamma})(x)$ and $(\sigma_* (e_y * \gamma_{\Gamma})*\kappa_B)(x)$ above using Lemma~\ref{l:gpoly} and Definition~\ref{d:hpoly}, yields
\[
\sigma_* (e_y * \gamma_{\Gamma})(x) = t^{ (\rho_B(x) - \rho_\Gamma(y))/2} h( \IInt(\Gamma_{\ge y})_x;t^{-1}),
\]
\[
(\sigma_* (e_y * \gamma_{\Gamma})*\kappa_B)(x) = t^{ (\rho_B(x) - \rho_\Gamma(y))/2}h( (\Gamma_{\ge y})_x;t^{-1}),
\]
and we deduce the equivalence with the third statement. 
\end{proof}

\begin{remark}\label{r:compose}
Let  $\tau: \Omega  \rightarrow  \Gamma$ and    $\sigma: \Gamma \rightarrow B$ be strong formal subdivisions.  Then 
$\sigma \circ \tau$ is strongly surjective and hence is a strong formal subdivision by Lemma~\ref{l:compose} and the acceptability criterion of Proposition~\ref{p:alternate}.
If $\Omega, \Gamma,B$ are lower Eulerian, then $\rk(\sigma \circ \tau) = \rk(\sigma) + \rk(\tau)$. 
\end{remark}

\begin{remark}\label{r:hpush}
If  $\sigma: \Gamma \rightarrow B$ is a strong formal subdivision, then for  all $y$ in $\Gamma$ and $x$ in $B$, a corollary of the proof of Proposition~\ref{p:alternate} is 
\[
\eta(x,y) = \sigma_* (e_y * \gamma_{\Gamma})(x) =   \left\{\begin{array}{cl} t^{ -(\rho_B(x) - \rho_\Gamma(y))/2} h( (\Gamma_{\ge y})_x;t)  & \text{if }  \sigma(y) \le x,  \\ 0 & \text{otherwise}. \end{array}\right.  
\]
\end{remark}

\begin{remark}\label{r:leading}
If  $\sigma: \Gamma \rightarrow B$ is a strong formal subdivision, then for  all $y$ in $\Gamma$ and $x$ in $B$ such that $\sigma(y) \le x$, condition~\eqref{i:pthird} in Proposition~\ref{p:alternate}, together with the 
fact that,  by definition, the degree of $g([y,y'];t)$ is strictly bounded by $(\rho_\Gamma(y') -  \rho_\Gamma(y))/2$ for $y < y' \in \Gamma$,  implies that 
\[
h( (\Gamma_{\ge y})_x;t) = \left\{\begin{array}{cl}  t^{\rho_B(x) - \rho_\Gamma(y)} + (\beta_{x,y} - ( \rho_B(x) - \rho_\Gamma(y) )  t^{\rho_B(x) - \rho_\Gamma(y) - 1} +  \textrm{l.o.t.} & \text{if } \sigma(y) = x,  \\ 
\beta_{x,y}   t^{\rho_B(x) - \rho_\Gamma(y) - 1} +  \textrm{l.o.t.} & \text{if } \sigma(y) < x.  \end{array}\right. ,
\]
where 
\[
\beta_{x,y}  = \# \{ y \le y' \mid \sigma(y') = x, \rho_\Gamma(y') = \rho_\Gamma(y) + 1 \}.
\]
\end{remark}

\begin{remark}
In the case that $\Gamma$ and  $B$ are lower Eulerian, 
 one may verify that the notion of a strong formal subdivision of rank $0$ is strictly stronger than Stanley's notion of a formal subdivision \cite[Definition~7.4]{StaSubdivisions}. 
\end{remark}

\section{Local invariants of strong formal subdivisions}\label{s:local}

In this section we introduce a 
symmetric polynomial associated to a strong formal subdivision called the \emph{local $h$-polynomial}.

We continue with the notation of the previous section, and let
$\sigma:\Gamma\rightarrow B$ be a strong formal subdivision between locally Eulerian posets with rank functions $\rho_\Gamma$ and $\rho_B$ respectively. By Proposition~\ref{p:alternate}, push-forward gives a left $S(R)$-module homomorphism:
\[
\sigma_*: \mathcal{A}(\Gamma,\kappa_\Gamma)\rightarrow \mathcal{A}(B,\kappa_B)
\]
We will write an expression for this linear map in terms of the bases $\{e_y*\gamma_\Gamma\}$ and  $\{e_x*\gamma_B\}$.
Explicitly, we will describe the elements $\lambda(x,y)\in S(R)$ defined  by
\begin{equation}\label{e:lambda}
\sigma_*(e_y*\gamma_\Gamma)=\sum_x \lambda(x,y)e_x*\gamma_B =  \sum_x \eta(x,y)e_x.
\end{equation}
First we need the following definition. 

\begin{definition} \label{d:localpoly} 
Let $\sigma: \Gamma \rightarrow B$ be a strong formal subdivision  between a lower  Eulerian poset $\Gamma$ and an Eulerian poset $B$. 
Then the  \define{local $h$-polynomial} $l_B(\Gamma;t) \in \Z[t]$ is defined by
\[
l_B(\Gamma;t) = \sum_{ x \in B } h(\Gamma_{x};t) (-1)^{\rho_B(x,\hat{1}_B)} g([x,\hat{1}_B]^*;t).
\]
We also introduce the following convenient notation. 
Let $\sigma: \Gamma \rightarrow B$ be a strong formal subdivision  between locally Eulerian posets with rank functions $\rho_\Gamma$ and $\rho_B$ respectively, and 
fix $y \in \Gamma$ and  $x \in B$. 
If $\sigma(y) \le x$, then
recall from Remark~\ref{r:localrestrict} that we may consider the restricted strong formal subdivision 
$\sigma: (\Gamma_{\ge y})_x \rightarrow [ \sigma(y),x]$. In this case we set 
\[
l_B(\Gamma,x,y;t) := l_{[ \sigma(y),x]}((\Gamma_{\ge y})_x;t) = \sum_{\sigma(y) \le x' \le x} h((\Gamma_{\ge y})_{x'};t) (-1)^{\rho_B(x',x)} g([x',x]^*;t). 
\]
We set $l_B(\Gamma,x,y;t) = 0$ if $\sigma(y) \nleq x$. We will abuse notation throughout by ignoring the dependence of the local $h$-polynomial on $\sigma$. 
\end{definition}

\begin{remark}\label{r:reverseh}
Let $\sigma: \Gamma \rightarrow B$ be a strong formal subdivision  between a lower  Eulerian poset $\Gamma$ and an Eulerian poset $B$. 
It follows from Theorem~\ref{t:inverse} that we recover the $h$-polynomial of $\Gamma$ via:
\[
h(\Gamma;t)= \sum_{ x \in B }  l_{[\hat{0}_B,x]}(\Gamma_x;t) g([x,\hat{1}_B];t).
\]
\end{remark}

\begin{remark}
Let $\sigma: \Gamma \rightarrow B$ be a strong formal subdivision  between a lower  Eulerian poset $\Gamma$ and an Eulerian poset $B$, and assume that  $\rk(\sigma) = 0$. 
In this case, 
the definition of the local $h$-polynomial $l_B(\Gamma;t)$ agrees with Stanley's definition \cite[Corollary~7.7]{StaSubdivisions}.
Moreover, for all $y \in \Gamma$, if $\sigma$ is induced by a homology subdivision of a simplex or a polyhedral subdivision of a polytope, then in this case the polynomial $l_B(\Gamma,\hat{1}_B,y;t)$ agrees with the \emph{relative local $h$-polynomial} introduced by
Athanasiadis in \cite{AthFlag} and Nill and Schepers in \cite{NSCombinatorial} respectively. 
\end{remark}

We may now describe the elements $\lambda(x,y)\in S(R)$. 

\begin{lemma}\label{l:localpoly}
 Let
$\sigma:\Gamma\rightarrow B$ be a strong formal subdivision between locally Eulerian posets with rank functions $\rho_\Gamma$ and $\rho_B$ respectively. Then
\[ \lambda(x,y) = t^{-(\rho_B(x) - \rho_\Gamma(y))/2}l_B(\Gamma,x,y;t).\] 
\end{lemma}
\begin{proof}
Note that our definition of $\lambda(x,y) \in S(R)$ given in \eqref{e:lambda} 
is clearly equivalent to 
\begin{equation*}
\lambda(x,y) = \sum_{x'} \eta(x',y)(e_{x'} *  \gamma_B^{-1})(x) = \sum_{x' \le x} \eta(x',y) \gamma_B^{-1}(x',x). 
\end{equation*}
The result follows by substituting in our expressions for $\eta(x',y)$ and $\gamma_B^{-1}(x',x)$ in Remark~\ref{r:hpush} and Theorem~\ref{t:inverse} respectively. 
\end{proof}

We immediately deduce the following symmetry property of local $h$-polynomials (cf. \cite[Remark~3.7]{AthFlag}, \cite[Corollary~7.7]{StaSubdivisions}).

\begin{corollary}\label{c:symmetry}
Let $\sigma: \Gamma \rightarrow B$ be a strong formal subdivision  between a lower  Eulerian poset $\Gamma$ of rank $r$ and an Eulerian poset $B$. 
Then 
\[l_B(\Gamma;t)=t^{r}l_B(\Gamma;t^{-1}).\]
\end{corollary}
\begin{proof}
By Lemma~\ref{l:localpoly},  $\lambda(\hat{1}_B,\hat{0}_\Gamma) = t^{-(\rho_B(\hat{1}_B) - \rho_\Gamma(\hat{0}_\Gamma))/2}l_B(\Gamma;t)= t^{-r/2}l_B(\Gamma;t)$. The result follows since
$\lambda(\hat{1}_B,\hat{0}_\Gamma) \in S(R)$ is symmetric by definition. 
\end{proof}

%

We also deduce the following property of the pushforward map. 

\begin{corollary}
 Let
$\sigma:\Gamma\rightarrow B$ be a strong formal subdivision between locally Eulerian posets. Then  $\sigma_*: \mathcal{A}(\Gamma,\kappa_\Gamma)\rightarrow \mathcal{A}(B,\kappa_B)$
is surjective.
\end{corollary}
\begin{proof}
By Lemma~\ref{l:localpoly}, $\lambda(x,y) = 0$ for $\sigma(y) \nleq x$. For any $x \in B$, since $\sigma$ is strongly surjective, there exists an element $y \in \Gamma$ such that $\rho_\Gamma(y) = \rho_B(x)$ 
and $\sigma(y) = x$. Then $y$ is maximal in $\sigma^{-1}(x)$, and it follows that $\lambda(x,y) = l_B(\Gamma,x,y;t) =1$. The result follows.
\end{proof}

We have seen in Remark~\ref{r:compose} that strong formal subdivisions are closed under composition. 
The following is an immediate consequence of Lemma~\ref{l:localpoly} and the definition of $\lambda(x,y)$ as matrix elements of a linear transformation (cf. \cite[Proposition~3.6]{AthFlag}).
The second statement follows from the first by substituting $x = \hat{1}_B$ and $z = \hat{0}_\Omega$. 

\begin{corollary} \label{c:hrefine}
Let  $\tau: \Omega  \rightarrow  \Gamma$ and let $\sigma: \Gamma \rightarrow B$ be strong formal subdivisions  between locally Eulerian posets. Then for $z\in\Omega$, $x\in\Gamma$,
\[
l_B(\Omega,x,z;t) =  \sum_{ y \in \Gamma } l_B(\Gamma,x,y;t) l_\Gamma(\Omega,y,z;t). 
\]
In particular, if $\Omega$,$\Gamma$ are lower Eulerian and $B$ is Eulerian, then 
\[
l_B(\Omega;t) =  \sum_{ y \in \Gamma } l_{[\sigma(y),\hat{1}_B]}(\Gamma_{\ge y};t) l_{[\hat{0}_\Gamma,y]}(\Omega_y;t). 
\]
\end{corollary}

\begin{example}\label{e:lidentity}
If $B$ is Eulerian of rank $n$ and $\sigma:B\rightarrow B$ is the identity, then
\[
l_B(B;t)  =  \left\{\begin{array}{cl} 1 & \text{if } n = 0,  \\ 0 & \text{otherwise}. \end{array}\right.
\]
\end{example}

\begin{example}\label{e:localsingle}
If $\sigma: \Gamma \rightarrow B$ is a formal subdivision of a lower Eulerian poset $\Gamma$ by the single element poset $B$, then 
 $l_B(\Gamma;t) = h(\Gamma;t)$. 
\end{example}

\begin{example}\label{e:lconstant} 
Let $\sigma: \Gamma \rightarrow B$ be a strong formal subdivision  between a lower  Eulerian poset $\Gamma$ of rank $r$ and an Eulerian poset $B$ of rank $n$.
Substituting the expression for $h( \Gamma_{x};t)$ into Definition~\ref{d:localpoly} and  using the fact that,  by definition, 
the degree of $g([x,\hat{1}_B]^*;t)$ is strictly bounded by $\rho_B(x,\hat{1}_B)/2$ for $x < \hat{1}_B$, 
together with the symmetry of $l_B(\Gamma;t)$ from Corollary~\ref{c:symmetry}, we have:

\begin{itemize}

\item The constant term of $l_B(\Gamma;t)$ is equal to the coefficient of $t^{r}$ in $l_B(\Gamma;t)$, which is equal to 
\[
\left\{\begin{array}{cl}  1 & \text{if } n = 0,  \\ 
0 & \text{if } n > 0.  \end{array}\right. 
\]

\item The linear coefficient of $l_B(\Gamma;t)$ is equal to the coefficient of $t^{r - 1}$ in $l_B(\Gamma;t)$, which is equal to 
\[
 \left\{\begin{array}{cl}   \beta- r  & \text{if } n=0,  \\ 
\beta  - 1  & \text{if } n=1, \\
\beta & \text{if }  n > 1.
  \end{array}\right. ,
\]
where 
\[
\beta = \# \{ y \in \Gamma \mid \sigma(y) = \hat{1}_B, \rho_\Gamma(\hat{0}_\Gamma,y) =  1 \}.
\]
\end{itemize}
\end{example}

\begin{example}\label{e:lsmallcases} 
Let $\sigma: \Gamma \rightarrow B$ be a strong formal subdivision of rank $\rk(\sigma) = r - n$ between a lower  Eulerian poset $\Gamma$ of rank $r$ and an Eulerian poset $B$ of rank $n$.
By Example~\ref{e:lconstant}, we have explicit formulas for the local $h$-polynomial when 
$r \le 3$:
\[
l_B(\Gamma;t) = \left\{\begin{array}{cl} 1 & \text{if } (n,r-n) =  (0,0),  \\ 
1 + t & \text{if } (n,r-n) =  (0,1),  \\ 
0 & \text{if } (n,r-n) =  (1,0),  \\ 
1 + (\beta - 2)t  + t^2 & \text{if } (n,r-n) =  (0,2),  \\ 
(\beta - 1)t & \text{if } (n,r-n) =  (1,1),  \\ 
\beta t & \text{if } (n,r-n) =  (2,0),  \\ 
1 + (\beta - 3)t  + (\beta - 3)t^2+ t^3 & \text{if } (n,r-n) =  (0,3),  \\ 
(\beta - 1) t(1 + t)  & \text{if } (n,r-n) =  (1,2),  \\ 
\beta  t(1 + t)  & \text{if } (n,r-n) =  (2,1),  \\ 
\beta  t(1 + t)  & \text{if } (n,r-n) =  (3,0),  \end{array}\right.
\]
where 
\[
\beta = \# \{ y \in \Gamma \mid \sigma(y) = \hat{1}_B, \rho_\Gamma(\hat{0}_\Gamma,y) =  1 \}.
\]
\end{example}

A poset $B$ with minimal element $\hat{0}$ is \define{simplicial} if for every $x$ in $B$, the interval $[\hat{0},x]$ is a Boolean algebra. This implies that every interval of $B$ is a Boolean algebra, 
and, in  particular,  $B$ is lower Eulerian. 
We have the following 
generalization of \cite[Proposition~2.2]{StaSubdivisions},  
which is useful for computing examples. 

\begin{lemma}\label{l:simplicialsub}
Let $\sigma: \Gamma \rightarrow B$ be a strong formal subdivision between 
a simplicial
 poset $\Gamma$  of rank $r$ and
 a Boolean algebra $B$ of rank $n$ with rank functions $\rho_\Gamma$ and $\rho_B$ respectively. 
 Then 
\[
l_B(\Gamma;t)  = \sum_{ y \in \Gamma } (-1)^{r - \rho_\Gamma(\hat{0}_\Gamma,y)} t^{r- e(y)}(t - 1)^{e(y)},
\]
where $e(y) = \rho_B(\sigma(y)) - \rho_\Gamma(y)$  is the \define{excess} of $y$.

\end{lemma}
\begin{proof}
By Definition~\ref{d:localpoly} and Example~\ref{e:boolean},
\[
l_B(\Gamma;t) = \sum_{ x \in B } h(\Gamma_{x};t) (-1)^{\rho_B(x,\hat{1}_B)},
\]
Also, by Definition~\ref{d:hpoly}, 
\[
t^{\rho_B(x) - \rho_\Gamma(\hat{0}_\Gamma)} h(\Gamma_{x};t^{-1}) = \sum_{ \substack{ y \in \Gamma \\ \sigma(y) \le x } }   (t - 1)^{\rho_B(x) - \rho_\Gamma(y)}.
\]
Hence 
\begin{align*}
l_B(\Gamma;t) &= \sum_{x \in B}  (-1)^{\rho_B(x,\hat{1}_B)} \sum_{ \substack{ y \in \Gamma \\ \sigma(y) \le x } }   t^{ \rho_\Gamma(\hat{0}_\Gamma,y)}  (1 - t)^{\rho_B(x) - \rho_\Gamma(y)} \\
&= \sum_{ y \in \Gamma } (-1)^{\rho_B(\hat{1}_B) - \rho_\Gamma(y)}t^{ \rho_\Gamma(\hat{0}_\Gamma,y)}  (t - 1)^{e(y)}   \sum_{\sigma(y) \le x \le \hat{1}_B} (t - 1)^{\rho_B(x) - \rho_B(\sigma(y))} 
\end{align*}
Since $[\sigma(y), \hat{1}_B]$ is a Boolean algebra, the latter sum in the above formula equals  $t^{\rho_B(\sigma(y),\hat{1}_B)}$, and the result follows.
\end{proof}

\begin{remark}\label{r:extend}
We note that the results of  Sections~\ref{s:locEulerian}  and \ref{s:local} hold for any commutative ring $R$ with involution that contains $\Z[t^{\pm 1/2}]$ as a subring such that the induced involution is $\bar{r(t)} = r(t^{-1})$. 
More specifically, let $B$ be a locally Eulerian poset and consider the $B$-kernel $\kappa(x,x')=q^{\rho(x,x')}$ with $q = t^{1/2} - t^{-1/2}$. 
If one chooses a splitting 
 $R = S(R) \oplus R'$, with $\{ \sum_{i \in \Z_{< 0}} \alpha_i t^{i/2} \mid \alpha_i \in \Z \} \subseteq R' $, then the corresponding acceptability operator is  described as in Lemma~\ref{l:gpoly}, i.e. $\gamma_B(x,x') = t^{-\rho(x,x')/2} g([x,x'];t)$. 
 The rest of the sections hold verbatim. In particular, the results above are independent of the choice of such an $R'$, and, as such, in what follows we will ignore such a choice.

 For example, we may set $R = K_0(\Var_k)[\L^{\pm 1/2}]$ with $\Z[\L^{\pm 1/2}] \subseteq R$ as in Example~\ref{e:groth} (see Section~\ref{s:trop}). 
 In subsequent sections, we will use the following two examples. 
 Firstly, we consider $R_{uv}=\Z[u^{\pm 1/2},v^{\pm 1/2}]$ with involution $\bar{r(u,v)} = r(u^{-1},v^{-1})$, and $\Z[(uv)^{\pm 1/2}] \subseteq R_{uv}$. Secondly, we consider 
  $R_{uvw}=\Z[u^{\pm 1/2},v^{\pm 1/2},w^{\pm 1}]$ with involution $\bar{r(u,v,w)} = r(u^{-1},v^{-1},w^{-1})$, and $\Z[(uvw^2)^{\pm 1/2}] \subseteq R_{uvw}$. 
\end{remark}

\section{The mixed $h$-polynomial}\label{s:relmixed}

In this section, we introduce and study the mixed $h$-polynomial, a two-variable invariant of
strong  subdivisions of posets.  The results of this section are presented for their own interest and are not required in the sequel.

\subsection{Definition}

We will continue with the notation of previous section with $R=\Z[t^{\pm 1/2}]$, but, as in Remark~\ref{r:extend} we will also consider the ring $R_{uv}=\Z[u^{\pm 1/2},v^{\pm 1/2}]$ with involution $\bar{r(u,v)} = r(u^{-1},v^{-1})$
and inclusion $\Z[(uv)^{\pm 1/2}] \subseteq R_{uv}$ i.e. $uv$ will play the role of $t$ in the previous sections. For example, as in Lemma~\ref{l:gpoly}, 
the acceptability operator of a locally Eulerian poset $B$ is given by $\gamma_B(x,x')|_{t = uv} = (uv)^{-\rho(x,x')/2} g([x,x'];uv)$.

We introduce our main definition below. 

\begin{definition}\label{d:mixedpoly}
Let $\sigma: \Gamma \rightarrow B$ be a strong formal subdivision  between a lower  Eulerian poset $\Gamma$ and an Eulerian poset $B$. 
Then the  \define{mixed $h$-polynomial} $h_B(\Gamma;u,v) \in \Z[u,v]$ is defined by
\[
h_B(\Gamma;u,v) =  \sum_{x \in B } v^{\rk(\Gamma_x)}l_{[\hat{0}_B,x]}(\Gamma_x;uv^{-1}) g([x,\hat{1}_B];uv). 
\]
We also introduce the following convenient notation. 
Let $\sigma: \Gamma \rightarrow B$ be a strong formal subdivision  between locally Eulerian posets with rank functions $\rho_\Gamma$ and $\rho_B$ respectively, and 
fix $y \in \Gamma$ and  $x \in B$. 
If $\sigma(y) \le x$, then
recall from Remark~\ref{r:localrestrict} that we may consider the restricted strong formal subdivision 
$\sigma: (\Gamma_{\ge y})_x \rightarrow [ \sigma(y),x]$. In this case we set 
\[
h_B(\Gamma,x,y;u,v) := h_{[ \sigma(y),x]}((\Gamma_{\ge y})_x;u,v) =  \sum_{\sigma(y) \le x' \le x} v^{\rho_B(x') - \rho_\Gamma(y)} l_{[\sigma(y),x']}((\Gamma_{\ge y})_{x'};uv^{-1}) g([x',x];uv).
\]
We set $h_B(\Gamma,x,y;u,v) = 0$ if $\sigma(y) \nleq x$. We will abuse notation throughout by ignoring the dependence of the mixed $h$-polynomial on $\sigma$. 
\end{definition}

We have the following interpretation of the mixed $h$-polynomial in terms of the pushforward map. 
Let $\sigma: \Gamma \rightarrow B$ be a strong formal subdivision  between locally Eulerian posets. 
Using the results of  Sections~\ref{s:locEulerian}  and \ref{s:local}, together with Remark~\ref{r:extend}, we have a left $S(R_{uv})$-module homomorphism
\[ \sigma_*: \mathcal{A}(\Gamma,\kappa_\Gamma; R_{uv})\rightarrow \mathcal{A}(B,\kappa_B; R_{uv}), \]
determined by 
\[\sigma_*(e_y*\gamma_\Gamma|_{t = uv}) = \sum_x \lambda(x,y)|_{t = uv} e_x*\gamma_B|_{t = uv} = \sum_x \eta(x,y)|_{t = uv} e_x.\]
Note that since  $\lambda(x,y) \in S(R)$, it follows that $\lambda(x,y)|_{t = uv^{-1}} \in S(R_{uv})$. Hence the definition below is well-defined.

\begin{definition}\label{d:mixedpush}
Let $\sigma: \Gamma \rightarrow B$ be a strong formal subdivision  between locally Eulerian posets. The \define{mixed push-forward} is the left $S(R_{uv})$-module homomorphism
\[ \tilde{\sigma}_*: \mathcal{A}(\Gamma,\kappa_\Gamma; R_{uv})\rightarrow \mathcal{A}(B,\kappa_B; R_{uv}), \]
defined by
\[ \tilde{\sigma}_*(e_y*\gamma_\Gamma|_{t = uv})=\sum_{x \in B} \lambda(x,y)|_{t = uv^{-1}}   e_x*\gamma_B|_{t = uv} = \sum_{x \in B} \tilde{\eta}(x,y) e_x,\]
for some $\tilde{\eta}(x,y) \in R_{uv}$.
\end{definition}

\begin{remark}\label{r:tildecompose}
If $\tau: \Omega \rightarrow \Gamma$  and $\sigma:\Gamma\rightarrow B$ are strong formal subdivisions of locally Eulerian posets, then $\tilde{(\sigma \circ \tau)}_* = \tilde{\sigma}_* \circ \tilde{\tau}_*$. It follows that 
for every $x \in B$ and $z \in \Omega$:
\[
\tilde{\eta}(x,z) = \sum_{y \in \Gamma} \tilde{\eta}(x,y) \lambda(y,z)|_{t = uv^{-1}}  . 
\]
\end{remark}

The following lemma expresses the mixed pushforward in terms of the mixed $h$-polynomial. 

\begin{lemma}\label{l:mixedpoly}
 Let
$\sigma:\Gamma\rightarrow B$ be a strong formal subdivision between locally Eulerian posets with rank functions $\rho_\Gamma$ and $\rho_B$ respectively. Then 
\[
\tilde{\eta}(x,y) = (uv)^{-(\rho_B(x) - \rho_\Gamma(y))/2}h_B(\Gamma,x,y;u,v).
\]

\end{lemma}
\begin{proof}
By definition,
\[
\tilde{\eta}(x,y) = \sum_{x' \le x} \lambda(x',y)|_{t = uv^{-1}} \gamma_B(x',x)|_{t = uv}. 
\]
The result follows by substituting the expressions for $\lambda(x',y)$ and  $\gamma_B(x',x)$ in Lemma~\ref{l:localpoly} and Lemma~\ref{l:gpoly} respectively into the right 
hand side and comparing with the definition of $h_B(\Gamma,x,y;u,v)$.
\end{proof}

Below we summarize some of the elementary properties of the mixed $h$-polynomial. These generalize properties of the $h$-polynomial from previous sections.

\begin{theorem}\label{t:refineprop}
Let $\sigma: \Gamma \rightarrow B$ be a strong formal subdivision of rank $r - n$  between a lower  Eulerian poset $\Gamma$ of rank $r$ and an Eulerian poset $B$ of rank $n$. 
Then the  mixed $h$-polynomial $h_B(\Gamma;u,v)$ satisfies the following properties:
\begin{enumerate}

\item\label{i:rh1'} ($uv$-interchange) The  mixed $h$-polynomial is invariant under the interchange of $u$ and $v$ i.e. \[ h_B(\Gamma;u,v) = h_B(\Gamma;v,u). \] 

\item\label{i:rh2'} (specialization) We recover the $h$-polynomial of $\Gamma$ via the specialization \[ h_B(\Gamma;u,1) = h(\Gamma;u).\] 

\item\label{i:rh7'} (symmetry) We have
\[
(uv)^{r} h_B(\Gamma;u^{-1},v^{-1}) = \sum_{x \in B} h_{[\hat{0}_B,x]}(\Gamma_x;u,v) (uv - 1)^{\rho_B(x,\hat{1}_B)}. 
\]

\item\label{i:rh3'} (constant terms) We have $h_B(\Gamma;0,v) = v^{r - n}$. 

\item\label{i:rh4'} (identity subdivision) If $\Gamma = B$ and $\sigma$ is the identity function, then $h_B(B;u,v) =  g(B;uv)$. 

\item\label{i:rh5'} (interior) If $n = 0$, i.e. $B$ is the single element poset, then 
 \[ h_B(\Gamma;u,v) = v^{r} l_B(\Gamma;uv^{-1}) = 
v^{r} h(\Gamma;uv^{-1}). \]

\item\label{i:rh6'} (degree) All terms in $h_B(\Gamma;u,v)$ have combined degree  in $u$ and $v$ at most $r$, and 
the terms of combined degree $r$ equal $v^{r}l_{B}(\Gamma;uv^{-1})$.

\item\label{i:rh6''} (inversion)  We recover the local $h$-polynomial via
\[
v^{r}  l_{B}(\Gamma;uv^{-1}) = \sum_{x \in B}  h_{[\hat{0}_B,x]}(\Gamma_x;u,v) (-1)^{\rho_B(x,\hat{1}_B)} g([x,\hat{1}_B]^*;uv). 
\]

\item\label{i:rh8'} (composition)
 Suppose that $\tau: \Omega \rightarrow \Gamma$  and $\sigma:\Gamma\rightarrow B$ are strong formal subdivisions of locally Eulerian posets.  Then for $z\in\Omega$, $x\in\Gamma$,
\[
h_B(\Omega,x,z;u,v) = \sum_{y \in \Gamma}  h_B(\Gamma,x,y;u,v) v^{\rho_\Gamma(y) - \rho_\Omega(z)}l_\Gamma(\Omega,y,z;uv^{-1}). 
\]
In particular, if $\Omega$,$\Gamma$ are lower Eulerian and $B$ is Eulerian, then 
\[
h_B(\Omega;u,v) = \sum_{y \in \Gamma}  h_{[\sigma(y), \hat{1}_B]}(\Gamma_{\ge y};u,v) v^{\rk(\Omega_y)} l_{[\hat{0}_\Gamma,y]}(\Omega_y;uv^{-1}). 
\]

\end{enumerate}
\end{theorem}
\begin{proof}

Property~\eqref{i:rh1'} follows from Definition~\ref{d:mixedpoly} and the symmetry of the local $h$-polynomial (Corollary~\ref{c:symmetry}).

Property~\eqref{i:rh2'} follows from Definition~\ref{d:mixedpoly} and Remark~\ref{r:reverseh}. 


Property~\eqref{i:rh7'} follows from Lemma~\ref{l:mixedpoly} together with 
the fact that  $\sum_{x \in B} \tilde{\eta}(x,y) e_x \in \mathcal{A}(B,\kappa_B; R_{uv})$ by definition. 

Property~\eqref{i:rh3'} follows by substituting  Example~\ref{e:gconstant} and Example~\ref{e:lconstant} into Definition~\ref{d:mixedpoly}.
 
Property~\eqref{i:rh4'} follows from Definition~\ref{d:mixedpoly} and Example~\ref{e:lidentity}.

Property~\eqref{i:rh5'} follows from Definition~\ref{d:mixedpoly} and Example~\ref{e:localsingle}. 

Property~\eqref{i:rh6'} follows since $g([x',x];uv)$ has  combined degree in $u$ and $v$ strictly less than $\rho_B(x',x)$ for $x' < x$.

Property~\eqref{i:rh6''} is equivalent to 
$\sum_{x \in B} \lambda(x,y)|_{t = uv^{-1}}   e_x = \sum_{x \in B} \tilde{\eta}(x,y) e_x *\gamma_B^{-1}|_{t = uv}$, using Lemma~\ref{l:localpoly} , Lemma~\ref{l:mixedpoly} and Theorem~\ref{t:inverse}. 

Property~\eqref{i:rh8'} follows from Remark~\ref{r:tildecompose}, using Lemma~\ref{l:localpoly} and Lemma~\ref{l:mixedpoly}. 
The second statement follows from the first by substituting $x = \hat{1}_B$ and $z = \hat{0}_\Omega$. 
\end{proof}

\subsection{Examples}

In subsequent sections, we will be interested in the case of a strong formal subdivision induced by a polyhedral subdivision of a polytope, as in Lemma~\ref{l:polyhedralsubdivision}. In this case, the coefficients of the
mixed $h$-polynomial are non-negative integers by Theorem~\ref{t:localcoef}. In general, the following example shows that the coefficients of the  mixed $h$-polynomial may be negative, even when the
$h$-polynomial itself has non-negative coefficients. 

\begin{example}
The following example is due to C. Chan, and was originally used to show that the local $h$-polynomial may have negative coefficients 
\cite[Example~2.3 (h)]{StaSubdivisions}.  
Let $B$ be the face poset of a simplex $P$ with vertex set $\{1,2,3,4\}$. 
Let $F$ be the face spanned by $\{1,2,3\}$. Consider a deformation $F'$ of $F$ such that the boundary of $F$ is unchanged, but the interior of $F'$ is contained in the interior of $P$. 
If we add a vertex $\{5 \}$ to the interior of $F$, then we may consider the simplicial complex $\Gamma$ consisting of two simplices $\{1,2,3,4\}$ and $\{1,2,3,5\}$ joined along $F'$. Using, for example, 
Lemma~\ref{l:simplicialsub2},  one verifies
that the corresponding map $\sigma: \Gamma \rightarrow B$ is a strong formal subdivision with mixed $h$-polynomial $h_B(\Gamma;u,v) = 1 + uv(u + v) - (uv)^2$. Note that $l_B(\Gamma;t) = -t^2$ and $h(\Gamma;t ) = h_B(\Gamma;1,t) = 1 + t$.
\end{example}

Below, we present some explicit formulas for the  mixed $h$-polynomial. 

\begin{example}\label{e:rhsmallcases} 
Let $\sigma: \Gamma \rightarrow B$ be a strong formal subdivision of rank $\rk(\sigma) = r - n$ between a lower  Eulerian poset $\Gamma$ of rank $r$ and an Eulerian poset $B$ of rank $n$.
Using Example~\ref{e:lsmallcases}, we have explicit formulas for the mixed $h$-polynomial when 
$r \le 3$:
\[
h_B(\Gamma;u,v) = \left\{\begin{array}{cl} 1 & \text{if } (n,r-n) =  (0,0),  \\ 
u + v & \text{if } (n,r-n) =  (0,1),  \\ 
1 & \text{if } (n,r-n) =  (1,0),  \\ 
u^2 + (\beta - 2)uv  + v^2 & \text{if } (n,r-n) =  (0,2),  \\ 
u + v + (\beta - 1)uv & \text{if } (n,r-n) =  (1,1),  \\ 
1 + \beta uv & \text{if } (n,r-n) =  (2,0),  \\ 
u^3 + (\beta - 3)u^2v  + (\beta - 3)uv^2+ v^3 & \text{if } (n,r-n) =  (0,3),  \\ 
u^2 + (\mu - 2)uv  + v^2 + (\beta - 1) uv(u + v)  & \text{if } (n,r-n) =  (1,2),  \\ 
u + v + (\mu - 2)uv + \beta  uv(u + v)  & \text{if } (n,r-n) =  (2,1),  \\ 
1 + (\mu + \nu - 3)uv + \beta  uv(u + v)  & \text{if } (n,r-n) =  (3,0),  \end{array}\right.
\]
where 
\[
\beta = \# \{ y \in \Gamma \mid \sigma(y) = \hat{1}_B, \rho_\Gamma(\hat{0}_\Gamma,y) =  1 \},
\]
\[
\mu  = \# \{ y \in \Gamma  \mid  \rho_B(\sigma(y),\hat{1}_B) = 1, \rho_\Gamma(\hat{0}_\Gamma,y) =  1 \},
\]
\[
\nu = \# \{ x \in B \mid \rho_B(\hat{0}_B,x) = 1 \}. 
\]
\end{example}

The following lemma provides interesting examples of mixed $h$-polynomials, as well as demonstrating their combinatorial significance (cf. Example~\ref{e:triangulation}). 
Recall that a  poset $B$ with minimal element $\hat{0}$ is simplicial if for every $x$ in $B$, the interval $[\hat{0},x]$ is a Boolean algebra.

\begin{lemma}\label{l:simplicialsub2}
Let $\sigma: \Gamma \rightarrow B$ be a strong formal subdivision between 
a simplicial
 poset $\Gamma$  of rank $r$ and
 a Boolean algebra $B$ of rank $n$ with rank functions $\rho_\Gamma$ and $\rho_B$ respectively. Then
\[
h_B(\Gamma;u,v)  =  \sum_{ y \in \Gamma }  u^{\rho_\Gamma(\hat{0}_\Gamma,y)}  (1 - u)^{\rho_B(\sigma(y),\hat{1}_B)} (v - u)^{e(y)}, 
\]
where $e(y) = \rho_B(\sigma(y)) - \rho_\Gamma(y)$ is the \define{excess} of $y$. For any non-negative integers 
$i,j$, let $f_{i,j} = \# \{ y \in \Gamma \mid \rho_\Gamma(\hat{0}_\Gamma,y) = i, e(y) = j \}$. Then 
\[
h_B(\Gamma;u,v) = \sum_{i,j\ge 0} f_{i,j}    u^{i}  (1 - u)^{r - i - j} (v - u)^{j}, 
\]
and the mixed $h$-polynomial $h_B(\Gamma;u,v)$ determines and is determined by the numbers $\{ f_{i,j} \}_{i,j \ge 0}$.
\end{lemma}
\begin{proof}
We compute using Example~\ref{e:boolean} and Lemma~\ref{l:simplicialsub}, 
\begin{align*}
h_B(\Gamma;u,v) &= \sum_{x \in B} v^{\rho_B(x) - \rho_\Gamma(\hat{0}_\Gamma) }l_{[\hat{0}_B,x]}(\Gamma_x;uv^{-1})  \\
&= \sum_{x \in B} v^{\rho_B(x) - \rho_\Gamma(\hat{0}_\Gamma) } \sum_{ y \in \Gamma_x } (-1)^{\rho_B(x) - \rho_\Gamma(y)} (uv^{-1})^{\rho_B(x) - \rho_\Gamma(\hat{0}_\Gamma) - e(y)}(uv^{-1} - 1)^{e(y)} \\
&= \sum_{y \in \Gamma} (v - u)^{e(y)} u^{\rho_\Gamma(\hat{0}_\Gamma,y)} \sum_{\sigma(y) \le x}   (-u)^{\rho_B(x) - \rho_B(\sigma(y))}\\
&= \sum_{y \in \Gamma} (v - u)^{e(y)} u^{\rho_\Gamma(\hat{0}_\Gamma,y)} (1 - u)^{\rho_B(\sigma(y),\hat{1}_B)}.
\end{align*}
The second equality is a direct consequence. Finally, to see that the mixed $h$-polynomial determines the numbers $f_{i,j}$, note that
\[
(1 + u)^{r}h_B(\Gamma;\frac{1}{1 + u},\frac{1 + v}{1 + u}) = \sum_{i,j} f_{i,j} u^{r - i  - j} v^j.  
\]
\end{proof}
 
Finally, we present the following interesting, concrete example of Lemma~\ref{l:simplicialsub2}.

\begin{example}
Let $\cS$ be the barycentric subdivision of a simplex $P$ with $n$ vertices. Note that $\cS$ may be realized as a rational polyhedral subdivision (cf. Section~\ref{s:geometric}). 
By either Lemma~\ref{l:bary} or Lemma~\ref{l:polyhedralsubdivision}, we may consider the corresponding strong formal subdivision $\sigma: \cS \rightarrow [\emptyset,P]$. 
Then the $h$-polynomial of $\cS$ is the \emph{Eulerian polynomial} $A_n(t)$ (see, for example, \cite[(7)]{StaSubdivisions}). That is,
given a permutation $w \in \Sym_n$, let $\ex(w) = \# \{ i \mid w(i) > i \}$ be the number of \emph{excedances} of $w$.  Then 
\[
h(\cS;t) = A_n(t) = \sum_{w \in \Sym_n} t^{\ex(w)}. 
\]
The \emph{derangements} $\DD_n \subseteq \Sym_n$ are the permutations without any fixed points. 
In \cite[Proposition~2.4]{StaSubdivisions}, Stanley proved that 
\[
l_P(\cS;t) := l_{[\emptyset,P]}(\cS;t) = \sum_{ w \in \DD_n} t^{\ex(w)}.
\]
Observe that $\ex(w^{-1}) = \# \{ i \mid w(i) < i \}$. Using Stanley's result, one may  calculate that 
\[
h_P(\cS;u,v)  : = h_{[\emptyset,P]}(\cS;u,v) = \sum_{ w \in \Sym_n } u^{\ex(w)} v^{\ex(w^{-1})}.
\]
More generally, a face $F$ of $\cS$ containing $k$ vertices and contained in a face $Q$ of $P$, corresponds to a chain of faces $\emptyset = S_0 \subsetneq S_1 \subsetneq \cdots  \subsetneq S_k  \subseteq S_{k + 1} = Q$. 
Let $r_i = \dim S_i - \dim S_{i - 1}$, for $i = 1,\ldots,k+1$. Then Athanasiadis and Savvidou proved the following in \cite[Example~5.2]{ASSymmetric}:
\[
l_P(\cS,Q,F;t) := l_{[\emptyset,P]}(\cS,Q,F;t) =\prod_{i = 1}^{k} A_{r_i}(t) \cdot  \sum_{ w \in \DD_{r_{k+ 1}}}  t^{\ex(w)}, 
\]
where $ \sum_{ w \in \DD_{r_{k+ 1}}}  t^{\ex(w)}:= 1$ when $r_{k + 1} = 0$. 
Using this result, a short computation gives
\[
h_P(\cS,Q,F;u,v) := h_{[\emptyset,P]}(\cS,Q,F;u,v) =\prod_{i = 1}^{k} v^{r_i - 1} A_{r_i}(uv^{-1}) \cdot  \sum_{ w \in \Sym_{r_{k+ 1}}}   u^{\ex(w)} v^{\ex(w^{-1})}.
\]
\end{example}

\begin{remark}\label{r:generalizations}
We briefly mention that there exists a natural generalization of the mixed $h$-polynomial to an invariant of three variables with similar properties. 
We will discuss the Ehrhart analogue of this invariant in detail in Section~\ref{s:rlmhstar}.
Let  $\tau: \Omega  \rightarrow  \Gamma$ and $\sigma: \Gamma \rightarrow B$ be strong formal subdivisions  between locally Eulerian posets, 
and assume that $\Omega$,$\Gamma$ are lower Eulerian and $B$ is Eulerian. 
Then we define
\[
l_B(\Omega,\Gamma;u,v) :=  \sum_{ y \in \Gamma } l_{[\sigma(y),\hat{1}_B]}(\Gamma_{\ge y};uv)v^{\rk(\Omega_y)} l_{[\hat{0}_\Gamma,y]}(\Omega_y;uv^{-1}), 
\]
\[
h_B(\Omega,\Gamma;u,v,w) :=  \sum_{x \in B} w^{\rk(\Omega_x)} l_{[\hat{0}_B,x]}( \Omega_x,\Gamma_x;u,v)g([x,\hat{1}_B]; uvw^2). 
\]
Observe that if $\sigma$ is the identity function, then it follows from Example~\ref{e:lidentity} and Definition~\ref{d:mixedpoly} that 
$l_B(\Omega,B;u,v) = v^{\rk(\Omega)}l_B(\Omega;uv^{-1})$ and  $h_B(\Omega,B;u,v,w) =  h_B(\Omega;uw,vw)$.
\end{remark}

\section{The geometry of the local $h$-polynomial}\label{s:geometric}

In this section, we consider polyhedral subdivisions of polytopes, as in Lemma~\ref{l:polyhedralsubdivision}, and give a geometric interpretation of the corresponding local $h$-polynomials 
in terms of intersection cohomology.  This interpretation is also studied by de Cataldo, Mustata, and Migliorini \cite{CMM} in the context of morphisms between toric varieties whose domain is simplicial.  We continue with the notation of the previous sections.

Let $\cS$ be a rational polyhedral subdivision of a polytope $P$.  That is, $\cS$ is a polyhedral subdivision, all of whose vertices have rational coordinates.
Recall from Lemma \ref{l:polyhedralsubdivision} that we have a corresponding strong formal subdivision $\sigma: \cS \rightarrow [\emptyset, P]$ of the face poset of $P$ by the face poset of $\cS$, 
where for every cell $F$ in $\cS$, $\sigma(F)$ denotes the smallest face of $P$ containing $F$. For any face $Q$ of $P$ containing $\sigma(F)$, we may consider the corresponding  local $h$-polynomial:
\[
l_P(\cS,Q,F;t) := l_{[\emptyset, P]}(\cS,Q,F;t) =  l_{[\sigma(F),Q]}(\lk_{\cS|_{Q}}(F);t).
\]
 
A natural class of polyhedral subdivisions are the regular subdivisions.  They are induced by a height function $\omega:A\rightarrow \R$, for a set of points $A\subset P$ containing the vertices of $P$. 
 The cells of the subdivision are the projections 
of the bounded faces of the convex hull of $\operatorname{UH}=\{ (u, \lambda) \mid \lambda \geq\omega(u) \} \in \R^n \times \R$.  A subdivision is said to be \define{regular} if it is induced by some height function.  For more details, see \cite{Triangulations,GKZ}.  A rational, regular subdivision, possibly after replacing $P$ by an integer dilation $nP$, is induced by a height function $\omega:P\cap \Z^n\rightarrow \R$.  Moreover, in this case we may ensure that $\omega$ takes integer values at lattice points. 

The main result of this section is the following:

\begin{theorem}\label{t:localcoef}
Let $\cS$ be a rational polyhedral subdivision of a polytope $P$. Then for every cell $F$ of $\cS$ and face $Q$ of $P$ containing $F$, the  local $h$-polynomial $l_P(\cS,Q,F;t)$ has non-negative coefficients. Moreover, if $\cS$ is a regular subdivision,  
then the coefficients of  $l_P(\cS,Q,F;t)$ are symmetric and unimodal. 
\end{theorem}

Our proof proceeds by giving a geometric interpretation of the  local $h$-polynomial.

We refer the reader to \cite{FulIntroduction} for the relevant background on toric varieties.  
If $P$ is a polytope in a vector space   $N_\R$, then let $\tau_P$ denote the cone over $P \times \{ 1 \}$ in $N_\R \times \R$.  
Let $\Sigma$ denote the fan refinement of $\tau_P$ induced by $\cS$, with cones given by the cones over $F \times \{ 1\}$, where $F$ is a cell of $\cS$ (the empty cell of $\cS$ corresponds to $\{0\}$). 
Let $X(\Sigma)$ and $X(P)$ denote the toric varieties corresponding to $\Sigma$ and $\tau_P$ respectively. For a face $Q\subseteq P$ (resp. cell $F$ of $\Sigma$), let $V_Q$ (resp. $V_F$) be the closed subvariety of $X(P)$ (resp. $X(\Sigma)$) corresponding to $\tau_Q$ (resp. $\tau_F$).
There is an inclusion reversing correspondence between the cells $F$ of $\cS$ and closed torus-invariant subvarieties $V_F$ of $X(\Sigma)$.  We have a corresponding proper, birational morphism of toric varieties 
\[
\pi: X(\Sigma) \rightarrow X(P). 
\]
Moreover, $\pi$ is projective precisely when $\cS$ is a regular polyhedral subdivision of $P$. In particular, by restricting $\pi$ to $V_F$, we have a proper map
\[
\pi_F: V_F \rightarrow V_{\sigma(F)} \subseteq X(P)
\]
which is projective when $\cS$ is a regular polyhedral subdivision.

We recall the 
following combinatorial interpretation of the Betti polynomials of  intersection cohomology \cite[Thm~6.2]{DLWeights}.
Using the fact that, by Example~\ref{e:hEulerian},  $g([Q,P];t) = h([Q,P];t)$,
\begin{equation}\label{e:gintersect}
g([Q,P];t)  = \sum_{i} \dim IH^{2i}(V_Q;\C) t^i.   
\end{equation}
For $V_F$, we have
\begin{equation}\label{e:intersect}
h(\lk_{\cS}(F);t)  = \sum_{i} \dim IH^{2i}(V_F;\C) t^i.   
\end{equation}

The decomposition theorem \cite{BBD} applied to $\pi_F$ (see \cite{CM,CMBulletin,CMM} for details) states that 
\begin{eqnarray} \label{mdecomp}
(\pi_F)_*IC_{V_F}&\cong& \bigoplus_{Q \supseteq 
\sigma(F)}  \bigoplus_i IC_{V_Q}(L_{i,Q,F})[-i]
\end{eqnarray}
for some constant local systems $L_{i,Q,F}$ on $V_Q$.  
Write 
 \[ \phi_Q(F,i)=\rk L_{2i,Q,F}\]
where each $\phi_Q(F,i)$ 
is non-negative.
Pushing forward to a point, we get the following formula
\begin{equation}\label{e:hformula}
h(\lk_{\cS}(F);t)  = \sum_{\sigma(F) \subseteq Q \subseteq P}  \phi_P(\cS,Q,F;t) g([Q,P];t),
\end{equation}
where $ \phi_P(\cS,Q,F;t)=\sum_i \phi_Q(F,\dim Q-\dim F-i)t^i$. 

\begin{lemma}\label{l:philocal} The polynomial $\phi_P(\cS,Q,F;t)$ only depends on $(Q,\cS|_Q,F)$ in the following sense:
$\phi_P(\cS,Q,F;t)=\phi_Q(\cS|_Q,Q,F;t)$ 
\end{lemma}

\begin{proof}
By the semi-simplicity theorem as described in \cite{CM}, the local systems have a description in terms of usual and perverse cohomology sheaves as
\[L_{i,Q,F}=i_Q^*\cH^{-(n-\dim Q)}(\cHp^i((\pi_F)_*IC_{V_F}))\]
where $i_Q:V_Q\hookrightarrow X(P)$. 
By naturality of this description, the conclusion follows.
\end{proof}

\begin{lemma} $\phi_P$ obeys the following:
\begin{enumerate}
\item \label{i:pd} The coefficients of $\phi_P(\cS,Q,F;t)$ are symmetric: 
\[\phi_P(\cS,Q,F;t) = t^{\dim Q - \dim F}\phi_P(\cS,Q,F;t^{-1});\]
\item \label{i:hardlefschetz} If $\cS$ is a regular polyhedral subdivision, then the coefficients of $\phi_P(\cS,Q,F;t)$ are unimodal.
\end{enumerate}
\end{lemma}

\begin{proof}
The first statement is  Poincar\'{e} duality for intersection cohomology.  The second is the relative hard Lefschetz theorem \cite[Theorem~1.6.3]{CMBulletin}.
\end{proof}

\begin{lemma} We have the identity
\[
l_P(\cS,Q,F;t) = \phi_P(\cS,Q,F;t).
\]
\end{lemma}
\begin{proof}
 For any 
face $Q \subseteq P$ containing $F$, by \eqref{e:hformula} and Lemma~\ref{l:philocal}, we have
\[
h(\lk_{\cS|_{Q}}(F);t)  = \sum_{\sigma(F) \subseteq Q' \subseteq Q} \phi_P(\cS,Q',F;t) g([Q',Q];t). 
\] 
On the other hand, by Remark~\ref{r:reverseh}, 
\[
h(\lk_{\cS|_{Q}}(F);t)  = \sum_{\sigma(F) \subseteq Q' \subseteq Q} l_P(\cS,Q',F;t) g([Q',Q];t). 
\]
As in Remark~\ref{r:reverseh}, it follows from Theorem~\ref{t:inverse} that both $l_P(\cS,Q,F;t)$ and $ \phi_P(\cS,Q,F;t)$ are determined by 
the equation in Definition~\ref{d:localpoly}. 
\end{proof}

\begin{remark}\label{r:conject}
When $F = Q = \emptyset$,  Stanley gave a geometric interpretation of the local $h$-polynomial $l_P(\cS;t)$ in terms of the decomposition theorem for intersection cohomology \cite[Theorem~5.2, Theorem~7.9]{StaSubdivisions} and proved the above theorem in this case.
When $\cS$ is a triangulation, the non-negativity of $l_P(\cS,Q,F;t)$ was conjectured by Nill and Schepers in \cite{NSCombinatorial}, 
and proved by Athanasiadis and Savvidou for triangulations of simplices in \cite[Theorem~5.4]{ASSymmetric}. For regular triangulations of simplices, 
the fact that the coefficients of  $l_P(\cS,Q,F;t)$ are symmetric and unimodal answers affirmatively a question of Athanasiadis \cite[Question~2.16]{AthSurvey}.
\end{remark}

\begin{remark}\label{r:irrational}
The condition in Theorem~\ref{t:localcoef} that the polyhedral subdivision is rational is almost certainly unnecessary.  In fact, when $\cS$ is a (not necessarily rational) polyhedral subdivision of a polytope $P$, the identification of 
the coefficients of the  local $h$-polynomial with
terms in the decomposition theorem can be phrased in terms of the intersection cohomology on fans, 
developed by Barthel, Brasselet, Fieseler and Kaup \cite{BBFK}. One may show in this way that the coefficients of $l_P(\cS,Q,F;t)$ are non-negative. However, at this time, an analogue of the 
relative hard Lefschetz theorem is a topic of current research, and hence the unimodality statement above can not currently be generalized. We note that 
the hard Lefschetz theorem for intersection cohomology on fans due to Karu \cite{KarHard} implies that if $\cS$ is regular then 
 $l_P(\cS,\sigma(F),F;t) = h(\lk_{\cS|_{\sigma(F)}}(F);t)$ 
has unimodal coefficients. 
Also, if $\cS$ is a regular triangulation of a simplex, then $\cS$ can be deformed to a rational triangulation without changing the corresponding strong formal subdivision of posets, and hence $l_P(\cS,Q,F;t)$ has unimodal coefficients.  
\end{remark}

For any face $Q$ of $P$ containing $\sigma(F)$, we may consider the corresponding  mixed $h$-polynomial:
\[
h_P(\cS,Q,F;u,v) := h_{[\emptyset, P]}(\cS,Q,F;u,v) =  h_{[\sigma(F),Q]}(\lk_{\cS|_{Q}}(F);u,v).
\]

\begin{corollary}\label{c:refine}
Let $\cS$ be a rational polyhedral subdivision of a polytope $P$. Then for every cell $F$ of $\cS$ contained in a face $Q$ of $P$, the  mixed $h$-polynomial $h_P(\cS,Q,F;u,v)$ has non-negative integer coefficients. Moreover, 
\[
h_P(\cS,Q,F;u,v) \ge v^{\dim Q - \dim F} l_P(\cS,Q,F;uv^{-1}). 
\]
If $\cS$ is a regular subdivision, and we write $h_P(\cS,Q,F;u,v) = \sum_{i,j \ge 0} h_{i,j} u^i v^j$ for some non-negative integers $h_{i,j}$, then the sequences $\{ h_{i,k-i} \}_{i = 0,\ldots,k}$ are
symmetric and unimodal.  
\end{corollary}
\begin{proof}
By Definition~\ref{d:mixedpoly},
\[
h_P(\cS,Q,F;u,v)  =  \sum_{\sigma(F) \le Q' \le Q} v^{\dim Q' - \dim F} l_{P}(\cS,Q',F;uv^{-1}) g([Q',Q];uv).
\]
It follows from Theorem~\ref{t:localcoef} 
and \eqref{e:gintersect} that the coefficients of each polynomial on the right hand side of the above equation are non-negative integers. The first statement follows. If $\cS$ is regular, then by Theorem~\ref{t:localcoef} the coefficients of $v^{\dim Q' - \dim F}l_{P}(\cS,Q',F;uv^{-1})$ are symmetric (with respect to the appropriate degree) and unimodal,
and the second statement follows.  
\end{proof}

\begin{remark}
Specializing Corollary~\ref{c:refine} by setting $Q = P$ and $v = 1$ gives 
\[
h(\lk_\cS(F);u) \ge  l_{[\sigma(F),P]}(\lk_\cS(F);u). 
\]
In the case when $F = \emptyset$ and $\cS$ is a triangulation of a simplex, Stanley proved the stronger statement that the coefficients of $h(\cS;u) -  l_P(\cS;u)$ form an $O$-sequence \cite[Corollary~4.8]{StaSubdivisions}. 
\end{remark}

\begin{remark} 
It would be interesting to have an explicit geometric description of $h_P(\cS,Q,F;u,v)$. 
From a different geometric perspective to the one in this section, we remark that when $\cS$ is a unimodular triangulation of a lattice polytope $P$, then the coefficients of $h_P(\cS;u,v)$ may be interpreted as mixed Hodge numbers (see Example~\ref{e:unimodular} and Remark~\ref{r:geometric}). 
\end{remark}

\begin{corollary}\label{c:geomrefine}
Let $\cS$ be a rational polyhedral subdivision of a polytope $P$, and let $\cS'$ be a rational polyhedral subdivision refining $\cS$. Then 
\begin{equation}\label{e:h}
h_P(\cS';u,v) \ge   h_P(\cS;u,v), 
\end{equation}
\begin{equation}\label{e:l}
l_P(\cS';t) \ge  l_P(\cS;t).
\end{equation}
Moreover, if $\cS'$ and $\cS$ are regular subdivisions, then $l_P(\cS';t) - l_P(\cS;t)$ is a polynomial with non-negative, symmetric, unimodal coefficients.
\end{corollary}
\begin{proof}
By \eqref{i:rh8'} in Theorem~\ref{t:refineprop}, 
\[
h_P(\cS';u,v) - h_P(\cS;u,v) = \sum_{\emptyset \ne F \in \cS} h_P(\cS,P,F;u,v) v^{\dim F + 1} l_F(\cS'|_F;uv^{-1}),
\]
\[
l_P(\cS';t) - l_P(\cS;t) = \sum_{\emptyset \ne F \in \cS} l_P(\cS,P,F;t)  l_F(\cS'|_F;t). 
\]
The first statement follows from Theorem~\ref{t:localcoef} and Corollary~\ref{c:refine}, since the coefficients of each polynomial on the right hand side of 
the above equations are non-negative integers. If $\cS'$ and $\cS$ are regular subdivisions, then by Theorem~\ref{t:localcoef},  the coefficients of each polynomial on the right hand side of 
the second equation have symmetric (with respect to the appropriate degree), unimodal coefficients and the second statement follows. 
\end{proof}

\begin{example} 
Let $\cS$ be a polyhedral subdivision of a polytope $P$, and fix a cell $F$ in $\cS$ contained in a face $Q$ of $P$. Then Example~\ref{e:rhsmallcases} gives an explicit description of $h_P(\cS,Q,F;u,v)$ when $\dim Q - \dim F \le 3$.
In particular, 
\[
h_P(\cS;u,v) =  \left\{\begin{array}{cl} 
1 & \text{if } \dim P =  0,  \\ 
1 + \beta uv & \text{if } \dim P = 1,  \\ 
1 + (\mu + \nu - 3)uv + \beta  uv(u + v)  & \text{if } \dim P =  2,  \end{array}\right.
\]
where $\beta$ is the number of interior vertices of $\cS$, $\mu$ is the number of vertices of $\cS$ contained in the interior of an edge of $P$, and $\nu$ is the number of vertices of $P$.

\end{example}

\begin{example}\label{e:triangulation}
Let $\cS$ be a triangulation of a simplex $P$. 
For any non-negative integers $i,j$, let $f_{i,j} = \# \{ F \in \cS \mid \dim F + 1 = i, \dim \sigma(F) - \dim F = j \}$. 
It is natural to ask what are the possible set of numbers $\{ f_{i,j} \}_{i,j \ge 0}$ for any triangulation of a simplex. 
After possibly deforming $P$ and $\cS$, we may assume, that the triangulation is rational. 
By Lemma~\ref{l:simplicialsub2}, 
\[
h_P(\cS;u,v) = \sum_{i,j\ge 0} f_{i,j}    u^{i}  (1 - u)^{\dim P + 1 - i - j} (v - u)^{j}, 
\]
and $h_P(\cS;u,v)$ determines and is determined by the numbers $\{ f_{i,j} \}_{i,j \ge 0}$. 
 If we write $h_P(\cS;u,v) = \sum_{i,j \ge 0} h_{i,j} u^i v^j$, then by Corollary~\ref{c:refine}, the coefficients $h_{i,j}$ are non-negative integers, and, if $\cS$ is a regular subdivision, 
 then the sequences $\{ h_{i,k-i} \}_{i = 0,\ldots,k}$ are
symmetric and unimodal.  This implies non-trivial relations
between the numbers $\{ f_{i,j} \}_{i,j \ge 0}$. 
\end{example}

\section{Ehrhart theory of lattice polytopes} \label{s:ehrhart}

\subsection{Review of Ehrhart theory}

In this section, we recall some basic combinatorial facts about lattice points in lattice polytopes. We continue with the notation of the previous sections.  We recommend \cite{BRComputing} as an introduction to Ehrhart theory.

 Let $P$ be a non-empty lattice polytope in a lattice $\Z^d$. 
 Consider the function $f_P(m)=\#(mP\cap \Z^d)$, for $m\in \Z_{> 0}$.  By Ehrhart's theorem  \cite[Section~3.3]{BRComputing}, $f_P(m)$ is a polynomial of degree $\dim P$, called
 the \define{Ehrhart polynomial} of $P$, and $f_P(0) = 1$.  
 Ehrhart recipricocity \cite{Macdonaldpolynomials} states that  for $m\in\Z_{>0}$
\[f_P(-m)=(-1)^{\dim P} f_{P^\circ}(m)\]
where 
 $f_{P^\circ}(m) =\#(\Int(mP)\cap \Z^d)$. 
 It follows that we can write
\begin{equation}\label{e:Ehrhartpoly}
f_P(m)=f_0(P)+f_1(P)m+\dots+f_{\dim P}(P)m^{\dim P}, 
\end{equation}
where $f_i(P)\in \Z$, and
\[
\Ehr_P(t)\equiv 1 + \sum_{m > 0} f_P(m) t^m = \frac{h^*(P;t)}{(1 - t)^{\dim P + 1}}, 
\]
where  $h^*(P;t)$ is a polynomial of degree at most $\dim P$ called the \define{$h^*$-polynomial} of $P$ (see, for example, \cite[Section~3.3]{BRComputing}). 
A non-trivial theorem of Stanley \cite{StaDecomp} states that the coefficients of the 
$h^*$-polynomial are non-negative integers. Throughout, we will use the notation\[
h^*(P; t) =  \sum_{i = 0}^{\dim P} h^*_i t^i. 
\]

\begin{example}\label{e:hstarlinear}
It follows from the definition that $h^*_0 = 1$  
and 
$h^*_1 = \#(mP\cap \Z^d) - \dim P  - 1$. 
It follows from Ehrhart reciprocity 
that $h^*_{\dim P} = \#(\Int(P)\cap \Z^d)$. 
The sum of the coefficients $h^*(P;1)$ is equal to the \emph{normalized volume} of $P$, i.e. after possibly replacing 
$\Z^d$ with a smaller lattice so that we may assume that $\dim P = d$, 
$h^*(P;1)$ is $d!$ times the Euclidean volume of $P$. 
\end{example}

We also define 
\[
\Ehr_{P^\circ}(t)\equiv  \sum_{m > 0} f_{P^\circ}(m) t^m.  
\]
If $P$ is empty, then we set  $f_P(m) \equiv 0$ and  $\Ehr_{P}(t) = \Ehr_{P^\circ}(t) = h^*(P;t) = 1$. One may verify that Ehrhart reciprocity is equivalent to the following (see, for example, \cite[(22)]{StaGeneralized}):
\begin{equation}\label{e:power}
\Ehr_{P}(t) = (-1)^{\dim P + 1}\Ehr_{P^\circ}(t^{-1})
\end{equation}

The following definition was introduced by Stanley in \cite[Example~7.13]{StaSubdivisions}, generalizing the definition of Betke and McMullen in the case of a simplex \cite{BMLattice}.  It was independently introduced by Borisov and Mavlyutov in \cite{BMString} as $\tilde{S}(t)$. 

\begin{definition}\label{d:localstar}
The \define{local $h^*$-polynomial} $\lc(P;t)$ of $P$ is 
\begin{equation*}
\lc(P; t) =  \sum_{Q \subseteq P} (-1)^{\dim P - \dim Q}h^*(Q;t)g([Q,P]^*;t). 
\end{equation*}
\end{definition}
Note that if $P$ is empty, then $\lc(P; t) = 1$.  
Throughout, we will use the notation
\[
\lc(P; t) =  \sum_{i = 0}^{\dim P} \lc_i t^i.
\]

We record some known properties of the local $h^*$-polynomial in the following lemma, that are analogous to properties of the local $h$-polynomial from Section~\ref{s:local}. 

\begin{lemma}\label{l:localprop}
Let $P$ be a non-empty lattice polytope. Then the local $h^*$-polynomial $\lc(P; t)$ satisfies the following properties:

\begin{enumerate}

\item\label{j:1} The local $h^*$-polynomial is symmetric in the sense that $\lc(P; t) = t^{\dim P + 1}\lc(P; t^{-1})$. 

\item\label{j:2} The $h^*$-polynomial of a polytope can be recovered from the local $h^*$-polynomials of its faces via 
\[
h^*(P;t) =  \sum_{Q \subseteq P} \lc(Q; t)g([Q,P];t).
\]

\item\label{j:3} The coefficients $\lc_i$ of $\lc(P;t)$ are non-negative integers. 

\item\label{j:4} We have $\lc_0 = 0$ and $\lc_1 = \lc_{\dim P} = h^*_{\dim P} =  \#(\Int(P)\cap \Z^d)$. 
\end{enumerate}

\end{lemma}
\begin{proof}
Properties~\eqref{j:1} and  Property~\eqref{j:2} are established in \cite[Example~7.13]{StaSubdivisions}.  We will prove them in 
 Lemma~\ref{l:reciprocity} and Remark~\ref{r:heartdef} below respectively. 


Property~\eqref{j:3} was conjectured by Stanley in \cite[Example~7.13]{StaSubdivisions}, and proved 
independently by Borisov and Mavlyutov \cite[Proposition~5.1]{BMString} and Karu \cite[Theorem~1.1]{KarEhrhart}.  

Properties \eqref{j:4} is proven in Example~4.7 of \cite{BNCombinatorial}. Note that by definition, for $Q \ne P$, the degree of $g([Q,P]^*;t)$ is strictly bounded by $(\dim P - \dim Q)/2$. It follows 
that $\lc(P;t)$ has degree at most $\dim P$ with  $\lc_{\dim P} = h^*_{\dim P}$. The result now follows from Property~\eqref{j:1} and Example~\ref{e:hstarlinear}.
\end{proof}

\begin{example} \label{e:hstarsimplex}
If $P \subseteq \Z^d$ is a simplex and $\Pi$ is the parallelepiped spanned by the vertices of $P \times \{ 1 \}$ in $\R^d \times \R$, then 
\[
\lc(P;t) = \sum_{ w \in \IInt(\Pi) \cap \Z^{d + 1}} t^{\psi(w)}, \; h^*(P;t) = \sum_{Q \subseteq P} \lc(Q;t)
\] 
where $\psi: \Z^{d + 1} \rightarrow \Z$ denotes projection onto the last co-ordinate. 
This was proven as Proposition~4.6 of \cite{BNCombinatorial} and was the original definition of Betke and McMullen  in \cite{BMLattice} of the local $h^*$-polynomial of a  simplex.
\end{example}

We introduce the following analogue of Definition~\ref{d:mixedpoly} (cf. Example~\ref{e:unimodular}). 

\begin{definition}\label{d:mixedstar}
Let $P$ be a non-empty lattice polytope. The \define{mixed $h^*$-polynomial} is defined by
\[
h^*(P;u,v) = \sum_{Q \subseteq P} v^{\dim Q + 1}\lc(Q; uv^{-1})  g([Q,P];uv). 
\] 
\end{definition}

We see from Lemma~\ref{l:localprop} and \eqref{e:gintersect} that the mixed $h^*$-polynomial has non-negative integer coefficients,  is invariant under the interchange of $u$ and $v$, and specializes to the $h^*$-polynomial by 
setting $v = 1$. We will study generalizations of this polynomial  in subsequent sections. For the moment, we have the following analogue of \eqref{i:rh6'} in Theorem~\ref{t:refineprop} (cf. Example~\ref{e:unimodular}). 

\begin{remark}\label{r:charmixed}
All terms in $h^*(P;u,v)$ have combined degree  in $u$ and $v$ at most $\dim P + 1$, and 
the terms of combined degree $\dim P + 1$ equal $v^{\dim P + 1}\lc(P;uv^{-1})$. 
\end{remark}

We introduce the following convenient way to visualize this polynomial. 

\begin{remark}\label{r:hstardiamond}
Assuming $P$ is non-empty, we may write 
\[
h^*(P;u,v) = 1 + uv \sum_{ 0 \le p,q \le \dim P - 1} h^*_{p,q} u^p v^q.
\]
 for some non-negative integers $h^*_{p,q}$.  We may visualize these coefficients and the coefficients of the local $h^*$-polynomial
in diamonds, which we call the \define{ $h^*$-diamond} and \define{local $h^*$-diamond} of $P$, by placing $h^*_{p,q}$ at point $(q - p, p + q)$ and $\lc_i$ at point  $(i-1,n- 1-  i)$ in $\Z^2$ respectively (see Figure~\ref{f:trivial} below). Observe that the  $h^*$-diamond is symmetric about the vertical axis, identically zero above the middle horizontal strip, and encodes the coefficients of the local $h^*$-polynomial along the middle horizontal strip.  Also, we recover the coefficients of $h^*(P;u)$ by summing the coefficients of the  $h^*$-diamond along a fixed choice of diagonal i.e. $h^*_{i + 1} = \sum_{p + q = i} h^*_{p,q}$.
It follows from Theorem~\ref{t:lower} below that each horizontal strip of the diamonds satisfies the following lower bound theorem: its first entry is a lower bound for the other entries i.e. 
$h^*_{k,0} \le h^*_{k-i,i}$ for $i = 0,\ldots, k$ and  $\lc_1 \le \lc_i$ for $i = 1,\ldots,\dim P$. 

\end{remark}

\begin{figure}[htb]
\begin{center}
\[
\begin{array}{p{5cm}p{5cm}}
\xy
(9,-12)*{}="A"; (-27,19)*{}="B"; (9,49)*{}="C"; (46,19)*{}="D";
"A"; "B" **\dir{-};
"A"; "D" **\dir{-};
"C"; "B" **\dir{-};
"C"; "D" **\dir{-};
(-9.5,19.5)*{
\xymatrix@=0.25em{ 
 &  &  & 0 & &                                              \\
  &  & 0 & & 0 &  &                                \\
   &  0  &  & 0& & 0 &                    \\
  \lc_{1}  &   & \lc_{2}  &  & \lc_{3} &  & \lc_{4}      \\
      &  h^*_{2,0}  &  & h^*_{1,1} & & h^*_{0,2} &                \\
        &  & h^*_{1,0} & & h^*_{0,1} &  &                         \\
        &  &  & h^*_{0,0} & &                                \\
 }
};
(12,-19)*{\hbox{\small $h^*$-diamond of $P$}};
\endxy
&
\xy
(9,-12)*{}="A"; (-27,19)*{}="B"; (9,49)*{}="C"; (46,19)*{}="D";
"A"; "B" **\dir{-};
"A"; "D" **\dir{-};
"C"; "B" **\dir{-};
"C"; "D" **\dir{-};
(-6,20.5)*{
\xymatrix@=0.5em{ 
 &  &  & 0 & &                                              \\
  &  & 0 & & 0 &  &                                \\
   &  0 &  &0 & & 0 &                    \\
     \lc_{1}  &   & \lc_{2}  &  & \lc_{3} &  & \lc_{4}      \\
      & 0  &  &0 & &0 &                \\
        &  & 0 & & 0 &  &                         \\
        &  &  & 0 & &                                \\
 }
};
(13,-17)*{\hbox{\small local $h^*$-diamond of $P$}};

\endxy
\end{array} 
\]
\end{center}

\caption{ $h^*$-diamond and  local $h^*$-diamond of $P$ when $\dim P = 4$} 
\label{f:trivial}
\end{figure}
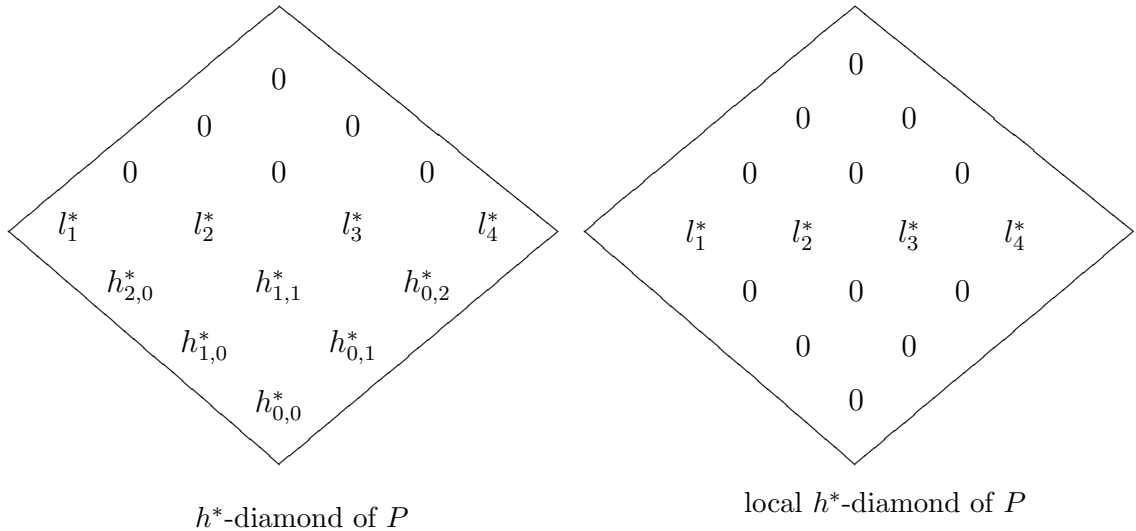

Recall that the face poset $[\emptyset, P]$ of $P$ is an Eulerian poset with rank function  $\rho(Q)=\dim Q+1$. 
More generally, if $\cS$ is a  lattice polyhedral subdivision of  $P$, then we identify $\cS$ with its lower Eulerian poset of cells, with rank function $\rho(F)=\dim F+1$. 
Recall that $R=\Z[t^{\pm 1/2}]$ with kernel $\kappa(Q,Q')=q^{\rho(Q')-\rho(Q)}$, where $q=t^{1/2}-t^{-1/2}$, and acceptability operators $\gamma_P := \gamma_{[\emptyset, P]}$ and $\gamma_\cS$ given by Lemma~\ref{l:gpoly}. 
Recall that $R_{uv}=\Z[u^{\pm 1/2},v^{\pm 1/2}]$, with the corresponding kernel and acceptability operators obtained by replacing $t$ with $uv$. 
We will write $R^P := R^{[\emptyset, P]}$ and $R_{uv}^P := R_{uv}^{[\emptyset, P]}$. 

\begin{definition}\label{d:ehrdef}
Let $\cS$ be a lattice polyhedral subdivision of  a lattice polytope $P$. Define elements $\eta^\heartsuit_\cS,\lambda^\heartsuit_\cS \in R^\cS$ 
and $\tilde{\eta}^\heartsuit_\cS \in R_{uv}^\cS$   by:
 \[ \eta^\heartsuit_\cS(F) = t^{-(\dim F + 1)/2}h^*(F;t) = \bar{q}^{\dim F + 1} \Ehr_F(t), \]
 \[ \lambda^\heartsuit_\cS(F) = t^{-(\dim F + 1)/2}l^*(F;t), \]
 \[ \tilde{\eta}^\heartsuit_\cS(F) = (uv)^{-(\dim F + 1)/2}h^*(F;u,v),\]
 respectively.
 When $\cS$ is the trivial subdivision of $P$, then we write $\eta^\heartsuit_P := \eta^\heartsuit_\cS$, $\lambda^\heartsuit_P := \lambda^\heartsuit_\cS$ 
 and $\tilde{\eta}^\heartsuit_P := \tilde{\eta}^\heartsuit_\cS$.
\end{definition}

\begin{remark}\label{r:heartdef}
By Theorem~\ref{t:inverse} and Definition~\ref{d:ehrdef}, $\eta^\heartsuit_\cS = \lambda^\heartsuit_\cS * \gamma_\cS$ and 
$\tilde{\eta}^\heartsuit_\cS = \lambda^\heartsuit_\cS|_{t = uv^{-1}} * \gamma_\cS|_{t = uv}$. 
\end{remark}

\begin{remark}\label{r:independent}
Note that $\eta^\heartsuit_\cS(F) = \eta^\heartsuit_F(F)$, $\lambda^\heartsuit_\cS(F) = \lambda^\heartsuit_F(F)$ and $\tilde{\eta}^\heartsuit_\cS(F) = \tilde{\eta}^\heartsuit_F(F)$ only depend on $F$ and not on $\cS$ 
by definition. 
\end{remark}

As observed in \cite[p. 201]{StaGeneralized}, $\eta^\heartsuit_\cS$ is acceptable. In fact, we have the following slightly stronger statement. 

\begin{lemma}\label{l:reciprocity}
The following statements are equivalent:
\begin{enumerate}

\item\label{i:first} Ehrhart reciprocity,

\item\label{i:third} $\lc(P; t) = t^{\dim P + 1}\lc(P; t^{-1})$  for all lattice polytopes $P$,

\item\label{i:fourth} $h^*(P;u,v) = h^*(P;v,u)$   for all lattice polytopes $P$,




\item\label{i:second'} $\eta^\heartsuit_\cS \in \mathcal{A}(\cS,\kappa_\cS)$ 
for all lattice polyhedral subdivisions $\cS$,

\item\label{i:fifth'} $\lambda^\heartsuit_\cS \in S(R)$ for all lattice polyhedral subdivisions $\cS$,


\item\label{i:seventh}  $\tilde{\eta}^\heartsuit_\cS \in \mathcal{A}(\cS,\kappa_\cS;R_{uv})$ 
for all lattice polyhedral subdivisions $\cS$.

\end{enumerate}

\end{lemma}
\begin{proof}
By definition, \eqref{i:third} is equivalent to \eqref{i:fifth'}. 
By Remark~\ref{r:heartdef} and Lemma~\ref{l:symacc}, \eqref{i:second'} and  \eqref{i:fifth'} are equivalent. 
The equivalence of \eqref{i:third} and \eqref{i:fourth} follows from Definition~\ref{d:mixedstar} and Remark~\ref{r:charmixed}. Evaluating at $v = 1$, we see that 
\eqref{i:seventh} implies \eqref{i:second'}. On the other hand, \eqref{i:fifth'} implies that $\lambda^\heartsuit_\cS|_{t = uv^{-1}} \in S(R_{uv})$, and hence, by Lemma~\ref{l:symacc} and Remark~\ref{r:heartdef},
  $\tilde{\eta}^\heartsuit_\cS =  \lambda^\heartsuit_\cS|_{t = uv^{-1}} * \gamma_\cS|_{t = uv}$ 
 is acceptable. 
We deduce that all conditions except \eqref{i:first} are equivalent. 


By definition, $\eta^\heartsuit_P$ is acceptable if and only if $\eta^\heartsuit_P = \bar{\eta^\heartsuit_P} * \bar{\kappa}$, if and only if for all faces $Q$ of $P$, 
\[
\Ehr_Q(t) = \sum_{Q' \subseteq Q} (-1)^{\rho(Q')} \Ehr_{Q'}(t^{-1}).
\]
Since every lattice point in a polytope lies in the relative interior of some face, $\Ehr_Q(t) = \sum_{Q' \subseteq Q} \Ehr_{(Q')^\circ}(t)$. Hence $\eta^\heartsuit_P$ is acceptable if and only if
\[
 \sum_{Q' \subseteq Q} (-1)^{\rho(Q')} \Ehr_{Q'}(t^{-1}) =  \sum_{Q' \subseteq Q} \Ehr_{(Q')^\circ}(t)
\]
for all faces $Q$ of $P$. It follows from  \eqref{e:power} that  Ehrhart reciprocity is equivalent to the fact that $\eta^\heartsuit_P$ is acceptable for all lattice polytopes $P$. The latter condition is equivalent to 
\eqref{i:second'} by Remark~\ref{r:independent}. 
\end{proof}

\subsection{Subdivisions and pushforwards in Ehrhart theory}

Now, we study the properties of the $h^*$-polynomial and local $h^*$-polynomial under subdivision.  These will 
be analogous to Corollary~\ref{c:hrefine} and \eqref{i:rh8'} in Theorem~\ref{t:refineprop}. 

\begin{lemma}\label{l:pushstar} Let $\sigma:\cS\rightarrow [\emptyset,P]$ be the strong formal subdivision induced by a lattice polyhedral subdivision $\cS$ of a lattice polytope $P$.  Then
the following holds:
\begin{enumerate}

\item\label{i:push1}  \[ \sigma_*(\eta^\heartsuit_\cS) = \eta^\heartsuit_P , \]

\item\label{i:push2}  \[ h^*(P;t) = \sum_{ \substack{F \in \cS \\ \sigma(F) = P} } h^*(F; t)  (t - 1)^{\dim P - \dim F},\]

\item\label{i:push3}  \[ h^*(P;t) = \sum_{F \in \cS} \lc(F; t)  h(\lk_\cS(F);t),\]

\item\label{i:push4}  \[ \lc(P; t) = \sum_{F \in \cS} \lc(F;t) l_{[\sigma(F),P]}(\lk_\cS(F);t).\]

\end{enumerate}

Moreover, the statements that each of these conditions hold for all pairs $(P,\cS)$ are equivalent. 
\end{lemma}
\begin{proof}
First note that expanding out the definitions yields the equivalence of \eqref{i:push1}  and \eqref{i:push2}. 
Secondly, by Remark~\ref{r:heartdef}, \eqref{i:push1} is equivalent to  $\sigma_*(\lambda^\heartsuit_\cS * \gamma_\cS)= \eta^\heartsuit_P$. 
Expanding out the definitions using 
\eqref{i:pthird} in Proposition~\ref{p:alternate} shows the 
equivalence of \eqref{i:push1}  and \eqref{i:push3}. Similarly, $\sigma_*(\lambda^\heartsuit_\cS  * \gamma_\cS) * \gamma_P^{-1}= \lambda^\heartsuit_P $
is equivalent to  \eqref{i:push4}. 

Note that  \eqref{i:push2} is equivalent to \[ \Ehr_P(t) =  \sum_{ \substack{F \in \cS \\ \sigma(F) = P} } (-1)^{\dim P - \dim F}\Ehr_F(t).\]
By Ehrhart reciprocity (see \eqref{e:power}), this is equivalent to 
\[
\Ehr_{P^\circ}(t) =  \sum_{ \substack{F \in \cS \\ \sigma(F) = P} } \Ehr_{F^\circ}(t).
\] 
The latter statement holds since every interior lattice point in $P$ 
lies in the relative interior of a unique cell of $\cS$. 
\end{proof}

\begin{remark}
 When $\cS$ is a triangulation, \eqref{i:push2} and  \eqref{i:push4}  in Lemma~\ref{l:pushstar} are due to  Betke and McMullen \cite[Theorem~1]{BMLattice} and Nill and Schepers \cite{NSCombinatorial}
 respectively. In the general case, \eqref{i:push4}  in Lemma~\ref{l:pushstar}  was suggested by Nill. 
\end{remark}

\begin{remark}
Let $\cS$ be a lattice polyhedral subdivision of a lattice polytope $P$, and let $\cS'$ be a lattice polyhedral subdivision refining $\cS$. 
Let $\sigma: \cS' \rightarrow \cS$ denote the corresponding strong formal subdivision. 
Then it follows immediately from 
Lemma~\ref{l:pushstar} that $\sigma_*(\eta^\heartsuit_{\cS'}) = \eta^\heartsuit_\cS$. 
\end{remark}

Recall from Definition~\ref{d:mixedpush} that the mixed pushforward is the left $S(R_{uv})$-module homomorphism
$\tilde{\sigma}_*: \mathcal{A}(\cS,\kappa_\cS; R_{uv})\rightarrow \mathcal{A}([\emptyset,P],\kappa_P; R_{uv})$ 
defined by
\[\tilde{\sigma}_*(e_F*\gamma_\cS|_{t = uv})=\sum_{Q \subseteq P} \lambda(Q,F)|_{t = uv^{-1}} e_Q*\gamma_P|_{t = uv}.\]

\begin{corollary}\label{c:mixedheart}
Let $\sigma:\cS\rightarrow [\emptyset,P]$ be the strong formal subdivision induced by a lattice polyhedral subdivision $\cS$ of a lattice polytope $P$.
Then $\tilde{\sigma}_*(\tilde{\eta}^\heartsuit_\cS) = \tilde{\eta}^\heartsuit_P$. 
\end{corollary}
\begin{proof}
By Remark~\ref{r:heartdef}, $\tilde{\sigma}_*(\tilde{\eta}^\heartsuit_\cS) = \tilde{\eta}^\heartsuit_P$ is equivalent to 
\[
\sum_{F \in \cS}  \lambda^\heartsuit_\cS(F)|_{t = uv^{-1}}  \sum_{Q \subseteq P} \lambda(Q,F)|_{t = uv^{-1}}  e_Q*\gamma_P|_{t = uv} = \sum_{Q \subseteq P}  \lambda^\heartsuit_P(Q)|_{t = uv^{-1}} e_Q*\gamma_P|_{t = uv}. 
\]
The fact that $\sum_{F \in \cS}  \lambda^\heartsuit_\cS(F)|_{t = uv^{-1}}  \lambda(Q,F)|_{t = uv^{-1}}  = \lambda^\heartsuit_P(Q)|_{t = uv^{-1}} $ follows immediately from \eqref{i:push4} of Lemma~\ref{l:pushstar}. 
\end{proof}

\begin{remark}
Let $\cS$ be a lattice polyhedral subdivision of a lattice polytope $P$, and let $\cS'$ be a lattice polyhedral subdivision refining $\cS$. 
Let $\sigma: \cS' \rightarrow \cS$ denote the corresponding strong formal subdivision. 
Then it follows immediately from 
Corollary~\ref{c:mixedheart} that $\tilde{\sigma}_*(\tilde{\eta}^\heartsuit_{\cS'}) = \tilde{\eta}^\heartsuit_\cS$. 
\end{remark}


Following the proof of
\cite{BMLattice} and the conjectural proof of \cite{NSCombinatorial}, we present a proof the non-negativity of the coefficients of the $h^*$-polynomial and local $h^*$-polynomial respectively.
\begin{theorem}\label{t:hlocalh}
 The coefficients of $h^*(P;t)$ and $\lc(P;t)$ are non-negative. 
 \end{theorem}

\begin{proof}
By Example \ref{e:hstarsimplex}, this conclusion is true for simplices.  Because every lattice polytope admits a lattice triangulation, and by the non-negativity of the $h$-polynomial and the local $h$-polynomial (Theorem~\ref{t:localcoef} and \eqref{e:intersect}), the conclusion follows from \eqref{i:push3} and  \eqref{i:push4} in Lemma~\ref{l:pushstar}. 
\end{proof}

\begin{remark}
It follows from the proof of Theorem~\ref{t:hlocalh}, that  
 we have the stronger statement that 
\[
h^*(P;t) \ge h(\cS;t), \; \lc(P;t) \ge l_P(\cS;t),
\]
for any lattice polyhedral subdivision $\cS$ of a lattice polytope $P$. 
\end{remark}

\begin{example}\label{e:unimodular}
Let $\cS$ be a unimodular triangulation  of a lattice polytope $P$, and consider the corresponding strong formal subdivision $\sigma: \cS \rightarrow [\emptyset,P]$. That is, $\cS$ is a lattice polyhedral subdivision such that every non-empty cell $F$ in $\cS$ is isomorphic to a standard simplex. 
If $F$ is a standard simplex, then   $\lc(F;t) = 0$ by Example~\ref{e:hstarsimplex}.
Hence, by \eqref{i:push3} and  \eqref{i:push4} in Lemma~\ref{l:pushstar}, 
\[
h^*(P;t) = h(\cS;t), \; \lc(P;t) = l_P(\cS;t). 
\]
In particular, it follows Definition~\ref{d:mixedpoly} and  Definition~\ref{d:mixedstar}  that in this case  the mixed $h^*$-polynomial is equal to the mixed $h$-polynomial:
\[
h^*(P;u,v) = h_P(\cS;u,v).
\]
Equivalently, in this case, $\sigma_*( e_{\emptyset_\cS} * \gamma_\cS) = \eta^\heartsuit_P  = \lambda^\heartsuit_P * \gamma_\cS$ and $\tilde{\sigma}_*( e_{\emptyset_\cS} * \gamma_\cS|_{t = uv}) = \tilde{\eta}^\heartsuit_P$.
More generally, if $\cS'$ is a unimodular triangulation refining a lattice polyhedral decomposition $\cS$ of $P$, with corresponding strong formal subdivision $\sigma: \cS' \rightarrow \cS$, then 
$\sigma_*( e_{\emptyset_{\cS'}} * \gamma_{\cS'}) = \eta^\heartsuit_\cS  = \lambda^\heartsuit_\cS * \gamma_{\cS}$ and $\tilde{\sigma}_*( e_{\emptyset_{\cS'}} * \gamma_{\cS'}|_{t = uv}) = \tilde{\eta}^\heartsuit_\cS$.

Moreover, if $\cS$ is a regular unimodular triangulation, then it follows from Theorem~\ref{t:localcoef}  that $\lc(P;t)$ has unimodal coefficients \cite[(4)]{NSCombinatorial}. 
Note that not all lattice polytopes admit unimodular subdivisions \cite{Triangulations}, but we can sometimes get unimodality with weaker hypotheses (see Remark~\ref{r:weakerhyp} below). \end{example}

Hibi's lower bound theorem \cite[Theorem~1.1]{HibLower} states that if $\Int(P)\cap \Z^d \ne \emptyset$, then $h^*_1 \le h^*_i$ for $i = 1, \ldots, \dim P -1$. We deduce the following 
counterpart for the local $h^*$-polynomial. 

\begin{theorem}\label{t:lower}
Let $P \subseteq \R^d$ be a non-empty lattice polytope.  Then the local $h^*$-polynomial $l^*(P;t) = l_1^*t + \cdots + l_{\dim P}^*t^{\dim P}$ satisfies $l_1^* = \#(\Int(P) \cap \Z^d) \le l_i^*$ for $1 \le i \le \dim P$. 
\end{theorem}
\begin{proof}
Let  $\cS$ be a regular, lattice polyhedral subdivision of $P$ such that for every positive-dimensional  cell in $\cS$ contains no interior lattice points. 
By  Theorem~\ref{t:localcoef}, \eqref{i:push4} in Lemma~\ref{l:pushstar} and \eqref{j:4} in Lemma~\ref{l:localprop}, 
\[
l^*(P;t) = l(\cS;t) + \alpha(t)t^2, 
\]
where $l(\cS;t)$ has non-negative, symmetric, unimodal coefficients and $\alpha(t)$ has non-negative integer coefficients. 
\end{proof}

\begin{remark}
It follows immediately from Theorem~\ref{t:lower} and the symmetry of the local $h^*$-polynomial (\eqref{j:1} in Lemma~\ref{l:localprop}) that the coefficients 
of the local $h^*$-polynomial are symmetric and unimodal for $\dim P \le 4$.  
\end{remark}

 \begin{example} \label{e:nonunimdal}
The coefficients of the local $h^*$-polynomial are not unimodal in general. For example, let $f = \frac{1}{3}(1,1,1,1,1)$ and $M = \Z^5 + \Z \cdot f$. Then 
if $P$ is the convex hull of the origin and the standard basis vectors $e_1,\ldots, e_5$ of $\Z^5$, one computes using Example~\ref{e:hstarsimplex} that
$\lc(P;t) = t^2 + t^4$.  
\end{example}

\begin{remark}\label{r:weakerhyp}
We don't need the existence of a unimodular triangulation to guarantee the unimodality of the coefficients of the local $h^*$-polynomial. Indeed, it follows from \eqref{i:push4} in Lemma~\ref{l:pushstar} and Theorem~\ref{t:localcoef} that
if $P$ admits a regular lattice polyhedral decomposition $\cS$ such that, for every cell $F$ in $\cS$, either $l_{[\sigma(F),P]}(\lk_\cS(F);t) = 0$ or $l^*(F;t)$ has symmetric, unimodal coefficients, then $l^*(P;t)$ has symmetric, unimodal coefficients. 
\end{remark}

\section{The limit mixed $h^*$-polynomial}\label{s:lmhstar}

In this section, we introduce the limit mixed $h^*$-polynomial,a two-variable invariant of a lattice polyhedral subdivision $\cS$ of a lattice polytope $P$.

\subsection{Definition of limit mixed $h^*$-polynomial}

We continue with the notation of previous sections with  $R =\Z[t^{\pm 1/2}]$ and $R_{uv}=\Z[u^{\pm 1/2},v^{\pm 1/2}]$. 
Recall that, by Proposition~\ref{p:alternate}, we have a pushforward map $\sigma_*: \mathcal{A}(\cS,\kappa_\cS; R_{uv}) \rightarrow \mathcal{A}([\emptyset,P],\kappa_P; R_{uv})$, satisfying
\[
\sigma_* ( e_F * \gamma_\cS|_{t = uv} )  = \sum_{Q \subseteq P} \eta(Q,F)|_{t = uv} e_Q = \sum_{Q \subseteq P} \lambda(Q,F)|_{t = uv} e_Q * \gamma_P|_{t = uv}. 
\]
Recall that, by Remark~\ref{r:heartdef} and  Lemma~\ref{l:reciprocity}, we may consider $\tilde{\eta}^\heartsuit_\cS = \sum_{F \in \cS} \lambda_\cS^\heartsuit(F)|_{t = uv^{-1}} e_F * \gamma_\cS|_{t = uv}  \in \mathcal{A}(\cS,\kappa_\cS; R_{uv})$. 

\begin{definition}
Let $\cS$ be a lattice polyhedral subdivision of a lattice polytope $P$.  Define $\eta_{\cS,\infty}^\heartsuit \in \mathcal{A}([\emptyset,P],\kappa_P; R_{uv})$ and $\lambda_{\cS,\infty}^\heartsuit \in S(R_{uv})^P$ by
\[
\eta_{\cS,\infty}^\heartsuit = \lambda_{\cS,\infty}^\heartsuit * \gamma_P|_{t = uv} := \sigma_* (\tilde{\eta}^\heartsuit_\cS).
\]
\end{definition}

The analogous polynomial invariants are described below.

\begin{definition}\label{d:new}
Let $\cS$ be a lattice polyhedral subdivision of a lattice polytope $P$.  
The \define{limit mixed $h^*$-polynomial} of $(P,\cS)$ is
\[
h^*(P,\cS;u,v) := \sum_{F \in \cS} v^{\dim F + 1}\lc(F; uv^{-1})  h(\lk_\cS(F);uv),
\]
and the \define{local limit mixed $h^*$-polynomial} of $(P,\cS)$ is
\[
\lc(P, \cS;u,v) := \sum_{F \in \cS} v^{\dim F + 1}\lc(F; uv^{-1}) l_{[\sigma(F),P]}(\lk_\cS(F);uv).
\]
If $P$ is empty, then $h^*(P,\cS;u,v) =  \lc(P, \cS;u,v)  = 1$.  Equivalently, we may define: 
\[
h^*(P,\cS;u,v) := (uv)^{(\dim P + 1)/2}\eta_{\cS,\infty}^\heartsuit(P),
\]
\[
\lc(P, \cS;u,v) := (uv)^{(\dim P + 1)/2} \lambda_{\cS,\infty}^\heartsuit(P).
\]
\end{definition}

These polynomials satisfy a number of immediate properties.

\begin{theorem}\label{t:basic}
Let $\cS$ be a lattice polyhedral subdivision of a lattice polytope $P$.  
Then the limit mixed $h^*$-polynomial $h^*(P,\cS;u,v)$ and  local limit mixed $h^*$-polynomial $\lc(P,\cS;u,v)$
satisfy the following properties:

\begin{enumerate}

\item\label{i:rh1limit*} ($uv$-interchange) The limit mixed $h^*$-polynomial and   local limit mixed $h^*$-polynomial are invariant under the interchange of $u$ and $v$ i.e.
\[
h^*(P,\cS;u,v) = h^*(P,\cS;v,u), \; \; \lc(P,\cS;u,v) = \lc(P,\cS;v,u).
\]

\item\label{i:rh2limit*} (specialization) We have the specializations:
\[
h^*(P,\cS;u,1)= h^*(P;u), \; \; \lc(P,\cS;u,1)  = \lc(P;u). 
\]

\item\label{i:rh7limit*} (symmetry) We have
\[
(uv)^{\dim P + 1} h^*(P,\cS ;u^{-1},v^{-1}) = \sum_{Q \subseteq P} h^*(Q,\cS|_Q;u,v) (uv - 1)^{\dim P - \dim Q},
\]
\[
(uv)^{\dim P + 1} \lc(P,\cS;u^{-1},v^{-1}) =\lc(P,\cS;u,v). 
\]

\item\label{i:rh3limit*} (constant terms) If $P$ is non-empty, then 
\[
h^*(P,\cS;0,v)  = 1,   \; \; \lc(P,\cS;0,v)   = 0.
\]

\item\label{i:rh4limit*} (identity subdivision) If $\cS$ is the trivial subdivision, then 
\[
h^*(P,\cS;u,v) =  h^*(P;u,v), \; \;\lc(P,\cS;u,v) = v^{\dim P + 1} \lc(P;uv^{-1}).
\]

\item\label{i:rh5limit*} (non-negativity) The coefficients of the limit mixed  $h^*$-polynomial and local limit mixed  $h^*$-polynomial are non-negative. 
 
\item\label{i:rh6'limit*} (inversion)  We have 
\begin{align*}
h^*(P, \cS;u,v) &= \sum_{Q \subseteq P} \lc(Q, \cS|_Q;u,v)g([Q,P];uv),\\
\lc(P, \cS; u,v) &=  \sum_{Q \subseteq P} (-1)^{\dim P - \dim Q}h^*(Q, \cS|_Q;u,v) g([Q,P]^*;uv).
\end{align*}

\end{enumerate}

\end{theorem}
\begin{proof}

Property~\eqref{i:rh1limit*} follows from the symmetry of the  local $h^*$-polynomial (\eqref{j:1} of Lemma~\ref{l:localprop}).

Property~\eqref{i:rh2limit*} follows from Lemma~\ref{l:pushstar}. 

Property~\eqref{i:rh7limit*} follows since  $\eta_{\cS,\infty}^\heartsuit \in \mathcal{A}([\emptyset,P],\kappa_P; R_{uv})$ and $\lambda_{\cS,\infty}^\heartsuit \in S(R_{uv})^P$.

Property~\eqref{i:rh3limit*} follows from \eqref{j:4} in Lemma~\ref{l:localprop}, Example~\ref{e:hconstant} and Example~\ref{e:lconstant}.

Property~\eqref{i:rh4limit*} follows from Remark~\ref{r:heartdef}. 

Property~\eqref{i:rh5limit*} follows from \eqref{j:3} in Lemma~\ref{l:localprop}, Example~\ref{e:intersect} and Theorem~\ref{t:localcoef}.

Property~\eqref{i:rh6'limit*} follows since  $\eta_{\cS,\infty}^\heartsuit = \lambda_{\cS,\infty}^\heartsuit * \gamma_P|_{t = uv}$ by definition, using Theorem~\ref{t:inverse}.

\end{proof}

The following generalizes Remark~\ref{r:hstardiamond}. 

\begin{remark}\label{r:hstardiamond2}
Assuming that $\cS$ is a lattice polyhedral subdivision of a  non-empty lattice polytope $P$, we may write 
\[
h^*(P,\cS;u,v) = 1 + uv \sum_{ 0 \le p,q \le \dim P - 1} h^*_{p,q} u^p v^q,
\]
\[
\lc(P,\cS;u,v) = uv \sum_{ 0 \le p,q \le \dim P - 1} \lc_{p,q} u^p v^q.
\]
 for some non-negative integers $h^*_{p,q}$ and $ \lc_{p,q}$.  We may visualize these coefficients 
in diamonds, which we call the \define{ $h^*$-diamond} and \define{local $h^*$-diamond} of $(P,\cS)$, 
by placing $h^*_{p,q}$ and $\lc_{p,q}$ at point $(q - p, p + q)$  in $\Z^2$ respectively (see Figure~\ref{f:diamond} below).
Note that when $\cS$ is the trivial subdivision,  the $h^*$-diamond and local $h^*$-diamond of 
$(P,\cS)$ coincide with the $h^*$-diamond and local $h^*$-diamond of 
$P$ introduced in Remark~\ref{r:hstardiamond}. Observe that the $h^*$-diamond and local $h^*$-diamond are symmetric about the vertical axis, and the local $h^*$-diamond is symmetric about the horizontal axis.
By comparing the coefficients of $u^{\dim P}v^{q + 1}$  in $h^*(P,\cS;u,v)$ and $\lc(P,\cS;u,v)$ using, for example, \eqref{i:rh6'limit*} in Theorem~\ref{t:basic}, 
the $h^*$-diamond and local $h^*$-diamond are identical along the part of the  boundary of the diamonds that is above and including the horizontal middle strip i.e $h^*_{\dim P - 1,q} = l^*_{\dim P - 1,q}$ for all $q$. 
We recover the coefficients of $h^*(P;u)$ and $\lc(P;u)$  by summing the coefficients of the $h^*$-diamond and local $h^*$-diamond respectively along a fixed choice of diagonal
i.e. $h^*_{i + 1} = \sum_{p + q = i} h^*_{p,q}$,  $l^*_{i + 1} = \sum_{p + q = i} l^*_{p,q}$. 
It follows from Theorem~\ref{t:lower} that each horizontal strip of the diamonds satisfies the following lower bound theorem: its first entry is a lower bound for the other entries i.e. 
$h^*_{k,0} \le h^*_{k-i,i}$, $h^*_{\dim P - 1,k} \le h^*_{\dim P - 1-i, k+i}$, $\lc_{k,0} \le \lc_{k-i,i}$  and $l^*_{\dim P - 1,k} \le l^*_{\dim P - 1-i, k+i}$. 
Moreover, if $\cS$ is regular then the coefficients of each vertical strip of the local $h^*$-diamond are symmetric and unimodal by Theorem~\ref{t:unimodality} below.  
\end{remark}

\begin{figure}[htb]
\begin{center}
\[
\begin{array}{p{5cm}p{5cm}}
\xy
(9,-13)*{}="A"; (-27,17)*{}="B"; (9,47)*{}="C"; (46,17)*{}="D";
"A"; "B" **\dir{-};
"A"; "D" **\dir{-};
"C"; "B" **\dir{-};
"C"; "D" **\dir{-};
(-10,20)*{
\xymatrix@=0.2em{ 
 &  &  & h^*_{3,3} & &                                              \\
  &  & h^*_{3,2} & & h^*_{2,3} &  &                                \\
   &  h^*_{3,1}  &  & h^*_{2,2} & & h^*_{1,3} &                    \\
     h^*_{3,0}  &   & h^*_{2,1}  &  & h^*_{1,2} &  & h^*_{0,3}      \\
      &  h^*_{2,0}  &  & h^*_{1,1} & & h^*_{0,2} &                \\
        &  & h^*_{1,0} & & h^*_{0,1} &  &                         \\
        &  &  & h^*_{0,0} & &                                \\
 }
};
(15,-17)*{\hbox{\small $h^*$-diamond of $(P,\cS)$}};
\endxy
&
\xy
(9,-13)*{}="A"; (-27,17)*{}="B"; (9,47)*{}="C"; (46,17)*{}="D";
"A"; "B" **\dir{-};
"A"; "D" **\dir{-};
"C"; "B" **\dir{-};
"C"; "D" **\dir{-};
(-8,20)*{
\xymatrix@=0.15em{ 
 &  &  & \lc_{3,3} & &                                              \\
  &  & \lc_{3,2} & & \lc_{2,3} &  &                                \\
   &  \lc_{3,1}  &  & \lc_{2,2} & & \lc_{1,3} &                    \\
     \lc_{3,0}  &   & \lc_{2,1}  &  & \lc_{1,2} &  & \lc_{0,3}      \\
      &  \lc_{2,0}  &  & \lc_{1,1} & & \lc_{0,2} &                \\
        &  & \lc_{1,0} & & \lc_{0,1} &  &                         \\
        &  &  & \lc_{0,0} & &                                \\
 }
};
(15,-17)*{\hbox{\small local $h^*$-diamond of $(P,\cS)$}};

\endxy
\end{array} 
\]
\end{center}

\caption{ $h^*$-diamond and  local $h^*$-diamond of $(P,\cS)$ when $\dim P = 4$}
\label{f:diamond}
\end{figure}

\begin{remark}\label{r:unimodalblah} As a consequence of  properties \eqref{i:rh5limit*} and \eqref{i:rh6'limit*} in Theorem~\ref{t:basic} together with the non-negativity of the $g$-polynomial justified by formula \eqref{e:gintersect},  we have $h^*(P,\cS;u,v) \ge \lc(P,\cS;u,v)$. In particular, 
the coefficients of the local $h^*$-diamond  are a lower bound for the corresponding coefficients of the $h^*$-diamond (more generally, see Theorem~\ref{t:lower entry}). \end{remark}

\begin{remark}\label{r:deducelower}
The fact that the coefficients of each vertical strip of the local $h^*$-diamond are symmetric and unimodal when $\cS$ is regular implies the lower bound theorem ,Theorem~\ref{t:lower}. Indeed, let
 $\cS$ be a regular, lattice polyhedral subdivision of $P$ such that every positive-dimensional  cell in $\cS$ contains no interior lattice points. By Example~\ref{e:lower}, $\lc_{0,q} = 0$ for all $q > 0$, and hence
 \[
 \lc_1 = \sum_q \lc_{q,0} = \lc_{0,0} \le \lc_{i,i} \le \sum_q  \lc_{q + i,i} = \lc_{i + 1}, 
 \]
 for $i \ge 0$. 
 
\end{remark}

Our invariants are well-behaved under refinement. The following result generalizes Lemma~\ref{l:pushstar}. 

\begin{proposition} \label{p:hstarrefinement}
Let $\cS$ be a lattice polyhedral subdivision of a lattice polytope $P$, and let $\cS'$ be a lattice polyhedral subdivision refining $\cS$. 
Then 
\begin{enumerate}
\item \label{hs:1} \[h^*(P, \cS';u,v) = \sum_{ \substack{ F \in \cS \\ \sigma(F) = P } } h^*(F,\cS'|_F;u,v)(uv - 1)^{\dim P - \dim F}  ,   \]
\item \label{hs:2} \[h^*(P,\cS';u,v) = \sum_{F \in \cS} \lc(F,\cS'|_F;u,v) h(\lk_\cS(F);uv),\]
\item \label{hs:3} \[\lc(P,\cS';u,v) = \sum_{F \in \cS} \lc(F,\cS'|_F;u,v) l_{[\sigma(F),P]}(\lk_\cS(F);uv),\]
\end{enumerate}
\end{proposition}

\begin{proof}
Let $\tau: \cS' \rightarrow \cS$ and $\sigma: \cS \rightarrow [\emptyset,P]$ be the corresponding strong formal subdivisions. By definition, for any cell $F$ in $\cS$, 
\[
\tau_* (\tilde{\eta}^\heartsuit_{\cS'})(F) = (uv)^{-(\dim F + 1)/2} h^*(F, \cS'|_F; u,v). 
\]
Hence \eqref{hs:1} follows from the fact that
$\eta^\heartsuit_{\cS',\infty} = \sigma_*(\tau_* (\tilde{\eta}^\heartsuit_{\cS'}))$. The other two formulas follow as in Lemma~\ref{l:pushstar}.
\end{proof}

\subsection{Examples}

\begin{example} \label{e:unimodularrefined} 
Let $\cS$ be a  unimodular triangulation of a lattice polytope $P$.  Then for every non-empty face $F$ of $\cS$, $l^*(F;t)=0$.  By definition, we have
\begin{eqnarray*}
h^*(P,\cS;u,v) &=&  h^*(P;uv)  = h(\cS;uv),\\
\lc(P, \cS;u,v) &=& \lc(P;uv)  = l_P(\cS;uv).
\end{eqnarray*}
In particular, the $h^*$-diamond and local $h^*$-diamond are concentrated on the central vertical strip. 
As in Example~\ref{e:unimodular}, if the triangulation is regular, then the coefficients of $\lc(P,\cS;u,v)$ as a polynomial in $uv$ are unimodal by Theorem~\ref{t:localcoef}.  
\end{example}

\begin{example}\label{e:lower}
We continue with the notation of Remark~\ref{r:hstardiamond2}.
Recall that $h^*(P,\cS;u,1) = h^*(P;u)$ and  $\lc(P,\cS;u,1) = \lc(P;u)$ and the linear coefficients of $h^*(P;u)$ and $\lc(P;u)$ are $\#(P \cap M) - \dim P - 1$ and $\#(\Int(P) \cap M)$ respectively, by  Example~\ref{e:hstarlinear} and both Properties \eqref{j:1} and \eqref{j:4} of Lemma~\ref{l:localprop} respectively.
By comparing coefficients of $uv^{q + 1}$  in both sides of the equations in Definition~\ref{d:new}, and using Example~\ref{e:hconstant} and Example~\ref{e:lconstant}, we deduce that for $q > 0$, 
\[
h^*_{0,q} = h^*_{q,0}    = \sum_{\substack{F \in \cS \\\ \dim F = q + 1}} \# (\Int(F) \cap M),
\]
\[
\dim P + 1 + h^*_{0,0}    = \sum_{\substack{F \in \cS \\\ \dim F \le 1}} \# (\Int(F) \cap M),
\]
\begin{align*}
\lc_{0,q} &= \lc_{q,0} =  \lc_{\dim P - 1,\dim P - 1 - q} =  \lc_{\dim P - 1 - q, \dim P - 1} = \\
&h^*_{\dim P - 1,\dim P - 1 - q} =  h^*_{\dim P - 1 - q, \dim P - 1} = \sum_{\substack{F \in \cS, \sigma(F) =  P \\\ \dim F = q + 1}} \# (\Int(F) \cap M),
\end{align*}
\[
\lc_{0,0} = \lc_{\dim P - 1,\dim P - 1} = h^*_{\dim P - 1,\dim P - 1} =  \sum_{\substack{F \in \cS, \sigma(F) =  P \\\ \dim F \le 1}} \# (\Int(F) \cap M). 
\]
This gives an explicit combinatorial description of the boundaries of the $h^*$-diamond and local $h^*$-diamond of $(P,\cS)$. In particular, this gives a complete description of both 
diamonds when $\dim P = 1,2$. 
\end{example}

\begin{example}\label{e:dim3}
If $\dim P = 3$, then Example~\ref{e:lower} describes every term in the $h^*$-diamond and local $h^*$-diamond except the middle terms $h^*_{1,1}$ and $\lc_{1,1}$ respectively. 
The remaining terms can be computed via the formulas
\[
\lc_2 = \lc_{1,1} + 2 \cdot  \sum_{\substack{F \in \cS, \sigma(F) =  P \\\ \dim F = 2}} \# (\Int(F) \cap M),
\]
\[
h^*_2 =  h^*_{1,1} +   2 \cdot  \sum_{\substack{F \in \cS, \sigma(F) =  P \\\ \dim F = 2}} \# (\Int(F) \cap M) +  \sum_{\substack{F \in \cS,  F \subseteq  \partial P \\\ \dim F = 2}} \# (\Int(F) \cap M).
\]
\end{example}

\begin{example}
In the case when $\cS$ is a lattice triangulation of $P$, we have the following description of $h^*(P,\cS;u,v)$. 
If $P$ is a lattice polytope in a lattice    $N$, then let $\tau_P$ denote the cone over $P \times \{ 1 \}$ in $N_\R \times \R$.  
Let $\Sigma$ denote the fan refinement of $\tau_P$ induced by $\cS$, with cones given by the cones over $F \times \{ 1\}$, where $F$ is a cell of $\cS$ (the empty cell of $\cS$ corresponds to $\{0\}$). 
Then the lattice points in $\tau_P$ admit a well-known involution $\iota$. More specifically, a non-zero lattice point $w$ in $\tau_P$ lies in the relative interior of a unique cone $\tau$ in $\Sigma$, corresponding to 
a non-empty face $F$ of $\cS$ with vertices $v_1, \ldots , v_r$. We may uniquely write $w = \sum_{i = 1}^{r} \alpha_i (v_i, 1)$  for some  $\alpha_i \in \Q_{>0}$. 
Then $\iota(w) =  \sum_{i = 1}^{r} (\lceil \alpha_i \rceil  + \lfloor \alpha_i \rfloor - \alpha_i) (v_i, 1)$, $\iota(0) = 0$, and 
\[
h^*(P, \cS;u,v)/(1 - uv)^{\dim P + 1} =  \sum_{w \in \tau_P \cap (N \times \Z)} u^{\psi(w)}v^{\psi(\iota(w))},
\]
where $\psi: N \times \Z \rightarrow \Z$ denotes projection onto the second co-ordinate. 
If we further assume that $P$ is a simplex, then \eqref{i:rh6'limit*} of Theorem~\ref{t:basic} implies that \[
 \lc(P, \cS;u,v)/(1 - uv)^{\dim P + 1} = \sum_{w \in \Int(\tau_P) \cap (N \times \Z)} u^{\psi(w)}v^{\psi(\iota(w))}.
\]
\end{example}

\section{The refined  limit mixed $h^*$-polynomial}\label{s:rlmhstar}

In this section, we introduce the refined limit mixed $h^*$-polynomial, 
a three-variable invariant of a lattice polyhedral subdivision $\cS$ of a lattice polytope $P$.

\subsection{Definition and properties}

We continue with the notation of previous sections with  $R =\Z[t^{\pm 1/2}]$ and $R_{uv}=\Z[u^{\pm 1/2},v^{\pm 1/2}]$. As in Remark~\ref{r:extend}, we also consider 
$R_{uvw}=\Z[u^{\pm 1/2},v^{\pm 1/2},w^{\pm 1}]$ with involution $\bar{r(u,v,w)} = r(u^{-1},v^{-1},w^{-1})$, and $\Z[(uvw^2)^{\pm 1/2}] \subseteq R_{uvw}$
i.e. $uvw^2$ will play the role of $t$ in the previous sections. For example, as in Lemma~\ref{l:gpoly} and  with a slight abuse of notation, 
the acceptability operator of a locally Eulerian poset $B$ is given by $\gamma_B(x,x')|_{t = uvw^2} = (uvw^2)^{-\rho(x,x')/2} g([x,x'];uvw^2)$.  

Recall that if $\cS$ is a lattice polyhedral subdivision of a lattice polytope $P$, 
then $\eta_{\cS,\infty}^\heartsuit \in \mathcal{A}([\emptyset,P],\kappa_P; R_{uv})$ and $\lambda_{\cS,\infty}^\heartsuit \in S(R_{uv})^P$ are defined by
\[
\eta_{\cS,\infty}^\heartsuit = \lambda_{\cS,\infty}^\heartsuit * \gamma_P|_{t = uvw^2} := \sigma_* (\tilde{\eta}^\heartsuit_\cS).
\]
Consider the natural inclusion $i_w: S(R_{uv})^P \rightarrow S(R_{uvw})^P, i_w(f(u,v)) = f(u,v)$. This induces a natural inclusion
\[
j_w:  \mathcal{A}([\emptyset,P],\kappa_P; R_{uv}) \rightarrow \mathcal{A}([\emptyset,P],\kappa_P; R_{uvw}),
\]
\[
 j_w(f) = \iota_w(f * \gamma_P^{-1}|_{t = uv}) * \gamma_P|_{t = uvw^2}.
\]

\begin{definition}\label{d:refinedhstar}
Let $\cS$ be a lattice polyhedral subdivision  of a lattice polytope $P$. Then the \define{refined limit mixed $h^*$-polynomial} of $(P,\cS)$ is
\[
h^*(P,\cS;u,v,w) = \sum_{Q \subseteq P} w^{\dim Q + 1}\lc(Q, \cS|_Q;u,v) g([Q,P];uvw^2).
\]
If $P$ is empty, then $h^*(P,\cS;u,v,w)  = 1$. 
Equivalently, we may define: 
\[
h^*(P,\cS;u,v,w) := (uvw^2)^{(\dim P + 1)/2} j_w(\eta_{\cS,\infty}^\heartsuit(P)).
\]
\end{definition}

The refined limit mixed $h^*$-polynomial satisfies a number of immediate properties. 

\begin{theorem}\label{t:combinatorics}
Let $\cS$ be a lattice polyhedral subdivision of a lattice polytope $P$.  
Then the refined limit mixed $h^*$-polynomial $h^*(P,\cS;u,v,w)$ 
satisfies the following properties:

\begin{enumerate}

\item\label{i:rh1refine*} ($uvw$-interchange) The refined limit mixed $h^*$-polynomial satisfies:
\[
h^*(P,\cS;u,v,w) = h^*(P,\cS;v,u,w), 
\]
\[
h^*(P,\cS;u,v,w) = h^*(P,\cS;u^{-1},v^{-1},uvw). 
\]

\item\label{i:rh2refine*} (specialization) We have the specializations:
\[
h^*(P,\cS;u,v,1)= h^*(P,\cS;u,v),
\]
\[
h^*(P,\cS;uw^{-1},1,w)= h^*(P;u,w).
\]
Hence, we have the following commutative diagram of invariants 
\[\xymatrix{
h^*(P,\cS;u,v,w)   \ar[r]^{\substack{u \mapsto uw^{-1} \\ v \mapsto 1}}   \ar[d]^{w \mapsto 1}  & h^*(P;u,w) \ar[d]^{w \mapsto 1}  &\\
 h^*(P,\cS;u,v) \ar[r]^{v \mapsto 1} &   h^*(P;u) \ar[r]^>>>>>{u \mapsto 1}  & (\dim P)!\vol(P),
}\] 
where $\vol(P)$ is the Euclidean volume of $P$.

\item\label{i:rh7refine*} (symmetry) We have
\[
(uvw^2)^{\dim P + 1} h^*(P,\cS ;u^{-1},v^{-1},w^{-1}) = \sum_{Q \subseteq P} h^*(Q,\cS|_Q;u,v,w) (uvw^2 - 1)^{\dim P - \dim Q}.
\]

\item\label{i:rh3refine*} (constant terms) We have $h^*(P,\cS;0,v,w) = h^*(P,\cS;u,v,0) = 1$.

\item\label{i:rh4refine*} (identity subdivision) If $\cS$ is the trivial subdivision, then 
\[ h^*(P,\cS;u,v,w) =  h^*(P;uw,vw).
\] 

\item\label{i:rh5refine*} (non-negativity) The coefficients of the refined limit mixed  $h^*$-polynomial are non-negative.

\item\label{i:rh6'refine*} (inversion)  We have 
\[
w^{\dim P + 1} \lc(P, \cS; u,v) =  \sum_{Q \subseteq P} (-1)^{\dim P - \dim Q}h^*(Q, \cS|_Q;u,v,w) g([Q,P]^*;uvw^2).
\]

\item\label{i:refinedegreelimit*} (degree) All terms in $h^*(P,\cS;u,v,w)$ have degree in $w$ at most $\dim P + 1$, and
the coefficient of  $w^{\dim P + 1}$ is $\lc(P, \cS; u,v)$.

\end{enumerate}

\end{theorem}
\begin{proof}
The equations in Property~\eqref{i:rh1refine*} follow from \eqref{i:rh1limit*} and \eqref{i:rh7limit*} in Theorem~\ref{t:basic} respectively.

The equations in Property~\eqref{i:rh2refine*} follow from \eqref{i:rh6'limit*} 
and \eqref{i:rh2limit*} in Theorem~\ref{t:basic} respectively.

Property~\eqref{i:rh7refine*} is equivalent to the fact that $j_w(\eta_{\cS,\infty}^\heartsuit)  \in  \mathcal{A}([\emptyset,P],\kappa_P; R_{uvw})$. 

Property~\eqref{i:rh3refine*} follows from \eqref{i:rh3limit*} in Theorem~\ref{t:basic} and Example~\ref{e:gconstant}.

Property~\eqref{i:rh4refine*} follows from \eqref{i:rh4limit*} in Theorem~\ref{t:basic} . 

Property~\eqref{i:rh5refine*} follows from \eqref{i:rh4limit*} in Theorem~\ref{t:basic} and \eqref{e:gintersect}.  

Property~\eqref{i:rh6'refine*} follows since  $i_w(\lambda_{\cS,\infty}^\heartsuit)  = j_w(\eta_{\cS,\infty}^\heartsuit) * \gamma_P^{-1}|_{t = uvw^2}$ by definition, using Theorem~\ref{t:inverse}. 

Property~\eqref{i:refinedegreelimit*} follows since $g([Q,P];uvw^2)$ has  degree in $w$ strictly less than $\dim P - \dim Q$ for $Q \nsubseteq P$.

\end{proof}

If  $\cS$ is a lattice polyhedral subdivision of a  non-empty lattice polytope $P$, we may write 
\[
h^*(P,\cS;u,v,w) = 1 + uvw^2 \sum_{r = 0}^{ \dim P - 1}  \sum_{ 0 \le p,q \le r} h^*_{p,q,r} u^p v^q w^r,
\]
 for some non-negative integers $h^*_{p,q,r}$.  
 We may visualize these coefficients 
in diamonds as follows: for $0 \le r \le \dim P - 1$, the \define{$r$-local $h^*$-diamond} of $(P,\cS)$ is obtained by placing $h^*_{p,q,r}$ at point $(q - p, p + q)$  in $\Z^2$ for $0 \le p,q \le r$ 
(see Figure~\ref{f:localdiamond} below). 
With the notation of Remark~\ref{r:hstardiamond2},  the local $h^*$-diamond of $(P,\cS)$ is equal to the $(\dim P - 1)$-local $h^*$-diamond of $(P,\cS)$ i.e. $ \lc_{p,q} = h^*_{p,q,\dim P - 1}$, and the $h^*$-diamond of $(P,\cS)$ is obtained by stacking the $r$-local $h^*$-diamonds of $(P,\cS)$ on top of each other and summing  for $0 \le r \le \dim P - 1$ i.e. $h^*_{p,q} = \sum_r h^*_{p,q,r}$. Each $r$-local $h^*$-diamond is symmetric about the horizontal and vertical axes. 
By summing the coefficients of the $r$-local $h^*$-diamond of $(P,\cS)$ along a fixed choice of diagonal, we recover the $(r +1)^{\textrm{st}}$ horizontal strip of the $h^*$-diamond of $P$. 
For example, if $\cS$ is the trivial subdivision, then the $r$-local $h^*$-diamond of $(P,\cS)$ is concentrated on the middle horizontal strip which coincides with the 
$(r +1)^{\textrm{st}}$ horizontal strip of the $h^*$-diamond of $P$.    By Theorem~\ref{t:lowerentry} below, the first entry of each horizontal strip of the $r$-local $h^*$-diamond is a lower bound for the other entries. 
By Theorem~\ref{t:unimodality} below, if $\cS$ is regular then the coefficients of each vertical strip of the $r$-local $h^*$-diamond are symmetric and unimodal.

 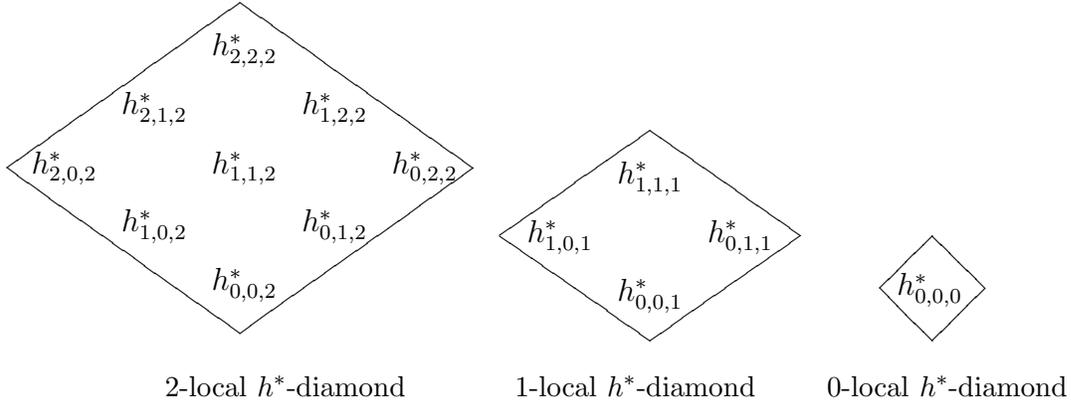
\begin{figure}[htb]
\begin{center}
\[
\begin{array}{p{5cm}p{4cm}p{5cm}}
\xy
(4,3)*{}="A"; (-27,25)*{}="B"; (4,47)*{}="C"; (35,25)*{}="D";
"A"; "B" **\dir{-};
"A"; "D" **\dir{-};
"C"; "B" **\dir{-};
"C"; "D" **\dir{-};
(-9.7,20.5)*{
\xymatrix@=0.2em{ 
   &  & h^*_{2,2,2} &                                               \\
    & h^*_{2,1,2} & & h^*_{1,2,2} &                                  \\
        h^*_{2,0,2}  &  & h^*_{1,1,2} & & h^*_{0,2,2}                 \\
          & h^*_{1,0,2} & & h^*_{0,1,2} &                           \\
         &  & h^*_{0,0,2} & &                                \\
 }
};
(10,-4)*{\hbox{\small $2$-local $h^*$-diamond}};
\endxy
&
\xy
(2,2)*{}="A"; (-18,16)*{}="B"; (2,30)*{}="C"; (22,16)*{}="D";
"A"; "B" **\dir{-};
"A"; "D" **\dir{-};
"C"; "B" **\dir{-};
"C"; "D" **\dir{-};
(-5,12)*{
\xymatrix@=0.2em{ 
          & h^*_{1,1,1} & {}                 \\
           h^*_{1,0,1} & & h^*_{0,1,1}                          \\
           & h^*_{0,0,1} &               \\
 }
};
(0,-4)*{\hbox{\small $1$-local $h^*$-diamond}};

\endxy
&
\xy
(-2,2)*{}="A"; (-9,9)*{}="B"; (-2,16)*{}="C"; (5,9)*{}="D";
"A"; "B" **\dir{-};
"A"; "D" **\dir{-};
"C"; "B" **\dir{-};
"C"; "D" **\dir{-};
(-5,9)*{
\xymatrix@=0.2em{ 
          & {} & {}                 \\
           {} & &{}                          \\
           & h^*_{0,0,0} &               \\
 }
};
(0,-4)*{\hbox{\small $0$-local $h^*$-diamond}};

\endxy

\end{array} 
\]
\end{center}

\caption{ $r$-local $h^*$-diamonds of $(P,\cS)$ when $\dim P = 3$}
\label{f:localdiamond}
\end{figure}

Substituting Definition~\ref{d:new} into Definition~\ref{d:refinedhstar} gives the following alternative expression for 
the refined limit mixed $h^*$-polynomial $h^*(P,\cS;u,v,w)$:
\begin{equation}\label{e:expanddef}
\sum_{ F \in \cS} \sum_{\sigma(F) \subseteq Q \subseteq P} w^{\dim Q + 1} v^{\dim F + 1} \lc(F;uv^{-1}) l_Q(\cS|_Q,F;uv) g([Q,P];uvw^2).
\end{equation}
By Theorem~\ref{t:hlocalh}, Theorem~\ref{t:localcoef}  and formula \eqref{e:gintersect} of Section~\ref{s:geometric}, the coefficients of all terms above have non-negative coefficients. 
We will use this expression below. 

\begin{theorem} \label{t:lowerentry}
The first entry of each horizontal strip of the $r$-local $h^*$-diamond is a lower bound for the other entries:
$h^*_{k,0,r} \le h^*_{k-i,i,r}$ and  $h^*_{r,k,r} \le h^*_{r-i, k+i,r}$. 
\end{theorem}

\begin{proof}
Consider the expression \eqref{e:expanddef} above. We have seen that all terms have non-negative coefficients.
Hence $h^*(P,\cS;u,v,w)$ is the sum of terms of the form $\alpha (uv)^k w^r \lc(F;uv^{-1})$, for some non-negative integers $\alpha,k,r$ and some cell $F \in \cS$. 
By Theorem~\ref{t:lower}, the polynomials  $v^{\dim F + 1}\lc(F;uv^{-1})$ satisfy the property that their linear coefficient is a lower bound for the other coefficients. The result follows. 
\end{proof}

The theorem below states that  if the lattice polyhedral subdivision $\cS$ is regular, then, for $r = 0, \ldots, \dim P - 1$,  the coefficients of each vertical strip of the $r$-local $h^*$-diamond form a 
symmetric, unimodal sequence. 

\begin{theorem} \label{t:unimodality} Let $\cS$ be a regular  lattice polyhedral subdivision of a lattice polytope $P$.  Then for $0 \le k \le r \le \dim P - 1$, the 
sequence $\{h^*_{k + i,i,r} \in \Z_{\ge 0} \mid 0 \le i \le r -k \}$ is symmetric, non-negative and unimodal.
\end{theorem}
\begin{proof}
Symmetry and non-negativity of the coefficients was established in \eqref{i:rh1refine*} and \eqref{i:rh5refine*} of Theorem~\ref{t:combinatorics} respectively. 
Consider the expression \eqref{e:expanddef} above. We have seen that   all terms have non-negative coefficients.
Hence $h^*(P,\cS;u,v,w)$ is the sum of terms of the form $\alpha u^pv^q w^r l_Q(\cS|_Q,F;uv)$, for some non-negative integers $\alpha,p,q,r$ and some cell $F \in \cS$ contained in a face $Q \subseteq P$, 
such that $r - p - q = \dim Q - \dim F$. 
Unimodality now follows from Theorem~\ref{t:localcoef} applied to $l_Q(\cS|_Q,F;uv)$. 
\end{proof}

\begin{remark} \label{r:geometric}
Theorem~\ref{t:unimodality} can also be proved geometrically by identifying the coefficients of the refined limit mixed $h^*$-polynomial with 
refined limit mixed Hodge numbers associated to a corresponding degeneration of hypersurfaces \cite[Corollary~5.11]{geometricpaper}. The unimodality statement then follows immediately from properties of the associated monodromy operator. This was the original motivation for the introduction of the refined limit mixed $h^*$-polynomial and was our original proof of Theorem~\ref{t:unimodality} and, 
as explained in Remark~\ref{r:deducelower}, Theorem~\ref{t:lower}. 
\end{remark}


The polytope $P$ specifies a projective toric variety.  It is useful for geometric applications to resolve singularities of the toric varieties.  This allows us to interpret $\lc$ as the Hodge-Deligne polynomial of a particular limit mixed Hodge structure.  We do this combinatorially here.  
Let $\Delta_P$ denote the normal fan to $P$ with all maximal cones removed. The cones $C_Q$ in $\Delta_P$  are in inclusion-reserving correspondence with the positive dimensional faces $Q$ of $P$.
Let $\Delta_P'$ denote a fan refinement of $\Delta_P$.  We will assume that $\Delta_P'$ is simplicial,  
which is possible by the resolution of singularities algorithm for toric varieties \cite[Sec. 2.6]{FulIntroduction}. Recall that $\Delta_P'$ is simplicial if  every cone $C'$ in $\Delta_P'$ is generated by precisely $\dim C'$ rays. Let $\Gamma,B$ be the lower Eulerian poset of cones of $\Delta_P',\Delta_P$ ordered by inclusion, respectively.  Let $\sigma:\Gamma\rightarrow B$ be defined by letting 
 $\sigma(C')$ be the smallest cone in $\Delta_P$ containing $C'$. It follows from Lemma~\ref{l:polyhedralsubdivision} that $\sigma$ is a strong formal subdivision. 
We set 
\[
\Phi(P,\cS,\Delta_P';u,v,w) = 
\sum_{  \substack{ Q \subseteq P  \\ \dim Q > 0   } } (-1)^{\dim Q} h^*(Q,\cS|_{Q};u,v,w)  \sum_{ \substack{ C' \in \Delta_P' \\  \sigma(C') = C_Q } }   (uvw^2 - 1)^{\dim C_Q - \dim C'},
\]
and 
\[
\Lambda(P,\cS,\Delta_P';u,v,w) = \sum_{ C' \in \Delta_P'  } (uvw^2 - 1)^{\dim P - \dim C'}   - \Phi(P,\cS,\Delta_P';u,v,w). 
\]

\begin{lemma}
We have the following symmetry
\[
\Lambda(P,\cS,\Delta_P';u,v,w) = (uvw^2)^{\dim P + 1}\Lambda(P,\cS,\Delta_P';u^{-1},v^{-1},w^{-1}).
\]
\end{lemma}

\begin{proof}
The cones $C_Q$ in the normal fan $\tilde{\Delta}_P$ of $P$ are in inclusion-reserving correspondence with the non-empty faces $Q$ of $P$.
We may consider a simplicial fan $\tilde{\Delta}_P'$ that refines $\tilde{\Delta}_P$, and contains $\Delta_P'$ as a subfan. 
Let $\tilde{\Gamma}$,$\tilde{B}$ be the lower Eulerian poset of cones of $\tilde{\Delta}_P'$,$\tilde{\Delta}_P$ ordered by inclusion, respectively.  
Consider the corresponding strong formal subdivision $\tilde{\sigma}: \tilde{\Gamma} \rightarrow \tilde{B}$. Since $h^*(Q,\cS|_{Q};u,v,w) = 1$ when $Q$ is a point, it follows that 
$\Lambda(P,\cS,\Delta_P';u,v,w)$ has the form:
\[
\sum_{ C' \in \tilde{\Delta}_P'} (uvw^2 - 1)^{\dim P - \dim C'} - \sum_{  \emptyset \ne Q \subseteq P } (-1)^{\dim Q} h^*(Q,\cS|_{Q};u,v,w)  \sum_{ \substack{ C' \in \tilde{\Delta}_P' \\  \tilde{\sigma}(C') = C_Q } }   (uvw^2 - 1)^{\dim C_Q - \dim C'}.
\]
It follows from Example~\ref{e:hEulerian} that $h(\tilde{\Gamma};t) = t^{\dim P}h(\tilde{\Gamma};t^{-1}) = \sum_{ C' \in \tilde{\Delta}_P'} (t - 1)^{\dim P - \dim C'}$. 
By \eqref{i:pthird} in Proposition~\ref{p:alternate}, we may rewrite the above as: 
\begin{equation}\label{e:newlambda}
\Lambda(P,\cS,\Delta_P';u,v,w) = h(\tilde{\Gamma};uvw^2) + \sum_{  \emptyset \ne Q \subseteq P } (-1)^{\dim Q + 1} h^*(Q,\cS|_{Q};u,v,w) h(\tilde{\Gamma}_{C_Q}; uvw^2). 
\end{equation}
Using the above expression \eqref{e:newlambda}, together with \eqref{i:rh7refine*} in Theorem~\ref{t:combinatorics} and \eqref{i:pthird} in Proposition~\ref{p:alternate}, we compute:
\begin{align*}
 &(uvw^2)^{\dim P + 1}\Lambda(P,\cS,\Delta_P';u^{-1},v^{-1},w^{-1}) = uvw^2h(\tilde{\Gamma};uvw^2) \\
&+ \sum_{  \emptyset \ne Q \subseteq P } (-1)^{\dim Q + 1} \sum_{Q' \subseteq Q} h^*(Q',\cS|_{Q'};u,v,w) (uvw^2 - 1)^{\dim Q - \dim Q'}  h(\IInt(\tilde{\Gamma})_{C_Q}; uvw^2). \\
\end{align*}
Dividing the above expression into the two cases when $Q' = \emptyset$ and $Q' \ne \emptyset$, we have the following formula for $(uvw^2)^{\dim P + 1}\Lambda(P,\cS,\Delta_P';u^{-1},v^{-1},w^{-1})$:
\begin{align*}
&uvw^2h(\tilde{\Gamma};uvw^2)  + (1 - uvw^2) \sum_{  \emptyset \ne Q \subseteq P } (1 - uvw^2)^{\dim Q}  h(\IInt(\tilde{\Gamma})_{C_Q}; uvw^2) \\
&+  \sum_{  \emptyset \ne Q' \subseteq P } (-1)^{\dim Q' + 1} h^*(Q',\cS|_{Q'};u,v,w)  \sum_{Q' \subseteq Q}  (1 - uvw^2)^{\dim Q - \dim Q'}h(\IInt(\tilde{\Gamma})_{C_Q}; uvw^2). 
\end{align*}
It follows from Definition~\ref{d:hpoly} 
that if $\sigma: \Gamma \rightarrow B$ is any strong formal subdivision between lower Eulerian posets, and $B$ has rank $n$ and rank function $\rho_B$, then 
\begin{align*}
t^{\rk(\Gamma)} h(\Gamma;t^{-1}) &= \sum_{y \in \Gamma}  g([\hat{0}_\Gamma,y];t) (t - 1)^{\rk(\Gamma) - \rho(\hat{0}_\Gamma,y)} \\
&= \sum_{x \in B}  (t - 1)^{n - \rho_B(\hat{0},x)}  \sum_{y \in \IInt(\Gamma_x)}  g([\hat{0}_\Gamma,y];t) (t - 1)^{\rk(\Gamma_x) - \rho(\hat{0}_\Gamma,y)} \\
&=  \sum_{x \in B}  (t - 1)^{n - \rho_B(\hat{0},x)}  t^{\rk(\Gamma_x) } h(\IInt(\Gamma)_x;t^{-1}).
\end{align*}
Replacing $t$ with $t^{-1}$ gives:
\begin{equation}\label{e:fromdef}
h(\Gamma;t) = \sum_{x \in B} (1 - t)^{n - \rho_B(\hat{0},x)} h(\IInt(\Gamma)_x;t). 
\end{equation}
Using \eqref{e:fromdef}, our previous expression for  $(uvw^2)^{\dim P + 1}\Lambda(P,\cS,\Delta_P';u^{-1},v^{-1},w^{-1})$ simplifies to 
\begin{align*}
 &uvw^2h(\tilde{\Gamma};uvw^2)  + (1 - uvw^2)h(\tilde{\Gamma};uvw^2)   \\
 &+ \sum_{  \emptyset \ne Q' \subseteq P } (-1)^{\dim Q' + 1} h^*(Q',\cS|_{Q'};u,v,w) h(\tilde{\Gamma}_{C_{Q'}}; uvw^2). 
\end{align*}
The result now follows by comparing the above expression with \eqref{e:newlambda}. 
\end{proof}

We have the following important characterization of the refined limit mixed $h^*$-polynomial. 

\begin{corollary}\label{c:superimportant}
The refined  limit mixed $h^*$-polynomial as an invariant of polyhedral subdivisions of lattice polytopes is uniquely characterized by the following properties:
\begin{enumerate}
\item\label{m:1} The degree of $h^*(P,\cS;u,v,w)$ as a polynomial in $w$ is at most $\dim P + 1$.  

\item\label{m:2} The refined  limit mixed $h^*$-polynomial specializes to the limit mixed $h^*$-polynomial i.e.
\[
h^*(P,\cS;u,v,1) = h^*(P,\cS;u,v). 
\]

\item\label{m:3} If $\Delta_P'$ denotes a simplicial fan refinement of $\Delta_P$ then for $\Lambda$ defined in terms of the refined limit mixed $h^*$-polynomial as above, we have 
\[
\Lambda(P,\cS,\Delta_P';u,v,w) = (uvw^2)^{\dim P + 1}\Lambda(P,\cS,\Delta_P';u^{-1},v^{-1},w^{-1}).
\]
\end{enumerate}
\end{corollary}
\begin{proof}
The properties  were established above.

 Note that 
\[
\Phi(P,\cS,\Delta_P';u,v,w)  = h^*(P,\cS;u,v,w) + \lambda(u,v,w),
\]
where  $\lambda(u,v,w)$ involves terms of the form $h^*(Q,\cS|_{Q};u,v,w)$ where $\dim Q < \dim P$. 
By induction on dimension, we may assume  that $\lambda(u,v,w)$ is determined. Then by \eqref{m:1}, as a polynomial in $w$, $\Phi(P,\cS,\Delta_P';u,v,w)$
is known in $w$-degree strictly greater than $\dim P + 1$. Hence by \eqref{m:3} it is known in  $w$-degree strictly less than $\dim P + 1$. Finally, \eqref{m:2} determines $h^*(P,\cS;u,v,w)$.
\end{proof}

Similarly, we have the following characterization of the mixed $h^*$-polynomial. 
With the notation above, we set $\Lambda(P,\cS,\Delta_P';u,w) := \Lambda(P,\cS,\Delta_P';uw^{-1},1,w)$. Using \eqref{prop2} in Theorem~\ref{t:combinatorics}, we may write this as:
\[
 \sum_{ C' \in \Delta_P'  } (uw - 1)^{\dim P - \dim C'} - \sum_{  \substack{ Q \subseteq P  \\ \dim Q > 0   } } (-1)^{\dim Q} h^*(Q,\cS|_{Q};u,w)  \sum_{ \substack{ C' \in \Delta_P' \\  \sigma(C') = C_Q } }   (uw - 1)^{\dim C_Q - \dim C'}.
\] 

\begin{corollary}
The mixed $h^*$-polynomial as an invariant of lattice polytopes is uniquely characterized by the following properties:
\begin{enumerate}
\item\label{m:1'} All terms in $h^*(P;u,w)$ have combined degree in $u$ and $w$ at most $\dim P + 1$.  

\item\label{m:2'} The mixed $h^*$-polynomial specializes to the $h^*$-polynomial i.e.
\[
h^*(P;u,1) = h^*(P;u). 
\]

\item\label{m:3'} If $\Delta_P'$ denotes a simplicial fan refinement of $\Delta_P$ then for $\Lambda$ defined in terms of the mixed $h^*$-polynomial as above, we have 
\[
\Lambda(P,\cS,\Delta_P';u,w) = (uw)^{\dim P + 1}\Lambda(P,\cS,\Delta_P';u^{-1},w^{-1}).
\]
\end{enumerate}
\end{corollary}

\subsection{Examples}

\begin{example} 
Let $\cS$ be a unimodular triangulation of a lattice polytope $P$. Then it follows from Example~\ref{e:unimodularrefined} that 
\[
h^*(P,\cS;u,v, w) =   \sum_{Q \subseteq P} w^{\dim Q + 1} l_Q(\cS|_Q;uv) g([Q,P];uvw^2).
\]
In this case, all $r$-local $h^*$-diamonds are concentrated on their central vertical strips. More generally, if 
$\cS$ is a lattice polyhedral subdivision of $P$ and there exists a unimodular triangulation $\cS'$ refining $\cS$, then 
it can be shown that, with the notation of Remark~\ref{r:generalizations},
\[
h^*(P,\cS;u,v,w) = h_P(\cS',\cS;u,v,w).
\]
\end{example}

\begin{example}\label{e:smallterms} We continue with the notation above.
We have an explicit description of some of the coefficients of  $h^*(P,\cS;u,v,w)$. Recall
that if $F$ is a (possibly empty) face of $\cS$, then $\sigma(F)$ denotes the smallest face of $P$ containing $F$.
Using Example~\ref{e:lower}, we compute that for $q,r > 0$,
\[
h^*_{0,q,r} = h^*_{q,0,r} = h^*_{r,r - q,r} =  h^*_{r- q,r,r} =  \sum_{\substack{F \in \cS \\\ \dim F = q + 1 \\\ \dim \sigma(F) = r + 1}} \# (\Int(F) \cap M),
\]
\[
h^*_{0,0,r} = h^*_{r,r,r} =    \sum_{\substack{F \in \cS \\\ \dim F \le 1 \\\ \dim \sigma(F) = r + 1}} \# (\Int(F) \cap M),
\]
\[
\dim P + 1 + h^*_{0,0,0} =   \sum_{\substack{F \in \cS \\\ \dim F \le 1 \\\ \dim \sigma(F) \le 1}} \# (\Int(F) \cap M).
\]

When $\dim P = 1$,  this gives an explicit description of $h^*(P,\cS;u,v,w)$:
\[
h^*(P,\cS;u,v,w) = 1 + uvw^2 \cdot  \# (\Int(P) \cap M).
\]
When $\dim P = 2$, this gives an explicit description of $h^*(P,\cS;u,v,w)$:
\[
h^*(P,\cS;u,v,w) = 1 + uvw^2  \big[ h^*_{0,0,0} + w\big[(1 + uv)h^*_{0,0,1} + (u + v)h^*_{0,1,1}\big].
\]
When $\dim P = 3$, we have 
\[
h^*(P,\cS;u,v,w) = 1 + uvw^2  \big[ h^*_{0,0,0} + w\big[(1 + uv)h^*_{0,0,1} + (u + v)h^*_{0,1,1}\big] 
\]
\[
+ w^2\big[ (1 + (uv)^2)h^*_{0,0,2} + (u + v)(1 + uv)h^*_{0,1,2} + (u^2 + v^2)h^*_{0,2,2} + uvh^*_{1,1,2}  \big]   \big],  
\]
where each term has an explicit description above except $h^*_{1,1,2}$. By  \eqref{i:rh2refine*} of Theorem~\ref{t:combinatorics}, 
$\sum_{p,q,r} h^*_{p,q,r} = 6\vol(P)$, and this determines $h^*_{1,1,2}$ and hence $h^*(P,\cS;u,v,w)$.
\end{example}

\section{Tropical geometry}\label{s:trop}

In this section, we give an outline of how to apply our setup to tropical geometry. We refer the reader to \cite{geometricpaper} for details. 

 Let $\K = \C(t)$ and 
let $X^\circ \subseteq (\K^*)^n$ be a closed subvariety that is \define{sch\"on} in the sense of Tevelev \cite[Definition~1.1]{TevComp}.  
The associated \define{tropical variety}  $\Trop(X^\circ) \subseteq \R^n$ may be given a rational polyhedral structure $\Sigma$. 
For any cell $\gamma \in \Sigma$, let  $\rec(\gamma)$ denote the corresponding 
recession cone i.e. $\gamma = P + \rec(\gamma)$, for some bounded polytope $P$. The collection of cones $\{ \rec(\gamma) \mid \gamma \in \Sigma \}$
forms a fan $\Delta$ called the \define{recession fan} of $\Sigma$. The corresponding dual posets are locally Eulerian with rank functions $\rho_{\Sigma^*}(\gamma) = - \dim \gamma$ and 
$\rho_{\Delta^*}(\tau) = - \dim \tau$ respectively. One can show that taking recession cones gives a strong formal subdivision: 
\[
\rec: \Sigma^* \rightarrow \Delta^*,
\]
\[
\gamma \mapsto \rec(\gamma). 
\]
Let $X_\Delta$ denote the closure of $X^\circ$ in the toric variety associated to $\Delta$ over $\K$. Then $X_\Delta$ admits a natural stratification 
$X_{\Delta} = \cup_{\tau \in \Delta} X^\circ_{\tau}$, where $X^\circ_{\{0\}} = X^\circ$. We may naturally view a variety defined over $\K$ as
a family of complex varieties over a small punctured disc $\D^*$ with co-ordinate $t$. In this case, there exists a canonical extension of 
$X_\Delta$ to a family over the whole disc $\D$, and the central fiber  $X_0 = \cup_{\gamma \in \Sigma} X^\circ_{\gamma}$ admits a natural 
stratification by explicit (and easily computed) complex varieties $X^\circ_{\gamma}$.

As in Example~\ref{e:groth}, let $R = K_0(\Var_\C)[\L^{\pm 1/2}]$ be an extension of the Grothendieck ring of complex varieties, with corresponding involution $\mathcal{D}_\C$. 
We define an element $\zeta \in I(\Sigma^*;R)$ in the incidence algebra of $\Sigma^*$ by:
\[
\zeta(\gamma) = (-1)^{\dim \gamma} \L^{\dim \gamma/2} [X^\circ_{\gamma}].
\]
We  define an element $\psi \in I(\Delta^*;R)$ in the incidence algebra of $\Delta^*$ by:
\[
\psi(\tau) = (-1)^{\dim \tau} \L^{\dim \tau/2} \psi_{X^\circ_\tau}, 
\]
where $\psi_{X^\circ_\tau} \in K_0(\Var_\C)$ is the \define{motivic nearby fiber} of $X^\circ_\tau$, as introduced by Denef and Loeser \cite{DLGeometry}.
Then our formula for the motivic nearby fiber \cite[Theorem~1.2]{geometricpaper}  is equivalent to the statement that $\zeta$ pushes forward to $\psi$ via $\rec$, i.e.
\[
\rec_*(\zeta) = \psi. 
\]

\bibliographystyle{amsplain}
\def\cprime{$'$}
\providecommand{\bysame}{\leavevmode\hbox to3em{\hrulefill}\thinspace}
\providecommand{\MR}{\relax\ifhmode\unskip\space\fi MR }
\providecommand{\MRhref}[2]{%
  \href{http://www.ams.org/mathscinet-getitem?mr=#1}{#2}
}
\providecommand{\href}[2]{#2}

\end{document}